\documentclass[10pt, reqno]{amsart}
\numberwithin{equation}{section}
\usepackage{amssymb}
\usepackage{hyperref}

\newcommand{\qtq}[1]{\quad\text{#1}\quad}

\let\Re=\undefined\DeclareMathOperator*{\Re}{Re}
\let\Im=\undefined\DeclareMathOperator*{\Im}{Im}
\DeclareMathOperator{\Id}{Id}
\newcommand{\R}{\mathbb{R}}
\newcommand{\C}{\mathbb{C}}
\newcommand{\wh}[1]{\widehat{#1}}

\newcommand{\eps}{\varepsilon}

\newcommand{\supp}{\text{supp}}

\newcommand{\cS}{\mathcal{S}}

\newcommand{\Z}{\mathbb{Z}}

\renewcommand{\H}{\mathcal{H}}

\newtheorem{theorem}{Theorem}[section]
\newtheorem*{theorem_sans_no}{Theorem}
\newtheorem{prop}[theorem]{Proposition}
\newtheorem{lemma}[theorem]{Lemma}

\newtheorem{corollary}[theorem]{Corollary}
\newtheorem{conjecture}[theorem]{Conjecture}
\newtheorem{proposition}[theorem]{Proposition}

\theoremstyle{definition}
\newtheorem{definition}[theorem]{Definition}
\newtheorem{remark}[theorem]{Remark}

\theoremstyle{remark}
\newcommand{\HS}{\mathfrak{I}_2}
\let\tr=\undefined\DeclareMathOperator*{\tr}{tr}

\DeclareMathOperator{\E}{\mathbb{E}}
\DeclareMathOperator{\PP}{\mathbb{P}}
\DeclareMathOperator{\Ai}{Ai}

\newcommand{\vk}{\varkappa}

\begin{document}

\title[Invariance of white noise]{Invariance of white noise for KdV on the line}

\author[R. Killip]{Rowan Killip}
\email{killip.r@gmail.com}
\address{Department of Mathematics, University of California, Los Angeles}

\author[J. Murphy]{Jason Murphy}
\email{jason.murphy@mst.edu}
\address{Department of Mathematics and Statistics, Missouri University of Science and Technology}

\author[M. Visan]{Monica Visan}
\email{visan.m@gmail.com}
\address{Department of Mathematics, University of California, Los Angeles}

\begin{abstract}
We consider the Korteweg--de Vries equation with white noise initial data, posed on the whole real line, and prove the almost sure existence of solutions.  Moreover, we show that the solutions obey the group property and follow a white noise law at all times, past or future.

As an offshoot of our methods, we also obtain a new proof of the existence of solutions and the invariance of white noise measure in the torus setting.
\end{abstract}

\maketitle

\section{Introduction}\label{S:introduction}

The Korteweg-de Vries equation
\begin{equation}\label{KdV}
\tfrac{d\,}{dt} q = - q''' + 6qq'
\end{equation}
takes its name from the paper \cite{KdV1895} where it is derived as a model for long waves of small amplitude in shallow water.  Since that time, it has grown to be one of the central models in mathematical physics because it sits at the nexus of numerous strands of research, both pure and applied.  It is one of the simplest models synthesizing nonlinear and dispersive effects, it is Hamiltonian, it supports solitons, and it is completely integrable.

In seeking to consider statistical ensembles of initial data for a mechanical system, one is naturally led to Gibbs measures, which model such a system in thermal equilibrium.  Such measures are constructed directly from the Hamiltonian structure and the temperature.  (More commonly, the temperature is expressed through the inverse temperature $\beta=\tfrac1{kT}$ where $T$ is the temperature and $k$ is the Boltzmann constant.)

The KdV equation admits multiple Hamiltonian descriptions; this is a common symptom of being completely integrable.  The best known of these descriptions is the following  (cf. \cite{Gardner}): The Hamiltonian
\begin{equation}\label{E:Hamil0}
H_\text{KdV}^\text{G}(q) := \int \tfrac12 q'(x)^2 + q(x)^3\,dx,
\end{equation}
generates the dynamics \eqref{KdV} via the Poisson bracket
\begin{equation}\label{Gardner}
\{ F, G \}_0 = \int  \frac{\delta F}{\delta q}(x) \biggl(\frac{\partial\ }{\partial x}\frac{\delta G}{\delta q}\biggr) (x) \,dx.
\end{equation}
(Regarding our notation for functional derivatives, see \eqref{derivative}.)   

The second description is based on the Magri--Lenard bracket (cf. \cite{Magri}):
\begin{equation}\label{Magri}
\{ F, G \}_1 = \int  \frac{\delta F}{\delta q}(x) \biggl(\biggl[ - \frac{\partial^3\ }{\partial x^3} + 2\frac{\partial\ }{\partial x} q(x) + 2q(x) \frac{\partial\ }{\partial x} \biggr]\frac{\delta G}{\delta q}\biggr) (x) \,dx,
\end{equation}
under which \eqref{KdV} is generated by
\begin{equation}\label{E:Hamil1}
H_\text{KdV}^\text{M}(q) := \int \tfrac12 q(x)^2 \,dx.
\end{equation}

Evidently, \eqref{E:Hamil0} and \eqref{E:Hamil1} both define conserved quantities for the KdV flow.  We should also mention the simplest of all the conserved quantities:
\begin{equation}\label{E:matter}
 M(q) = \int q(x) \,dx,
\end{equation}
which in the water-wave setting represents any surplus/deficit of water relative to equilibrium and thereby conservation of matter.

These three quantities, $M$, $H_\text{KdV}^\text{M}$, and $H_\text{KdV}^\text{G}$, are merely the first members of an infinite hierarchy of such conservation laws; see \cite{MR0252826}.  In fact, the full list can be reconstructed via the Lenard recursion: the flow generated by \eqref{Magri} for one entry in the list agrees with the flow generated by \eqref{Gardner} for the next entry in the list.  In particular, $M(q)$ is a Casimir for \eqref{Gardner}, but generates translations under \eqref{Magri};  $H_\text{KdV}^\text{M}$ generates translations under \eqref{Gardner}, but KdV under \eqref{Magri}; lastly,  $H_\text{KdV}^\text{G}$ generates KdV under \eqref{Gardner}, but 5th order KdV under \eqref{Magri}.

One is naturally lead to ask which Poisson structure is the `physical' one.  The answer depends on which physical system one is modeling.  In fact, Olver \cite{MR0925382} has shown that even in the case of waves in a shallow channel, the answer depends on which approach one takes to deriving KdV from the full water waves system.  Far from creating a paradox, his arguments actually present a compelling resolution of one, namely, why KdV is completely integrable.

Thus, in place of discussing which Hamiltonian structure is more physical, we should really be focusing our attention on the question of which Gibbs measure is most satisfactory.  Here the answer is more clear cut.  Because the highest order term (in powers of $q$) in $H_\text{KdV}^\text{G}$ has an indefinite sign, the corresponding Gibbs measure is unnormalizable, even in finite volume; moreover, this cannot be remedied by passing to a free energy of the form $F_\mu= H_\text{KdV}^\text{G} + \mu H_\text{KdV}^\text{M}$.

On the other hand, if one adopts $H_\text{KdV}^\text{M}$ as the Hamiltonian, then the proper meaning of the Gibbs measure is trivial: it is the Gaussian process with covariance given by the identity operator!  This process is more popularly referred to as \emph{white noise}.  The `white' property of this process is that it inhabits all frequencies to an equal degree.  (For a more rigorous definition of white noise, see subsection~\ref{S:white}; for the equipartition property, see \eqref{ONB rep}.)  We set the temperature parameter $\beta=1$ in our definition of white noise, since other temperatures can be recovered by scaling (which also respects the dynamics).

It is our belief that white noise provides the best realization of a `soliton soup' available at this time.  As we will describe shortly, a number of measures have been constructed in finite volume that compete for this moniker; however, none seem to survive in the thermodynamic limit.  On the other hand, since the measure itself captures none of the soliton's behavior, it is essential to consider the dynamics.  That is precisely the ambition of this paper: \emph{to construct dynamics for the soliton soup in the thermodynamic limit.}  More precisely, we will show the following:

\begin{theorem_sans_no}
The KdV dynamics generates a measure-preserving one-parameter group of transformations on the space of tempered distributions on $\R$ endowed with white noise measure.
\end{theorem_sans_no}

Further clarification of the meaning of this assertion will be provided as we proceed.  The impatient reader may also refer directly to Section~\ref{S:KdV}, with particular attention to Theorem~\ref{T:ktoinfinity} and Corollaries~\ref{C:intrinsic} and~\ref{C:KdV group}.

It is natural to ask if white noise measure is, in fact, supported on solitons as opposed to radiational waves.  The only way that seems reasonable to distinguish these possibilities is to look at the spectrum of the associated Lax operator, that is, the one-dimensional Schr\"odinger operator with white noise potential.  In the case of KdV with smooth localized data, we see that solitons accompany eigenvalues of the Lax operator, while dispersive wave phenomena are captured by the absolutely continuous portion of the spectrum.  For stationary (i.e. translation invariant) random potentials, the phenomenon of Anderson localization (i.e. pure point spectrum with localized eigenfunctions) has proven to be essentially universal in one spatial dimension; for the case of white noise potential, it was proved by Minami in \cite{MR1003604}. This suggests that our `soup' is indeed comprised solely of solitons, undiluted by radiation.

We turn now to a brief discussion regarding other constructions of invariant measures.  This subject has grown explosively in recent years and so we will need to curtail our discussion rather sharply.  We would like to focus our attention on two concrete scenarios: focusing equations and the thermodynamic limit.  

The dominant model in the study of invariant measures for focusing equations (essentially those supporting solitons) has been the nonlinear Schr\"odinger equation, stimulated by the pioneering work \cite{MR0939505}.  In this setting, the physical field is complex-valued and the Hamiltonian takes the form
\begin{equation}\label{E:Hamil NLS}
H_\text{NLS}(\psi) := \int \tfrac12 |\nabla \psi(x)|^2  - \tfrac{1}{p+2} |\psi(x)|^{p+2}\,dx,
\end{equation}
As with $H_\text{KdV}^\text{G}$, this Hamiltonian is not coercive and the usual Gibbs measure is unnormalizable.  Working on the torus and with $p\leq 4$, Lebowitz, Rose, and Speer constructed a modified Gibbs measure by incorporating a sharp (i.e. compactly supported) mass cutoff.  Here `mass' refers to $\int |\psi|^2$, which is invariant under the NLS flow.  The almost sure existence of solutions for such initial data was shown by Bourgain \cite{MR1309539}.  In the same paper, Bourgain also shows the invariance of the $L^2$-truncated $H_\text{KdV}^\text{G}$-Gibbs measure under the KdV flow on the torus.

Two ways of modifying the sharp mass cutoff are suggested in \cite{MR0939505} and analyzed more thoroughly  in \cite{MR3455152,MR3034397}, namely, (i) introducing a super-exponentially decaying weight in the mass and (ii) restricting to a constant mass sphere.  Once again, this is set on the one-dimensional torus.  For the two-dimensional torus, see \cite{MR1447302}.   The existence and invariance of analogous cutoff Gibbs measures in finite volume has also been actively pursued for other focusing equations,\cite{MR3900225,MR3385985,MR3869074,MR3518561,MR2574736,MR3520810,MR2532115,MR2579711,MR2928851,MR2727169,MR3346690,MR3257548,MR3281949}, with particular attention paid to Benjamin--Ono and Derivative NLS, both of which are completely integrable.

We turn now to the question of taking the infinite-volume limit of these measures.  This has been investigated by Rider \cite{MR1912096,MR2013697}, who considers the thermodynamic limit of NLS in one dimension, and by Chatterjee \cite{MR3263670}, who considers a simultaneous continuum and infinite-volume limit of discretized NLS in general dimension at fixed mass and energy.  In both cases, the limit was proven \emph{not} to exist.  More precisely, the statistical ensembles of initial data were shown to increasingly collapse upon a single solitary wave (randomly placed) and background noise of negligible magnitude.  We believe that these important works illuminate a universal truth, namely, that no such cutoff statistical ensembles for focusing equations admit a thermodynamic limit.  

By contrast, the construction of Gibbs measures in the thermodynamic limit and of corresponding almost-sure dynamics has been very successful for \emph{defocusing} Hamiltonian PDE.  We note, in particular, the work on nonlinear wave equations in one dimension by McKean and Vaninsky \cite{McKean,MR1277197}, by Xu \cite{XuPreprint} in three dimensions with spherical symmetry, and the work by Bourgain \cite{BourgainCMP} on NLS in one dimension.  Of these, the work of Bourgain is most pertinent to the problem discussed herein, because unlike NLS and KdV, the wave equation enjoys finite speed of propagation.

As is natural for any process describing thermal equilibrium, white noise is invariant under translation (more formally, it is a stationary process).  In fact, the process is not only ergodic under translation but (strong) mixing.  (This is easily deduced from independence on disjoint intervals, cf. Remark~\ref{R:cov}.)

As we will briefly discuss, there has been a recent surge in activity regarding KdV with initial data sampled from almost periodic processes \cite{MR3859361,MR3486173,MR3894927,Kotani}.  Recall that \emph{almost periodicity} means that translates of the initial data form a precompact set in $L^\infty(\R)$.  This should be regarded as antithetical to mixing; indeed, representative examples are periodic and quasi-periodic functions.  A major goal of these investigations of almost periodic initial data is to resolve a conjecture of Deift \cite{MR2411922,MR3622647}, namely, that such initial data leads to almost periodic flows in time.  By comparison, white noise is mixing under translations and it is natural to believe that it is also mixing under the KdV flow.  By constructing almost-sure solutions in this paper, we are able to make this assertion precise; see Conjecture~\ref{mixing conjecture}.  This seems a very challenging problem.  For a typical linear flow, one can prove that the evolution of white noise is mixing directly from the Fourier transform.  The only nonlinear result we know of in this direction is that of McKean \cite{McKean} regarding the $\sinh$-Gordon equation on the line with Gibbs-distributed initial data, which relies heavily on both the complete integrability and the causal structure/Lorentz invariance of that model.

The state of affairs laid out above explains our great enthusiasm for studying the KdV equation with white noise initial data. Indeed, we are not the first to be so intrigued.  In particular, invariance of white noise for KdV on the torus has been resolved (in several different ways) in a series of papers by Oh, Quastel, and Valk\'o \cite{Oh,Oh2,OhQuastelValko,QuastelValko}.  The existence of global dynamics for KdV on the torus with white noise initial data predates these works, since Kappeler and Topalov \cite{MR2267286} show that KdV is globally well-posed in $H^{-1}(\R/\Z)$.  This result is used crucially in the approach of \cite{QuastelValko}, but not in the other papers.  We shall present one further solution of this problem in Section~\ref{S:commute} as an offshoot of the main thrust of our argument.

Incidentally, while it is now known \cite{KV} that KdV is also globally wellposed in $H^{-1}(\R)$, this is of little direct relevance to the question of white noise initial data on the line, which has no decay at infinity.  Indeed, periodic data is a better proxy for white noise, since both are ergodic under translations; no law on $H^{-1}(\R)$ has this property.  This ergodicity also forbids any local smoothing effect, which is one of the key tools usually employed in the study of dispersive equations in infinite volume.

While it is tempting to try to take the infinite-volume limit of the torus situation (which is completely settled), this seems absolutely hopeless at this time.  We have no control on the transportation of norm from one spatial location to another.  Moreover, the central trick of exploiting Fubini to convert statistical conservation into space-time bounds is useless if one cannot make sense of the nonlinearity pointwise in time, as is the case here --- one simply cannot square white noise (with or without Wick ordering).

For the state of the art in this approach, see \cite{BourgainCMP}, in which Bourgain treats Gibbs measures for one-dimensional defocusing NLS (with nonlinearity not exceeding cubic) by taking the infinite-volume limit of torus dynamics.  Note that sample paths from the Gibbs measure considered in \cite{BourgainCMP} are essentially bounded; indeed, the defocusing nonlinearity confines the paths no less tightly than for the Ornstein--Uhlenbeck process, for which one already has at most logarithmic growth. 

We should also note that in treating such infinite-volume problems, taking a limit of some cutoff model is one of the few methods that have any chance of success.  Direct local wellposedness arguments are hopeless: if the length of the time interval (on which the solution is constructed) needs to depend on any facet of the initial data, then by ergodicity, it  must be zero.  Or, to put it more colloquially, for an ergodic process, anything bad that can happen, will happen (and with positive density in space). 

By this reasoning, the first step in our analysis must be to find the right cutoff model.  It is in this regard (and almost no other) that what we do here builds fundamentally on the work \cite{KV}.  We will use (a modification of) the commuting flows we introduced in \cite{KV} as our cutoff model.  The fundamental obstruction to be overcome in \cite{KV} is very low regularity, common to both line and circle, and (by scaling) it suffices to consider initial data of small norm.   On the other hand, white noise is rather more regular (it belongs to $H^{-1/2-}_\text{loc}$), but has no decay and cannot be considered small.  (For any threshold and norm, there is \emph{somewhere} in space for which the norm exceeds that threshold.)  In light of these distinctions, we view the problem of white noise on the whole line as an independent test of the methodology based on commuting flows.

The commuting flows employed in \cite{KV} were defined through the Gardner bracket \eqref{Gardner} and their Hamiltonians:
\begin{align*}
\H_\kappa(q) := \int - 16 \kappa^5 \rho(x;\kappa,q) + 2 \kappa^2 q(x)^2\,dx,
\end{align*}
where the parameter $\kappa>0$ and the integral is taken over one period in the periodic case or over the whole real line.  Here,
\begin{align*}
\rho(x;q,\kappa) := \kappa - \tfrac{1}{2g(x;\kappa)} + \tfrac12\int_{\R} e^{-2\kappa|x-y|} q(y)\,dy,
\end{align*}
where $g(x;\kappa)$ denotes the diagonal Green's function at energy $-\kappa^2$:
\begin{align*}
g(x;q,\kappa) := \langle \delta_x, (-\partial_x^2 + q + \kappa^2)^{-1} \delta_x\rangle.
\end{align*}

It is not difficult to see that the Hamiltonians are well defined for small data in $L^2$ and that the resulting flows are well defined for small data in $H^{-1}$; see \cite{KV} for details.  The smallness assumption here depends on $\kappa$ and is necessary; indeed, for large potentials (of indefinite sign) the spectrum of the Schr\"odinger operator can collide with $-\kappa^2$ and it will not be possible to define $\H_\kappa(q)$ nor the resulting flow.  In the case of white noise, for example,  it is relatively easy to show that the spectrum of the whole-line Schr\"odinger operator is the entire real axis!

Our remedy here is not terribly surprising: We move $\kappa$ off the real axis, calling the new parameter $k$; see \eqref{complex H_k} for the precise Hamiltonian and \eqref{H_k flow q} for the resulting dynamics.  On the other hand, the underlying size problem (that forced us to change the flows) also breaks all the technology developed in \cite{KV}, which was based on a perturbative analysis about $q\equiv 0$.  Setting aside for the moment how this is to be resolved, let us continue with the overarching plan of this paper. 

We begin with the $\H_k$ flow on the torus $\R/L\Z$ of circumference $L$.  The well-posedness of the flow is elementary since the nonlinearity is Lipschitz on $H^{-1}(\R/L\Z)$; see Proposition~\ref{P:H kappa}.  We show that this flow preserves (periodized) white noise measure in Theorem~\ref{T:Hktorus}.  While this could also be proved by the (now standard) method of finite-dimensional approximation, it is more efficient to prove it directly using integration by parts in Gauss space. 

The last result of Section~\ref{S:commute} is a new proof of the invariance of white noise under the KdV flow on the torus; see Theorem~\ref{T:KdV-torus}.  The argument is rather elementary; we simply exploit the fact that the $\H_k$ flows conserve white noise and converge to the KdV flow as $k\to\infty$.

The key strategic decision in treating white noise initial data on the whole line is to first send the volume to infinity (at fixed $k$) and then to send $k\to\infty$ to recover the KdV flow.   The virtue of this choice is that the $\H_k$ flow effectively has finite speed of propagation; the effective speed limit is of order $\Re k^2$. This allows us to minimize the effect of the discrepancy between the two initial data, namely, periodic white noise and infinite-volume white noise.  To prove almost sure convergence, it is necessary to couple the finite-volume and infinite-volume initial data together in an appropriate way.  We do this by starting with a sample of white-noise in infinite volume, truncating it to the interval $[-L,L]$, and then extending it periodically thereafter.

A concrete manifestation of finite speed of propagation for the $\H_k$ flows can be found, for example, in the proof of Lemma~\ref{T:L0-to-infinity}.  There we see that the exponentially weighted $H^{-1}$ norm grows at most exponentially in time with a rate that is bounded by $|k|^2$.  We have no such control on the transport of local $H^{-1}$ norm in the KdV setting.  Indeed, as shown in \cite[\S7]{KV}, the transport of such a quantity is mediated by the local $L^2_{t,x}$ norm of the solution, which is guaranteed to be infinite in the white noise setting.

To prove convergence, we first need estimates.  Indeed, here lies the fundamental challenge in sending $L\to\infty$ for the $\H_k$ flow. The price to pay for limiting the propagation speed and for improving the regularity of the nonlinearity is that the evolution equation has become nonlocal. This is apparent in \eqref{H_k flow q} through the appearance of the diagonal Green's function.  In this way, we see that the spectre of far-away, low-probability, bad events ruining everything remains; nevertheless, we claim it has been ameliorated (and this is why we will ultimately succeed).
 
Sections~\ref{S:singlescale} and~\ref{S:multiscale} are devoted to obtaining the bounds we need for the diagonal Green's function, both to send $L\to\infty$ at fixed $k$ in Section~\ref{S:L} and then to send $k\to\infty$ in Section~\ref{S:KdV}.  The key phenomenon to leverage in controlling bad events is the exponential decay of the Green's function (away from the diagonal) --- this is the backbone of Fr\"ohlich--Spencer \cite{MR0696803} multiscale analysis, which we will adapt to our needs.  Various forms of multiscale analysis have grown to be central tools in the study of the Anderson model; see \cite{MR2509111} for a survey.  While is true that the Schr\"odinger operator with white noise potential is a natural form of Anderson model, the needs of our dispersive analysis are rather different from those of proving Anderson localization.  First, we work with energies at a positive distance from the real axis; thus, we do not need the potential to generate the exponential decay, we only have to prevent the potential from destroying the decay already present in the free resolvent.  This distinction manifests, for example, in the development of our initial length scale estimates Propositions~\ref{P:4} and~\ref{P:alpha}.

On the other hand, for our purposes, the mapping properties of the resolvent (particularly $H^{-1}\to H^{+1}$) are far more important than the mere decay of the Green's function.  This is vital, for example for controlling the dependence of the resolvent upon the potential, due to the extreme irregularity of our potentials.  The need for such operator bounds is also a reason to prefer multiscale analysis of the resolvents over transfer matrix techniques (commonly used in 1D models), which work at the level of the Green's function.  The desire for clean resolvent bounds also influences our decision to eschew the traditional technique of Dirichlet decoupling in our treatment of the multiscale analysis.

A second (and more decisive) reason for wishing to avoid Dirichlet decoupling is the need (in Section~\ref{S:KdV}) for \emph{lower} bounds on the diagonal Green's function.  Evidently, a Dirichlet boundary condition will drive $g(x)$ to zero.  In the proof of Proposition~\ref{P:L56}, we see that even obtaining pointwise lower bounds on the diagonal Green's function in expectation requires significant control on the resolvents (including Lemma~\ref{L:18}), as well as considerable additional gymnastics. 

Our approach to multiscale analysis is to work consistently on the whole real line and to successively reveal the potential $q$ in dyadic windows centered on the origin.  This leads to the following expansion of resolvents:
\begin{equation}\label{E:intro MSA}
R_{L} = R_1 - \sum_{\ell=1}^{L/2} R_{2\ell}(q_{2\ell}-q_\ell)R_\ell \qtq{where} q_\ell(x)=q(x)\chi_{[-\ell,\ell]}(x)
\end{equation}
and $R_\ell$ denotes the resolvent with potential $q_\ell$.   We then obtain our infinite-volume bounds (Proposition~\ref{P:8} and~\ref{P:9'}) by combining this expansion (both at finite and infinite $L$) with the initial length scale estimates of Section~\ref{S:singlescale}.  Although the estimates of Section~\ref{S:singlescale} demonstrate good exponential decay, they contain prefactors depending (weakly but unfavorably) on $L$.  Removing these prefactors is essential for taking the $L\to\infty$ limit and this is what multiscale analysis achieves.   

There is one further wrinkle in this story that we have omitted, namely, that in order to prove the convergence of the $\H_k$ flows, we need bounds not only in infinite volume, but also on finite tori, but with bounds independent of the circumference.  To do this, we need a two-parameter version of the argument outlined above, where $L$ denotes the length of potential revealed near the origin (as in \eqref{E:intro MSA}) and then this section of potential is repeated in each period cell, which has length $2L_0$; see \eqref{qL} for the precise formula.   

Let us now turn our attention to Section~\ref{S:L}, which considers the $\H_k$ flow on the whole line.  The main result here is Theorem~\ref{T:invariance-Hk-line}, which summarizes the results of Propositions~\ref{P:666} and~\ref{P:uniq}.  The former proposition constructs solutions of the $\H_k$ flow obeying certain additional bounds (that hold almost surely); the latter shows that solutions obeying such bounds are unique.  One important consequence of the uniqueness statement is that it allows us to ensure that almost surely the full trajectory avoids the null set of initial data on which we do not construct solutions.  (This would be trivial for a discrete dynamical system, but our set of times is uncountable.)  In particular, it is meaningful to ask if our flows have the group property.  They do!  See Corollary~\ref{C:semigroup} for details. 

There are two main components to our construction of solutions to the $\H_k$ flows in infinite volume as limits of the corresponding flows on ever-larger tori.  The first is to obtain bounds on the finite-volume solutions that are uniform on compact time intervals.  This is the content of Lemmas~\ref{T:L0-to-infinity} and~\ref{L:66}, which rely crucially on the effective speed limit discussed above.  The second component is to demonstrate convergence of the dynamics by controlling the time derivative \eqref{H_k flow q}, which ultimately means showing convergence of the diagonal Green's functions.  The first step in this direction is Lemma~\ref{L:6.4?}; however, this only provides information pointwise in time.  The reason is that the `Lipschitz constant' of the mapping $q\mapsto g$ involves the norms of resolvents, which (since we need volume-independent bounds) can only be provided through the multiscale analysis, which in turn yields bounds in expectation only.  The remedy is to prove equicontinuity in time, for which we employ methods closely connected with the Kolmogorov continuity theorem; see Lemmas~\ref{L:KCT} and~\ref{L:KCT'} for the underlying idea and Lemma~\ref{L:gL dif} for the specific implementation.

In Section~\ref{S:KdV}, we show solutions to the $\H_k$ flows on the whole line with white noise initial data converge uniformly on compact time intervals (almost surely) as $k\to\infty$; see Theorem~\ref{T:ktoinfinity}.  In view of existing deterministic results, it is natural to simply declare that these limits are the solutions to KdV with white noise initial data.  However, in Corollary~\ref{C:intrinsic} we go one step further and verify that the limit $q(t)$ and its associated diagonal Green's function $g(t)$ obey the same integral equation as in the deterministic case, namely, \eqref{green solution}.  Thus $q(t)$ is a solution of KdV in the only intrinsic sense that we currently know makes sense.  One further virtue of our flows is that they obey the group property; see Corollary~\ref{C:KdV group}.  This is a further indication of uniqueness --- the evolution of a state is independent of how (or when) it appears in the limiting process.  

Needless to say, the difficulty in making direct sense of the KdV flow for highly irregular data also makes it difficult to show convergence of the $\H_k$ flows.  We adopt here an idea from \cite{KV}, namely, to use the diagonal Green's function.  A key difference between $q(t)$ and $g(t)$ is that all the terms in the differential equation obeyed by $1/g(t)$ under the KdV flow make sense as tempered distributions pointwise in time.  By employing all the estimates proved in the preceding sections, we are able to show that the diagonal Green's functions associated to the $\H_k$ flows converge as $k\to\infty$.  Initially (see Proposition~\ref{P:12main}), this convergence is only in expectation and in a weighed Sobolev space with $H^{-2}$ regularity.  These defects are removed by demonstrating equicontinuity in time (in the manner used in Section~\ref{S:L}) and by employing the extra spatial regularity proved in Proposition~\ref{P:L56}.  In this way, we are able to upgrade the convergence to a weighted $H^1$ norm in space, uniformly on compact time intervals.

In the setting of \cite{KV}, the convergence of the Green's functions guaranteed the convergence of the solutions $q(t)$ themselves by virtue of the diffeomorphism property; see \cite[Proposition~2.2]{KV}.  That result is proved by the inverse function theorem and as such, is confined to small initial data.  Thus, we need here a new way of connecting the diagonal Green's function back to the potential.  As $k$ is now complex, we no longer can expect a diffeomorphism ($g$ is complex-valued and $q$ real-valued); at best one may hope for an embedding.  This is what we prove in Lemma~\ref{L:diffeo}, albeit in the limited setting of $q\in H^{-1}$.  The key new relation is \eqref{E:q from g}; indeed, it is this relation that we show is retained as we send $L\to\infty$ and then $k\to\infty$.  And it is this relation that ultimately allows us to prove Theorem~\ref{T:ktoinfinity} by deducing convergence of the solutions $q(t)$ (uniformly on compact time intervals in weighted $H^{-1}(\R)$) from the behavior of their diagonal Green's functions.

The remainder of Section~\ref{S:KdV} is devoted to proving the auxiliary properties of our solutions outlined above and to describing two interesting (but challenging) directions for further investigation.  First, we highlight open questions related to uniqueness for KdV at low regularity. This is open both in the deterministic and random settings.  In particular, it is currently unknown whether Kappeler--Topalov \cite{MR2267286} or Killip--Visan \cite{KV} solutions are unique in some intrinsic sense, independent of their appearance as the unique limit of Schwartz-class solutions.  The second major question we raise is whether the KdV flow with white noise initial data is mixing (in time).  We refer here specifically to the whole-line case. This clearly fails on the torus; known conservation laws show that the torus flow is not even ergodic.

\subsection*{Acknowledgements} R. K. was supported by NSF grant DMS-1600942.  J.~M. was supported by a Simons Collaboration grant.  M. V. was supported by NSF grants DMS-1500707 and DMS-1763074.

The authors are also grateful to Changxing Miao for hosting us at the Institute of Applied Physics and Computational Mathematics in Beijing, where part of this work was completed.

\section{Notation and preliminaries}\label{S:notation}

We write $A\lesssim B$ or $B\gtrsim A$ to denote $A\leq CB$ for some $C>0$.  Dependence on various parameters will be indicated by subscripts.  If $A\lesssim B$ and $B\lesssim A$, we write $A\approx B$.  The other common notation for this relation, namely $\sim$, is used in this paper to indicate the law of a random variable.  We employ the Japanese bracket notation $\langle x\rangle=\sqrt{1+x^2}$.  We write $\mathbb{N}=\{1,2,3,\dots\}$ and $\mathbb{N}_0=\mathbb{N}\cup\{0\}.$ 

Throughout, primes will represent derivatives with respect to the spatial variable~$x$. 

Our convention for the Fourier transform is 
\[
\hat f(\xi)=\tfrac{1}{\sqrt{2\pi}}\int_\R e^{-ix\xi}f(x)\,dx,\qtq{so that} f(x)=\tfrac{1}{\sqrt{2\pi}}\int_\R e^{ix\xi}\hat f(\xi)\,d\xi.
\]

For $s\in\R$ and $\kappa>0$, we define 
$$
\|f\|_{H^s(\R)}^2 = \int |\hat f(\xi)|^2(4+\xi^2)^s\,d\xi \quad\text{and} \quad \|f\|_{H^s_\kappa(\R)}^2 = \int |\hat f(\xi)|^2(4\kappa^2+\xi^2)^s\,d\xi. 
$$
We have the simple but useful estimate
\begin{equation}\label{829-gain}
\|1\|_{H_{\kappa}^1\to H_{\kappa}^{-1}} \lesssim \tfrac{1}{\kappa^2}. 
\end{equation}

For $p\geq 1$, we write $\mathfrak{I}_p$ for the Schatten class of compact operators whose singular values lie in $\ell^p$.  We will primarily use the Hilbert--Schmidt class $\mathfrak{I}_2$.  Recall that an operator $A$ on $L^2(\R)$ is Hilbert--Schmidt if and only if it admits an integral kernel $a(x,y)\in L^2(\R\times\R)$; moreover,
\begin{equation}\label{I2vsL2}
\|A\|_{L^2\to L^2}^2 \leq \|A\|_{\mathfrak{I}_2}^2 = \tr\{A^*A\} = \iint |a(x,y)|^2\,dx\,dy. 
\end{equation}
The product of two Hilbert--Schmidt operators $A$ and $B$ is in the trace class $\mathfrak{I}_1$; moreover,
\[
\tr\{AB\} = \iint a(x,y)b(y,x)\,dy\,dx = \tr\{BA\}\qtq{and} |\tr\{AB\}| \leq \|A\|_{\mathfrak{I}_2}\|B\|_{\mathfrak{I}_2}. 
\]
Hilbert--Schmidt operators form a two-sided ideal in the algebra of bounded operators; concretely,
\begin{equation}\label{ideal}
\|BAC\|_{\mathfrak{I}_2} \leq \|B\|_{L^2\to L^2}\|A\|_{\mathfrak{I}_2}\|C\|_{L^2\to L^2}. 
\end{equation}
We refer the reader to \cite{MR2154153} for more information.

Given a function $m$, the corresponding multiplication operator satisfies \begin{equation}\label{multiplier1}
\|m\|_{H_{\kappa}^1\to H_{\kappa}^1}=\|m\|_{H_{\kappa}^{-1}\to H_{\kappa}^{-1}} \lesssim \|m\|_{L^\infty}+\|m'\|_{L^\infty}.
\end{equation}

We will employ the $L^2$ pairing throughout the paper.  This also informs our notation for functional derivatives:
\begin{equation}\label{derivative}
\frac{d\ }{ds}\biggr|_{s=0} F(q+sf) = dF\bigl|_q (f) = \int \frac{\delta F}{\delta q}(x) f(x)\,dx .
\end{equation}

\subsection{Probabilistic preliminaries} 
Throughout the paper we fix a probability space with probability measure $\PP$.  Expectation (i.e. integration with respect to $d\PP$) is denoted by $\E$.  

The hypercontractive property of polynomials in multivariate Gaussian random variables is well understood thanks to the definitive work of Nelson.  We need here only the most basic manifestation of this phenomenon, whose proof we include for completeness.

\begin{lemma}\label{L3} Let $Q$ be a positive semidefinite quadratic form in jointly Gaussian random variables. Then
\begin{align}\label{equad}
\E\bigl\{ \exp\bigl[\tfrac{\theta Q}{\E\{Q\}}\bigr]\bigr\} \leq (1-2\theta)^{-\frac12} \qtq{for any}0\leq\theta<\tfrac12. 
\end{align}
Consequently, for $1\leq p<\infty$ we have
\begin{align}\label{pquad}
\E\bigl\{ |Q|^p\bigr\} \lesssim (2p\E\{Q\})^p.
\end{align}
\end{lemma}

\begin{proof}  Let us write $Q=\vec X^T A\vec X$, where $\vec X\sim N(0,\Sigma)$ and $A=A^T\neq 0$ is positive semidefinite.  We choose $B$ so that $\Sigma=BB^T$ and let $O$ be an orthogonal matrix diagonalizing $B^T AB$. We write $\lambda_k\geq 0$ for the eigenvalues (repeated with multiplicity) of $B^TAB$. 

If $\vec Z\sim N(0,Id)$, then (computing the characteristic function, say) one finds that $BO\vec Z \sim N(0,\Sigma)$ (the law of $\vec X$).  Moreover,
\[
(BO\vec Z)^T A(BO\vec Z)=\vec Z^T O^T B^T A B O \vec Z = \sum_k \lambda_k Z_k^2. 
\]
Thus the left-hand side of the stated inequality can be estimated as
\begin{align*}
\text{LHS\eqref{equad}}& =  \E\biggl\{\prod_k e^{\theta c_k Z_k^2}\biggr\}\qtq{where} c_k=\frac{\lambda_k}{\sum \lambda_l},\qtq{and so}\sum_k c_k =1 \\
& \leq \prod_k \bigl[\E\bigl\{ e^{\theta Z_k^2}\bigr\}\bigr]^{c_k}\quad\text{by H\"older's inequality} \\
& = \prod_k \bigl\{(1-2\theta)^{-\frac12}\bigr\}^{c_k} = (1-2\theta)^{-\frac12}. 
\end{align*}
Here we have used $\E(e^{\theta Z_k^2})=(1-2\theta)^{-\frac12}$ for $Z_k\sim N(0,1)$, which follows from computing the Gaussian integral. 

From \eqref{equad}, Tchebychev's inequality yields
\[
\PP\bigl\{ |Q|>\lambda\bigr\} \leq (1-2\theta)^{-\frac12} \exp\bigl\{-\tfrac{\theta \lambda}{\E(Q)}\bigr\}\qtq{for any}\theta\in[0,\tfrac12)\qtq{and}\lambda>0.
\]
Choosing $\theta=\frac14$, we then deduce \eqref{pquad} in the usual way.\end{proof}

\begin{lemma}\label{L:sup} Let $\{X_n\}_{n=1}^N$ be a collection of random variables satisfying
\[
\sup_{n} \PP\{|X_n|>\lambda\} \leq C e^{-c\lambda}
\]
for some $c,C>0$. Then
\begin{align}\label{1051}
\PP\bigl\{\sup_{n}|X_n|>\lambda\bigr\} \leq C e^{-c\lambda/2} \quad \text{for all $\lambda\geq 2c^{-1}\log{N}$}.
\end{align}

\end{lemma}

\begin{proof} For $\lambda$ as given, $\text{LHS}\eqref{1051}\leq \sum_n \PP\{|X_n|>\lambda\} \leq NCe^{-c\lambda}\leq \text{RHS}\eqref{1051}$.
\end{proof}

Our next lemma follows from what is shown in any elementary probability text under the rubric of uniform integrability.  However, as there is no catchy name for this result, we write it out in full here:

\begin{lemma}\label{L:same conv}
If $\sup_n \E |X_n|^p <\infty$ for every $p<\infty$ and $X_n\to X$ in probability, then $X_n\to X$ also in $L^p(d\PP)$ for each $p<\infty$.
\end{lemma}

The next lemma encapsulates the idea of the Kolmogorov Continuity Theorem (cf. \cite[\S2.1]{StroockVaradhan}) in a form that will be useful to us.

\begin{lemma}\label{L:KCT}
Given $T,\alpha, \eps>0$, $1\leq p<\infty$, a Banach space $X$, and a process $F:[-T,T]\to X$ that is almost surely continuous, 
\begin{align*}
\E\Bigl\{ \| F \|_{C_t^\alpha X}^p \Bigr\} \lesssim_{p,\eps,T} \E\Bigl\{ \| F(0) \|_{X}^p \Bigr\}
	+ \sup_{-T\leq s<t\leq T} \E\biggl\{ \frac{\| F(t) -F(s) \|_{X}^p}{|t-s|^{1+\alpha p + \eps}} \biggr\}.
\end{align*}
\end{lemma}

\begin{proof}
We provide details for $T=1$; this can then be iterated to yield the result for larger $T$.  If $F$ is continuous and $2^{-n}\leq |t-s|<2^{1-n}$ then
\begin{align*}
\| F(t) - F(s)  \|_{X} &\leq  2 \sum_{m=n}^\infty \sup_{|\ell|\leq 2^{m}} \| F(\ell 2^{-m}) - F([\ell+1] 2^{-m})\|_X
\end{align*}
and consequently,
\begin{align*}
\sup_{s<t} \frac{\| F(t) - F(s)  \|_{X}}{|t-s|^\alpha} &\leq  2 \sum_{m=0}^\infty 2^{\alpha m} \sup_{|\ell|\leq 2^{m}} \| F(\ell 2^{-m}) - F([\ell+1] 2^{-m})\|_X .
\end{align*}
Thus, (using H\"older's inequality on the sum in $m$),
\begin{align*}
\| F \|_{C_t^\alpha X}^p &\lesssim_{p, \eps}   \| F(0) \|_{X}^p
	+ \sum_{m=0}^\infty \;2^{(\alpha p +\frac{\eps}2)m} \! \sum_{|\ell|\leq 2^{m}} \| F(\ell 2^{-m}) - F([\ell+1] 2^{-m})\|_X^p .
\end{align*}
The result now follows by taking expectations and exploiting its linearity.
\end{proof}

While the estimates we prove in Section~7 do yield H\"older continuity in time of solutions to KdV (in some weighted Sobolev spaces) via this lemma, our principal reason for including it here is to provide the requisite equicontinuity to upgrade convergence in probability at each time to $L^p(d\PP;C_t X)$ convergence.  We may encapsulate what we need in the following way:

\begin{lemma}\label{L:KCT'}
Given $T, \alpha>0$, a Banach space $X$, and a sequence of processes $F_n:[-T,T]\to X$ satisfying
\begin{align*}
\sup_n \,\E\Bigl\{ \| F_n \|_{C_t^\alpha X}^p \Bigr\} <\infty\qtq{and}
\lim_{n\to\infty} \sup_{|t|\leq T} \PP\Bigl\{ \sup_{m>n} \| F_n(t)-F_m(t)\|_X > \eps \Bigr\} =0
\end{align*}
for all $1\leq p < \infty$ and all $\eps>0$, there is a limit process $F\in C_t X$ with
\begin{align*}
\lim_{n\to\infty} \E\Bigl\{ \| F_n - F \|_{C_t X}^p \Bigr\} = 0, \quad\text{for all $1\leq p<\infty$}. 
\end{align*}
\end{lemma}

\begin{remark}
In view of the $L^p_\omega C_t^\alpha X$ bound, one may trivially upgrade the convergence to $L^p_\omega C_t^\beta X$ for any $0\leq\beta<\alpha$.  On the other hand, convergence in $C^\alpha$ may fail even for deterministic sequences.  
\end{remark}

\begin{proof}
For $t$ fixed, we have that $F_n(t)$ are bounded in $L^p$ for any $1\leq p<\infty$ and converge in probability (because they are Cauchy in this sense).  Thus, by Lemma~\ref{L:same conv},  these random variables converge in $L^p$ sense for any $1\leq p<\infty$ and
$$
\lim_{n\to\infty} \,\sup_{m>n} \E\Bigl\{ \| F_n(t) - F_m(t) \|_{X}^p \Bigr\} =0.
$$

On the other hand, given natural numbers $n,m,N$, 
\begin{align*}
\| F_n - F_m \|_{C_t X} 
&\leq \sum_{|\ell|\leq N} \| F_n(\ell T/N) - F_m(\ell T/N) \|_{X} + N^{-\alpha} \bigl[\| F_n \|_{C_t^\alpha X} + \| F_m\|_{C_t^\alpha X} \bigr].
\end{align*}
Thus, taking $N$ large and then $n$ large (depending on $N$), we obtain
\begin{align*}
\lim_{n\to\infty} \sup_{m>n} \E\Bigl\{ \| F_n - F_m \|_{C_t X}^p \Bigr\} =0.
\end{align*}
This shows that $F_n$ is Cauchy in $L^p_\omega C_t X$ and so convergent there.
\end{proof}

\subsection{White noise}\label{S:white}

For concreteness, we record here some well-known basic properties of white noise.  Note that we consider here only real-valued white noise.

\begin{theorem}
A random variable $q$ taking values in $\mathcal S'(\R)$ is said to be white noise distributed if
\begin{equation}\label{Minlos defn}
\E \bigl\{ e^{i \langle f, q\rangle } \bigr\} = \exp\bigl\{ - \tfrac12 \| f
\|_{L^2}^2 \bigr\}
\end{equation}
for all real-valued Schwartz functions $f$.  Moreover,  given any
orthonormal basis $\{\psi_n\}$ of $L^2(\R)$ and $\{X_n\}$ i.i.d.
$N(0,1)$ random variables,
\begin{equation}\label{ONB rep}
\sum _n X_n \psi_n
\end{equation}
converges almost surely in $\mathcal S'(\R)$ and follows the white noise law.
\end{theorem}

\begin{remark}\label{R:cov}
In view of \eqref{Minlos defn}, we may say that white noise is the Gaussian process with covariance
\[
\E\{\langle \phi,q\rangle\langle \psi,q\rangle\} = \langle \phi,\psi\rangle,\qtq{or colloquially} \E\{q(x)q(y)\}= \delta(x-y).
\]
This shows that $\langle \phi, q\rangle$ and $\langle \psi, q\rangle$ are independent whenever $\phi$ and $\psi$ are orthogonal in $L^2$.  In particular, if $I$ and $J$ are disjoint intervals, then $q|_I$ and $q|_J$ are independent.  This independence property immediately yields the ergodicity (indeed, mixing property) of white noise under spatial translations.
\end{remark}

\begin{proof}[Outline of Proof]
The existence and uniqueness of a cylinder measure on $\mathcal
S'(\R)$ obeying \eqref{Minlos defn} follows from the Minlos Theorem;
see, for example,  \cite[Theorem~I.2.3]{Simon Functional
Integration}.

We turn now to \eqref{ONB rep}.  For any finite set $N$,
\begin{align}\label{1058}
\E\Bigl\{ \bigl\| \langle x\rangle^{-1}  \sum_{n\in N} X_n \psi_n
\bigr\|_{H^{-1}(\R)}^2 \Bigr\}
&= \sum_{n\in N} \bigl\langle\psi_n,   \langle x\rangle^{-1} (-\Delta+4)^{-1}  \langle x\rangle^{-1} \psi_n\bigr\rangle \notag\\
&\leq  \|    \langle x\rangle^{-1} (-\Delta+4)^{-1}  \langle
x\rangle^{-1}  \|_{\mathfrak I_1} \lesssim 1.
\end{align}
Thus $L^2(d\PP)$-convergence in this weighted Sobolev space is guaranteed.  Almost-sure convergence in this weighted Sobolev space (and so also in $\mathcal S'(\R)$) can then be deduced from the maximal inequality for martingales in Hilbert space (see \cite[Theorem~5.3.27]{Stroock}).  The fact that the limit satisfies \eqref{Minlos defn} is elementary.
\end{proof}

A more precise version of the calculation \eqref{1058} is the following: For $w\in L^2(\R)$,
\begin{equation}\label{weighted-wn}
\E\bigl\{ \| wq\|_{H^{-1}_\kappa}^2\bigr\} =  \tr\bigl\{ w (-\Delta+4\kappa^2)^{-1} w\bigr\} = \tfrac1{4\kappa} \int w^2 \,dx.
\end{equation} 

The construction \eqref{ONB rep} works equally well on any torus, say $\R/2L\Z$.  Equivalently, one may begin with white noise on $\R$, restrict to the interval $[-L,L]$ and then construct $q_L$ by periodizing (as in \eqref{qL} with $L=L_0$). The analogue of Remark~\ref{R:cov} in this setting is
\begin{align}\label{qL cov}
\E\{q_L(x)q_L(y)\}= \sum_{n\in\Z} \delta(x-y-2nL) .
\end{align}
Arguing as in \eqref{weighted-wn}, one readily sees that
\[
\E\bigl\{ \|q_L\|_{H^{-1}(\R/2L\Z)}^2\bigr\} \lesssim L \qtq{and thence} \E\bigl\{ \|q_L\|_{H^{-1}(\R/2L\Z)}^{2p}\bigr\} \lesssim_p L^p 
\]
for any $1\leq p<\infty$, by using Lemma~\ref{L3}.

The following result represents integration by parts in Gauss space; see \cite{Malliavin} for further details.

\begin{lemma}\label{T:Wick}  Let $q$ be white noise and suppose $F$ belongs to $D^2_1$, which is to say
\begin{align}\label{247}
\| F \|_{D^2_1} ^2 := \E \Bigl\{ \bigl\| \tfrac{\delta F}{\delta q} (q) \bigr\|_{L^2}^2 +\bigl|F(q)\bigr|^2 \Bigr\} < \infty.
\end{align}
Then for any $\varphi\in L^2$ we have
$$
\E\bigl\{\langle q, \varphi\rangle F(q)\bigr\} = \E\bigl\{\langle \tfrac{\delta F}{\delta q}, \varphi\rangle\bigr\}.
$$
\end{lemma}

\subsection{Resolvents and the diagonal Green's function}\label{S:resolvent} 

For $k\in\C$ with $\Re k>0$, the free resolvent 
\[
R_0(k)=(-\partial^2 + k^2)^{-1}
\]
has integral kernel
\[
G_0(x,y;k)= \tfrac{1}{2k}e^{-k|x-y|}.
\]

\begin{definition}\label{D:admissible} We call nonzero $k\in\C$ \emph{admissible} if $0\leq \arg k < \tfrac{\pi}{4}$.  For any admissible $k$, we set
\begin{align}\label{b=a}
\kappa := |k|, \quad E:=\Re(k^2) \qtq{and} \sigma:=\Im(k^2).
\end{align}
For $\kappa_0>0$, we define 
\[
A(\kappa_0)=\{k\text{ admissible}:|k|\geq \kappa_0\}.
\]

We say an admissible $k$ is \emph{strictly admissible} if
$$
\Im(k^2)\geq C \qtq{and} |\tfrac{\pi}{8} - \arg k| < \tfrac{\pi}{16}.
$$
Here $C\geq 1$ is an absolute constant large enough to ensure that Lemma~\ref{P:11} below holds for all strictly admissible $k$.  Note that for a strictly admissible $k$ we have $|\Im(k^2)|\approx |k|^2$.
\end{definition}

We have the basic estimate
\[
\| R_0(k)\|_{H_\kappa^{-1}\to H_\kappa^1}\lesssim 1
\]
uniformly over all admissible $k$. 
For $k>0$, $R_0(k)$ is positive and self-adjoint and hence has a unique positive square root.  For admissible $k$ we define $\sqrt{R_0(k)}$ as a Fourier multiplier operator, with symbol analytically continued from $k>0$.  We then have
\[
\|\sqrt{R_0(k)}\|_{L^2\to H_\kappa^1} + \|\sqrt{R_0(k)}\|_{H_{\kappa}^{-1}\to L^2} \lesssim 1
\]
uniformly over admissible $k$.

We record a few estimates for the free resolvent.

\begin{lemma}\label{I2H1}
For $\kappa>0$,
\[
\|\sqrt{R_0(\kappa)} f\sqrt{R_0(\kappa)}\|_{\mathfrak{I}_2} =\kappa^{-\frac12}\|f\|_{H_{\kappa}^{-1}}. 
\]
\end{lemma}

\begin{proof} Direct computation (as in \cite[Proposition~2.1]{KV}) shows
\begin{align*}
\|\sqrt{R_0(\kappa)}f \sqrt{R_0(\kappa)}\|_{\mathfrak{I}_2}^2  = \tr\bigl\{ R_0(\kappa)f R_0(\kappa) \bar f\bigr\}  = \tfrac{1}{\kappa}\int_\R \frac{|\hat f(\xi)|^2}{\xi^2+4\kappa^2}\,d\xi = \tfrac{1}{\kappa}\|f\|_{H_{\kappa}^{-1}}^2, 
\end{align*}
as desired. \end{proof}

The following lemma is easily verified from the explicit kernel of $R_0(k)$.  However, we use this opportunity to introduce the Combes--Thomas argument that will be very important later in the paper.  For a pedagogical introduction to this technique, see \cite{Hislop}.

\begin{lemma}\label{P:11} For $k$ admissible with $\kappa=|k|$ large enough and $0\leq n\leq 100$, 
\[
\| \langle x\rangle^{\pm n} R_0(k) \langle x\rangle^{\mp n}\|_{H_{\kappa}^{-1}\to H_{\kappa}^1} \lesssim 1. 
\]
\end{lemma}

\begin{proof} We write $\langle x\rangle^{\pm n} = e^{\rho(x)}$ with $\rho(x)=\pm n\log\langle x\rangle$. The inverse of the operator appearing in the lemma can be written
\[
e^{\rho}(-\partial_x^2+k^2)e^{-\rho}=-\partial_x^2+k^2+B_\rho,
\]
where
\[
B_{\rho} = \rho'(x)\partial_x + \partial_x \rho'(x) - \rho'(x)^2. 
\]
By duality, one has
\[
\|B_\rho\|_{H_{\kappa}^{1}\to H_{\kappa}^{-1}} \lesssim \kappa^{-1}\|\rho'\|_{L^\infty} + \kappa^{-2}\|\rho'\|_{L^\infty}^2 \lesssim n\kappa^{-1} +n^2\kappa^{-2}.
\]
Thus for $\kappa$ large enough we have 
\[
\|\sqrt{R_0(k)}B_\rho \sqrt{R_0(k)}\|_{L^2\to L^2} \leq \tfrac12
\]
and so we may write 
\[
e^{\rho}R_0(k) e^{-\rho}=\sqrt{R_0(k)}\bigl[1+\sqrt{R_0(k)}B_\rho\sqrt{R_0(k)}\bigr]^{-1} \sqrt{R_0(k)}.
\]
This implies the result. \end{proof}

We next discuss the existence of the resolvent $R(k)=(-\partial_x^2+q+k^2)^{-1}$ for $q\in H^{-1}(\R/2L_0\Z)$, where $L_0$ may be finite or infinite.  Note that when $L_0$ is finite, we still regard the operator as acting on the line, albeit with periodic potential.    

\begin{proposition}\label{P:HR} Fix $L_0\geq 1$ and $q\in H^{-1}(\R/2L_0\mathbb{Z})$. There is a unique semi-bounded self-adjoint operator $H=H_q$ on $L^2(\R)$ such that
\[
\langle \psi, H\psi\rangle = \int |\psi'(x)|^2 + q(x)|\psi(x)|^2\,dx \qtq{for all} \psi\in H^1(\R). 
\]
Indeed, 
\[
E_0(q):=\inf \sigma(H_q)\gtrsim -\|q\|_{H^{-1}}^4.
\]

The resolvent $R(k):=(H+k^2)^{-1}$ exists as a jointly analytic function of $(q,k)$  on the domain
\[
\{(q,k):q\in H^{-1}(\R/2L_0\mathbb{Z})\qtq{and} k^2\in \C\backslash(-\infty,-E_0(q)]\},
\]
taking values in the space of bounded operators from $H^{-1}(\R)$ to $H^1(\R)$.

Moreover, the diagonal Green's function $g(x;q,k):=\langle \delta_x,R(k)\delta_x\rangle$ and its reciprocal $\tfrac{1}{g(x;q,k)}$ are jointly analytic functions on the same domain, taking values in $H^1(\R/2L_0\mathbb{Z})$. 
\end{proposition}

\begin{proof} Everything in the first paragraph was observed already in \cite[Sections~2 and 6]{KV}.  As was also discussed there, for $\vk \gtrsim 1+\|q\|_{H^{-1}}^2$, one may construct the resolvent via the series
\begin{equation}\label{resolvent-series}
R(\vk) = \sum_{\ell=0}^\infty (-1)^\ell \sqrt{R_0(\vk)} \Bigl( \sqrt{R_0(\vk)}\,q\,\sqrt{R_0(\vk)}\Bigr)^\ell \sqrt{R_0(\vk)},
\end{equation}
which converges in the space of bounded operators from $H^{-1}(\R)$ to $H^{1}(\R)$.  Indeed, one sees that the form domain of $H_q$ is precisely $H^1(\R)$.   Boundedness of $R(k)$ from $H^{-1}(\R)$ to  $H^{1}(\R)$ for general $k$ then follows abstractly, as one sees from the resolvent identity
\begin{equation}\label{name}
R(k) = R(\vk) - ( k^2 - \vk^2 ) \sqrt{R(\vk)} R(k) \sqrt{R(\vk)}.
\end{equation}
Analyticity in $k$ and $q$ follows from the resolvent identities.

That $g(x;q,k)$ belongs to $H^1(\R/2L_0\Z)$ and is jointly analytic now follows from the expansion
\begin{align}\label{g-to-quadratic}
g(x;q, k) &= \tfrac{1}{2k} - \langle\delta_x, R_0(k) q R_0(k) \delta_x\rangle + \langle\delta_x, R_0(k) q R(k) q R_0(k) \delta_x\rangle\notag\\
&=\tfrac{1}{2k} -\langle\delta_x, \tfrac{1}{k}R_0(2k)q\rangle + \langle \delta_x, R_0(k) qR(k)qR_0(k)\delta_x\rangle;
\end{align}
see \cite{KV} for the details in the case $k\in\R$.

That $1/g(x)$ belongs to $H^1$ and is analytic will follow from the analogous statements for $g(x)$, once one shows that $g(x)$ is nowhere vanishing.  This non-vanishing property follows immediately from the spectral theorem; indeed, $R(k)$ is a (non-isometric) isomorphism of $H^{-1}\to H^1$ and so we have
\begin{gather}
\Re \langle \delta_x, R(k) \delta_x\rangle  \geq (\Re k^2 + E_0) \langle \delta_x, R^*(k) R(k) \delta_x\rangle > 0 \qtq{for} \Re k^2 > - E_0,
\end{gather}
as well as
\begin{gather}
\Im \langle \delta_x, R(k) \delta_x\rangle  = - \Im(k^2) \langle \delta_x, R^*(k) R(k) \delta_x\rangle \gtrless 0 \qtq{when}  \Im k^2 \lessgtr 0.
\end{gather}
We use here that  $\delta_x\in H^{-1}$.
\end{proof}

Let us recall some basic facts about the resolvent.  For $k$ strictly admissible, we have
\begin{equation}\label{RL2L2}
\|R(k)\|_{L^2\to L^2} \leq \tfrac{1}{\sigma}.
\end{equation} 
In what follows, we define $\sqrt{R(k)}$ via the spectral theorem, again choosing the branch by analytic continuation from $k\gg 1$. 

By using the series expansion \eqref{resolvent-series} and Lemma~\ref{I2H1}, we deduce that
\begin{equation}\label{I2H1-2}
|\tr\{R(q,\kappa)fR(q,\kappa)h\}| \lesssim \kappa^{-1}\|f\|_{H^{-1}}\|h\|_{H^{-1}}\qtq{provided} \kappa\gtrsim 1+\|q\|_{H^{-1}}^2.
\end{equation}
In particular,
\begin{equation}\label{I2H1-3}
\| \sqrt{R(q,\kappa)}f\sqrt{R(q,\kappa)}\|_{\mathfrak{I}_2} \lesssim \kappa^{-\frac12} \|f\|_{H^{-1}}\qtq{provided} \kappa\gtrsim 1+\|q\|_{H^{-1}}^2.
\end{equation}

A computation using the resolvent identity yields
\begin{equation}\label{E:dgdq}
\frac{d\ }{ds}\biggr|_{s=0} g(x;q+sf,k) =  -\langle \delta_x,R(q,k)fR(q,k)\delta_x\rangle \quad\text{for $f\in H^{-1}$}.
\end{equation} 

The diagonal Green's function satisfies the following important diffeomorphism property.  A local version of this property was detailed in \cite[Proposition~2.2]{KV}, which considered only the $\delta$-ball in $H^{-1}$ and required $k \gtrsim 1+ \delta^2$.

\begin{lemma}[Global diffeomorphism property]\label{L:diffeo}
Fix a strictly admissible $k$ and $1\leq L_0\leq \infty$.  Then the mappings
$$
q\mapsto g(k)-\tfrac1{2k} \qtq{and} q\mapsto k-\tfrac1{2g(k)} 
$$
are diffeomorphisms from $H^{-1}(\R/2L_0\Z;\R)$ onto their ranges in $H^1(\R/2L_0\Z;\C)$, which are smoothly embedded submanifolds.  In particular, one may recover $q$ from $g$ via
\begin{equation}\label{E:q from g}
q = \bigl[\tfrac{g'(k)}{2g(k)}\bigr]' + \bigl[\tfrac{g'(k)}{2g(k)}\bigr]^2 +  \bigl[\tfrac{1}{4g^2(k)} - k^2 \bigr]  .
\end{equation}
\end{lemma}

\begin{proof}
The central point here is to prove \eqref{E:q from g}.  In view of Proposition~\ref{P:HR}, both sides of this equality are analytic $H^{-1}$-valued functions of $q\in H^{-1}$.  Thus, it suffices to restrict attention to $q\in C^\infty(\R/2L_0\Z)$, in the case that $L_0$ is finite, or to Schwartz class $q$, when $L_0=\infty$.  In this setting (see \cite{MR0069338}), we may express the Green's function in terms of the Weyl solutions
$$
-\psi_\pm'' + q \psi_\pm = - k^2 \psi_\pm
$$
which we normalize to have Wronskian
$$
\psi_+(x) \psi_-'(x) - \psi_+'(x) \psi_-(x) \equiv 1
$$
and to be square-integrable at $\pm\infty$, respectively.  In this way, we have
$$
g(x;q, k) = \psi_+(x)\psi_-(x),
$$
which is nowhere-vanishing.  Now we simply compute:
$$
\frac{g'(k)}{2g(k)} = \frac{\psi_+'}{2\psi_+} + \frac{\psi_-'}{2\psi_-}
	\qtq{so} \Bigl(\frac{g'(k)}{2g(k)}\Bigr)' =  (q+k^2) - \tfrac12\Bigl(\frac{\psi_+'}{\psi_+}\Bigr)^2 - \tfrac12\Bigl(\frac{\psi_-'}{\psi_-}\Bigr)^2
$$
and thus
$$
\Bigl(\frac{g'(k)}{2g(k)}\Bigr)' + \Bigl(\frac{g'(k)}{2g(k)}\Bigr)^2 =  (q+k^2) - \tfrac14\Bigl(\frac{\psi_+'}{\psi_+} - \frac{\psi_-'}{\psi_-}\Bigr)^2
	= q + k^2 - \frac{1}{4g^2(k)},
$$
by using the Wronskian relation.  This proves \eqref{E:q from g}.

We are now ready to show that both maps are global embeddings.  From Proposition~\ref{P:HR} we know that the two mappings are smooth (indeed, real analytic).  On the other hand, \eqref{E:q from g} shows not only that the mappings are injective but so is the differential:
$$
\Bigl[\tfrac{1}{2g(k)}\partial_x^2 - \tfrac{1}{2g(k)} \tfrac{g'(k)}{g(k)} \partial_x - \tfrac{1}{2g(k)} \bigl(\tfrac{g'(k)}{g(k)}\bigr)' -\tfrac{1}{2g^3(k)}\Bigr] \tfrac{\delta g(k)}{\delta q} = \Id.
$$

Note that the operator in square brackets is bounded from $H^1\to H^{-1}$; thus, the implicit function theorem guarantees that the mappings are indeed smooth embeddings.
\end{proof}

\section{Invariance of white noise under the \texorpdfstring{$\H_k$}{Hk} flow on the torus}\label{S:commute}

As in \cite{KV}, an essential ingredient in our analysis will be the use of a suitable Hamiltonian approximation to the flow \eqref{KdV}. We now introduce the key quantities in our analysis, adapted from their definition in \cite{KV} to allow complex parameters $k$.  Specifically, for $1\leq L_0\leq \infty$, $q\in H^{-1}(\R/2L_0\Z)$, and $k$ strictly admissible, we define
\begin{gather}
\label{complex rho}
\rho(x;q,k) := k - \tfrac{1}{2g(x;q,k)} + \tfrac12\int_\R e^{-2k|x-y|} q(y)\,dy,  \\
\label{complex alpha}
\alpha(q,k) := \int_{-L_0}^{L_0} \rho(x;q,k)\,dx, \\
\label{complex H_k}
\H_k(q) := \Re \Bigl\{- 16 k^5 \alpha(q,k) + 2 k^2 \int_{-L_0}^{L_0} q(x)^2\,dx\Bigr\}.
\end{gather}
Naturally, many identities derived in \cite{KV} carry over immediately, either by analytic continuation, or by simply repeating the earlier arguments.  Note that 
$$
\rho(x;q,\bar k)=\overline{\rho(x;q,k)} \qtq{and correspondingly,} \alpha(q,\bar k)=\overline{\alpha(q,k)}.
$$

The necessity of considering complex $k$ stems from the fact that we must simultaneously avoid the spectrum of $H_q$ for all samples of white noise in order to define the resolvents $R(q,k)$ as random variables (cf. Proposition~\ref{P:HR}). In particular, while white noise almost surely belongs to $H^{-1}$ on the torus (in fact, the distribution of its $H^{-1}$ norm has exponential tails), there is no choice of $\kappa_0>0$ such that $R(q,\kappa)$ exists almost surely for all $\kappa\geq \kappa_0$. 

Note that formally, at least, the $\H_k$ flows approach the KdV flow as $|k|\to\infty$.  Indeed, as was discussed in \cite{KV}, for Schwartz functions $q$ one has
\[
\alpha(q,k)=\tfrac{1}{4k^3}P(q) - \tfrac{1}{16k^5} \H_{\text{KdV}}(q) + \mathcal{O}(\tfrac{1}{k^7}),
\]
where $P(q) = \tfrac12\int q^2$ is the momentum (i.e. generator of translations).  

We first establish an $H^{-1}$ global well-posedness result for the $\H_k$ flows on the torus.  For $k\geq 1$ and data small in $H^{-1}$, this appears already in \cite{KV}.  The argument here differs in two ways from that in \cite{KV}: $k$ is complex and (more significantly) we consider arbitrarily large $H^{-1}$ data.

\begin{prop}\label{P:H kappa} Fix $L_0\geq 1$ and $k$ strictly admissible.  Then the Hamiltonian flow induced by $\H_k$,
\begin{align}\label{H_k flow q}
\tfrac{d\ }{dt} q(x) = \Re \bigl\{ 16 k^5 g'(x;k) + 4k^2 q'(x) \bigr\},
\end{align}
is globally well-posed on $H^{-1}(\R/2L_0\Z)$ and commutes with the KdV flow.  If $q$ evolves according to \eqref{H_k flow q}, then for any strictly admissible $\vk\neq k$ the diagonal Green's function obeys
\begin{equation}\label{dt1/g}
\partial_t \tfrac{1}{2g(\vk)} = \biggl[\tfrac{k^2+ \bar k^2}{g(\vk)}-	\tfrac{2k^5}{\vphantom{(^{2^2}}k^2-\vk^2}\tfrac{g(k)}{g(\vk)}- \tfrac{2\bar k^5}{\vphantom{(^{2^2}}\bar k^2-\vk^2}\tfrac{g(\bar k)}{g(\vk)}\biggr]'. 
\end{equation}
\end{prop}

\begin{proof} 
The derivation of \eqref{H_k flow q} and \eqref{dt1/g} is easily adapted from the corresponding statements in \cite[Proposition~3.2]{KV}, where $k, \vk$ were real with $\vk\neq k\geq 1$. 

We turn to the question of well-posedness.  We rewrite \eqref{H_k flow q} in Duhamel form:
\begin{equation}\label{Duhamel}
q(t) = e^{4\Re k^2 t\partial} q(0)+\int_0^t e^{4\Re k^2(t-s)\partial} \Re[16 k^5 g'(q(s),k)]\,ds.
\end{equation}
To establish local well-posedness, it suffices to prove that $q\mapsto g'(q,k)$ is Lipschitz from any ball in $H^{-1}$ into $H^{-1}$.  Equivalently, we may show that $q\mapsto g(q,k)$ is Lipschitz from any ball in $H^{-1}$ into $L^2$.  To this end, we let $q_0\neq q_1$ in $B_r(0)\subset H^{-1}$ and endeavor to estimate
\[
|\langle g(q_1,k)-g(q_0,k),h\rangle|=\biggl| \int_0^1 \tr\{R(q_\theta,k)(q_1-q_0)R(q_\theta,k)h\}\,d\theta\biggr|,
\]
where $h\in L^2$ is a unit vector, $q_\theta:=\theta q_1+(1-\theta)q_0$, and we have utilized \eqref{E:dgdq}.  To proceed, we omit the dependence on $q_\theta$ (using only that $\|q_\theta\|_{H^{-1}}\leq 2r$), write $f=q_1-q_0$, and let $\kappa_0\geq 1$ to be chosen below.  Using \eqref{name}, we estimate 
\begin{align}
|\tr\{& R(k)fR(k)h\}| \nonumber \\
&\leq |\tr\{R(\kappa_0)fR(\kappa_0)h\}| \label{wp1}\\
& \quad + |k^2-\kappa_0^2|\tr\{\sqrt{R(\kappa_0)}R(k)\sqrt{R(\kappa_0)}f R(\kappa_0)h\}| \label{wp2}\\
& \quad + |k^2-\kappa_0^2|\tr\{R(\kappa_0)f\sqrt{R(\kappa_0)}R(k)\sqrt{R(\kappa_0)}h\}| \label{wp3}\\
& \quad + |k^2-\kappa_0^2|^2|\tr\{\sqrt{R(\kappa_0)}R(k)\sqrt{R(\kappa_0)}f\sqrt{R(\kappa_0)}R(k)\sqrt{R(\kappa_0)}h\}|.\label{wp4} 
\end{align}

We first use \eqref{I2H1-2} to get 
\[
\eqref{wp1} \lesssim \kappa_0^{-1}\|f\|_{H^{-1}}\|h\|_{L^2}
\]
provided $\kappa_0\gtrsim 1+r^2$.  Similarly, using \eqref{I2H1-3} and \eqref{RL2L2}, we can bound
\[
\eqref{wp2}+\eqref{wp3} \lesssim \tfrac{|k^2-\kappa_0^2|}{\sigma} \kappa_0^{-1}\|f\|_{H^{-1}}\|h\|_{L^2}
\]
and
\[
\eqref{wp4} \lesssim \tfrac{|k^2-\kappa_0^2|^2}{\sigma^2} \kappa_0^{-1}\|f\|_{H^{-1}}\|h\|_{L^2}.
\]
We conclude that $q\mapsto g(q,k)$ is Lipschitz from any ball in $H^{-1}$ into $L^2$, which implies local well-posedness for the $\mathcal{H}_k$ flow. 

Global well-posedness follows from the conservation of $\alpha$, which in turn follows from \eqref{dt1/g} as in \cite[Proposition~3.2]{KV}.  Indeed, as
\begin{equation}\label{alpha-H-1}
\alpha(q,\vk)\approx \tfrac{1}{\vk}\| q\|_{H_{\vk}^{-1}}^2 \qtq{provided}\vk\gtrsim 1+\|q\|_{H_{\vk}^{-1}}^2 
\end{equation}
(cf. \cite[Equation (2.20)]{KV} and \cite[Equation~(6.10)]{KV}), conservation of $\alpha$ yields
\begin{equation}\label{H-1bds}
\sup_{t\in\R} \|q(t)\|_{H^{-1}} \leq C\bigl[\|q(0)\|_{H^{-1}}+\|q(0)\|_{H^{-1}}^3\bigr]
\end{equation}
for some absolute constant $C>0$. 

As $\alpha$ and the momentum $P$ are both conserved by the $\H_k$ and KdV flows, we see that not only do KdV and $\H_k$ commute, but so do $\H_k$ and $\H_\vk$ for any strictly admissible $k, \vk$. 
\end{proof}

Next, we wish to observe that for a given initial data, the $\H_k$ flows converge to the KdV flow as $|k|\to\infty$.  In the case where $k >0$, this was proved both on the line and on the torus in \cite[Theorem~5.1]{KV}.  With Proposition~\ref{P:H kappa} in place, minor modifications are needed to treat complex $k$. 

\begin{proposition}[Convergence of flows]\label{P:Hk-to-KdV} Fix $L_0\geq 1$ and $q^0\in H^{-1}(\R/2L_0\Z)$.  Given $k$ strictly admissible, let
$q_k$ denote the global solution to the $\H_k$ flow with initial data $q^0$ and let $q$ denote the solution to KdV with the same initial data. For any $T>0$, we have
\[
\lim_{|k|\to\infty} \sup_{|t|\leq T}\|q(t)-q_k(t)\|_{H^{-1}}=0. 
\]
\end{proposition}

As white noise on the torus is almost surely in $H^{-1}$, Proposition~\ref{P:H kappa} allows us to solve the $\H_k$ flow \eqref{H_k flow q} on any torus with white noise initial data.  We next prove that this flow preserves white noise measure.

\begin{theorem}\label{T:Hktorus}  Fix $L_0\geq 1$ and $k$ strictly admissible.  The $\H_k$ flow \eqref{H_k flow q} preserves white noise measure on $\R/2L_0\mathbb{Z}$. 
\end{theorem}

\begin{proof}
Suppose, contrary to the theorem, that white noise is not preserved by
the $\H_k$ flow \eqref{H_k flow q}.  Then there is a Schwartz function
$\phi$ and a time $T>0$ so that
\begin{equation}\label{E:P:T:Hktorus}
\E\Bigl\{ e^{i\langle \phi, q(T)\rangle} \Bigr\} \neq e^{-\frac12\int \phi^2}.
\end{equation}
Here $q(t)$ denotes the solution to the $\H_k$ flow with white-noise initial data.

Proceeding with this choice of $\phi$ and $T$, we define $f:\R\times
H^{-1}\to \C$ by
\[
f(t,q) = \exp\Bigl( i \bigl\langle \phi,\ e^{(T-t)J\nabla \H_k}q
\bigr\rangle\Bigr).
\]
Note that by the proof of Proposition~\ref{P:H kappa}, the data-to-solution map for the $\H_k$ flow is smooth on $H^{-1}$ with derivatives growing at
most polynomially.  Correspondingly, $q\mapsto f(t,q)$ is smooth and
\begin{equation}\label{Hktorus deriv bounds}
\bigl\| \tfrac{\delta f(t,q) }{\delta q}  \bigr\|_{H^{1}} \lesssim
\Bigl[ 1 + \| q \|_{H^{-1}}  \Bigr]^{C}
\end{equation}
uniformly for $t\in[0,T]$ and some absolute constant $C$.

Given the manner in which $f$ is defined and the fact that $q(0)$ is
white noise distributed, we may recast \eqref{E:P:T:Hktorus} as
\[
\E\bigl\{ f(0,q(0)) \bigr\} \neq \E\bigl\{ f(T,q(0)) \bigr\}.
\]
The remainder of this proof is devoted to refuting this assertion by
proving that
\begin{equation}\label{E:P:T:Hktorus'}
\partial_t \E\bigl\{ f(t,q(0)) \bigr\} \equiv 0.
\end{equation}

Using the group property of a well-posed flow and translation
invariance of white noise, we have
\begin{align*}
\E\Bigl\{ f(t+h,q(0)) \Bigr\} = \E\Bigl\{ f\bigl(t, e^{-hJ\nabla \H_k}
q(0)) \Bigr\} = \E\Bigl\{ f\bigl(t, e^{-hJ\nabla \H_k} e^{4\Re k^2 h\partial}
q(0)\bigr) \Bigr\}.
\end{align*}
On the other hand, by
the Duhamel formula \eqref{Duhamel},
\[
\tfrac{d\ }{dh}\Bigr|_{h=0}   [e^{-hJ\nabla \H_k} e^{4\Re k^2 h\partial} q(0)]
= -16 \Re k^5 g'(x;q(0),k).
\]
In this way, we deduce that $\E\{ f(t,q(0)) \}$ is differentiable with respect to $t$ and 
\[
\partial_t \E\Bigl\{ f(t,q(0)) \Bigr\} = \E\Bigl\{ -16  \bigl\langle \tfrac{\delta f}{\delta q}(t,q(0)),\ \Re k^5 g'(q(0),k) \bigr\rangle \Bigr\}.
\]
Thus, it suffices to show that
\begin{equation}\label{E:P:T:Hktorus''}
\E\Bigl\{ \bigl\langle \tfrac{\delta F}{\delta q}(q(0)),\ g'(q(0),k)
\bigr\rangle \Bigr\} = 0
\end{equation}
for any function $F\in D^2_1$; see \eqref{247}.  Note that $q\mapsto f(t, q)$ satisfies \eqref{247} by virtue of \eqref{Hktorus deriv bounds}.

In order to verify \eqref{E:P:T:Hktorus''}, it suffices to treat a dense class of functions in $D^2_1$.  A convenient class of such functions are `polynomials', that is, finite linear combinations of functions of the form $F(q)= \prod_n \langle \varphi_n, q\rangle$ where $\varphi_n$ are (not necessarily distinct) finitely many Schwartz functions.  In fact, the Hilbert space $D^2_1$ admits an orthogonal basis of such polynomials, namely,  finite products of Hermite polynomials in the individual $X_k$ appearing in \eqref{ONB rep}; see \cite[Chapter V]{Malliavin}.

In what follows, we write $q(0)=q$ and $g(q(0),k)=g$.  We begin with the case of monomials: $F(q)=\langle \varphi, q\rangle^n$, where $ \varphi \in \mathcal S$ is fixed.  Differentiating $F$ and then integrating by parts, we are left to prove
\begin{equation}\label{Bn0}
\E\bigl\{ \langle \varphi,q\rangle^{n-1}\langle \varphi',g\rangle\bigr\} =0. 
\end{equation}

We proceed by induction.  If $n=1$, the desired result follows from translation invariance of white noise, which guarantees that the law of $g$ is also translation invariant.  Indeed, we have
\[
0=\partial_h\Bigl|_{h=0}\E\bigl\{ \langle \varphi_h, g\rangle\bigr\}=\E\bigl\{ \langle \varphi',g\rangle\bigr\},
\]
where $\varphi_h(x)=\varphi(x+h)$.

Now suppose $n\geq 2$ and that \eqref{Bn0} holds for monomials of lesser degree.  By translation invariance,
\begin{align*}
0=\partial_h\Bigl|_{h=0} \E\bigl\{ \langle \varphi_h, q\rangle^{n-1}\langle \varphi_h, g\rangle\bigr\}.
\end{align*}
Using this identity and then applying Lemma~\ref{T:Wick} in the direction $\varphi'$, the fact that $\langle \varphi, \varphi'\rangle=0$, and \eqref{E:dgdq}, we get
\begin{align}\label{Eversion1}
\E\bigl\{ \langle \varphi,q\rangle^{n-1}\langle\varphi', g\rangle\} &= -(n-1) \E\bigl\{\langle \varphi,q\rangle^{n-2}\langle\varphi', q\rangle\langle \varphi, g\rangle\bigr\}\notag\\
&=(n-1) \E\bigl\{\langle \varphi,q\rangle^{n-2}\tr\{R(q)\varphi R(q)\varphi'\}\bigr\}.
\end{align}

On the other hand, applying Lemma~\ref{T:Wick} in the direction $\varphi$, using \eqref{E:dgdq}, and recalling the inductive hypothesis, we compute
\begin{align*}
\E\bigl\{ \langle \varphi,q\rangle^{n-1}\langle \varphi', g\rangle\bigr\}
& = (n-2)\E\{\langle \varphi,q\rangle^{n-3}\langle \varphi',g\rangle\} - \E\{\langle \varphi,q\rangle^{n-2}\tr\{R(q)\varphi R(q)\varphi'\}\bigr\} \\
& = -\E\bigl\{\langle \varphi,q\rangle^{n-2}\tr\{R(q)\varphi R(q)\varphi'\}\bigr\}.
\end{align*}
Comparing with \eqref{Eversion1}, we deduce
\[
E\{\langle \varphi,q\rangle^{n-1}\langle \varphi', g\rangle\} = 0,
\]
as desired.  This concludes the proof of \eqref{E:P:T:Hktorus''} for the case of monomials. 

Simple algebra (applying the above to linear combinations of test functions) allows one to pass from the case of monomials to multinomials
$
F(q)= \prod_n \langle \varphi_n, q\rangle
$
and thence to polynomials, as desired.
\end{proof}

We are now in a position to present a new proof of the invariance of white noise for KdV on the torus.  As noted in the introduction, this result was shown previously in \cite{Oh, Oh2, OhQuastelValko, QuastelValko} by different methods.  The problem of constructing dynamics for white noise initial data was resolved earlier in \cite{MR2267286}, which proved well-posedness on the whole space $q^0\in H^{-1}(\R/2L_0\Z)$.  Note that while our construction of the solutions follows \cite{KV}, the solutions themselves coincide with the Kappeler--Topalov solutions; they are the \emph{unique} limits of smooth solutions.  

\begin{theorem}[Invariance of white noise for KdV on the torus]\label{T:KdV-torus} Fix $L_0\geq 1$ and let $q^0$ be white noise distributed on $\R/2L_0\Z$.  Let $q$ denote the global solution to KdV with initial data $q^0$.  Then $q(t)$ is white noise distributed for all $t\in\R$. 
\end{theorem}

\begin{proof} As $q^0\in H^{-1}(\R/2L_0\Z)$ almost surely, we can solve both the $\H_k$ flow (for any $k$ strictly admissible) and KdV with initial data $q^0$, yielding global solutions $q_k$ and $q$, respectively.  By Theorem~\ref{T:Hktorus}, we have that $q_k(t)$ is white noise distributed at each $t\in\R$.

For $t\in \R$, Proposition~\ref{P:Hk-to-KdV} yields
\[
\lim_{|k|\to\infty} \|q(t)-q_k(t)\|_{H^{-1}} = 0 \qtq{almost surely.} 
\]
In particular, for any $\varphi\in\cS$ we have
\[
\lim_{|k|\to\infty} \langle \varphi,q_k(t)\rangle = \langle \varphi,q(t)\rangle \qtq{almost surely.}
\]
Thus, by dominated convergence and the fact that $q_k(t)$ is white noise distributed, we deduce
\[
\E\bigl\{ e^{i\langle \varphi,q(t)\rangle} \bigr\} = \exp\bigl\{-\tfrac12\|\varphi\|_{L^2}^2\bigr\},
\]
which completes the proof. \end{proof}

\begin{remark}In the preceding argument, we approximated KdV on the torus using $\H_k$ flows with strictly admissible $k$.  This anticipates the approach we will take later when we consider KdV on the line, in which case it is necessary to restrict to complex $k$.

However, if one is only interested in the torus, then it is enough to use $\H_{\kappa}$ flows with real $\kappa$, even though these flows cannot be defined for all samples of white noise.  Fixing $r>0$, we define $\kappa_r= \delta^{-1}(1+r^2)$,
$$
B_r =\{q: \, \|q\|_{H^{-1}_{\kappa_r}}\leq r \text{ and } \alpha(q, \kappa_r) < \delta\}, \qtq{and} \Omega_r= \{\omega:\, q^0\in B_r\},
$$ 
where $\delta>0$ is a small constant dictated by the results of \cite{KV}.  From that paper we see that for $\kappa\geq \kappa_r$, the $\H_\kappa$ flow is globally well-posed on the set $B_r$ which is invariant under the flow.   Moreover, $B_r$ is also invariant for the KdV flow and the $\H_\kappa$ flows converge to KdV on $B_r$ as $\kappa\to \infty$.  As the probability of $\Omega_r^c$ is exponentially small, this is sufficient to prove that the solution $q(t)$ to KdV is white noise distributed at each $t$.  Indeed, for any $r>0$, one has
\begin{align*}
\E\bigl\{e^{i\langle \varphi,q(t)\rangle}\bigr\} & = \E\bigl\{1_{\Omega_r} e^{i\langle \varphi,q(t)\rangle}\bigr\} + \mathcal{O}(e^{-cr}) \\
& = \lim_{\kappa\to \infty} \E\bigl\{ 1_{\Omega_r}e^{i\langle \varphi, q_\kappa(t)\rangle}\bigr\} + \mathcal{O}(e^{-cr}) \\
& = \lim_{\kappa\to \infty}\E\bigl\{ 1_{\Omega_r} e^{i\langle \varphi,q_\kappa(0)\rangle}\bigr\} + \mathcal{O}(e^{-cr}) \\
& = \exp\{-\tfrac12\|\varphi\|_{L^2}^2\} + \mathcal{O}(e^{-cr}).
\end{align*}
Sending $r\to \infty$ yields the result. 
\end{remark}

\section{Single scale analysis}\label{S:singlescale}
Throughout this section we restrict $L_0\in 2^{\mathbb{N}_0}\cup \{\infty\}$ and let $q$ be white noise distributed on $\R$.  For dyadic $1\leq L\leq L_0\in 2^{\mathbb{N}_0}$, we define
\begin{equation}\label{qL}
q_L(x) = \sum_{n\in\mathbb{Z}} \bigl[ 1_{[-L,L]}q\bigr](x-2nL_0),
\end{equation}
which produces a distribution that is $2L_0$-periodic.  If $L_0=\infty$ and $L\in 2^{\mathbb N_0}$, we set
\[
q_L(x) = \bigl[1_{[-L,L]} q\bigr](x).
\]

We further define
\[
R_L(k) = (H_L + k^2)^{-1}, 
\]
for strictly admissible $k$, where $H_L$ is the Schr\"odinger operator with potential $q_L$ (cf. Proposition~\ref{P:HR}). 

Given a sample from white noise on the line, the $L=L_0$ case of  \eqref{qL} provides a recipe for constructing a copy of white noise on the torus $\R/2L_0\Z$.  Indeed, this is precisely the manner in which we will ultimately couple our evolution problems on the line and on the torus.  The rationale for allowing $L<L_0$ will not be apparent in this section; rather, it is inspired by the needs of the multiscale analysis in the next section.  In fact, the `multiple scales' are precisely the different values of $L$, as we successively `reveal' ever more of the potential by sending $L\to L_0$.

The main results of this section are probabilistic estimates for the operators $R_L(k)$.  While the bounds we obtain in this section do deteriorate as $L\to \infty$, they provide the crucial foundation for the next section, where we employ multiscale analysis to obtain bounds independent of $L$ and $L_0$.  We begin with the following lemma.

\begin{lemma}\label{L:HS} For $\kappa>0$, $q$ white noise distributed, $\phi\in L^2$, and $1\leq p<\infty$, 
\[
\E\bigl\{ \bigl\| \sqrt{R_0(\kappa)} q\phi \sqrt{R_0(\kappa)} \bigr\|_{\HS}^p\bigr\}\lesssim_p \|\phi\|_{L^2}^p \kappa^{-p}. 
\]
\end{lemma}

\begin{proof} Lemma~\ref{I2H1} implies 
\begin{equation}\label{HS-norm}
\begin{aligned}\E\bigl\{ \bigl \|\sqrt{R_0(\kappa)}q\phi\sqrt{R_0(\kappa)} \bigr\|_{\HS}^2 \bigr\}  = \tfrac{1}{\kappa}\int_\R\frac{\E\{|\wh{\phi q}(\xi)|^2\}}{\xi^2+4\kappa^2}\,d\xi. 
\end{aligned}
\end{equation}
As $q$ is white noise distributed, we have that
\begin{align}\label{931}
\E\{|\widehat{\phi q}(\xi)|^2\} & = \tfrac{1}{2\pi} \E\bigl\{ |\langle \phi e^{-ix\xi},q\rangle |^2\} \notag\\
& = \tfrac{1}{2\pi} \E\bigl\{ \bigl[\langle \cos(x\xi)\phi,q\rangle\bigr]^2+\bigl[\langle \sin(x\xi)\phi,q\rangle\bigr]^2\bigr\} \notag \\
& = \tfrac{1}{2\pi} \int [\phi(x)]^2[\cos^2(x\xi)+\sin^2(x\xi)]\,dx = \tfrac1{2\pi} \|\phi\|_{L^2}^2,
\end{align}
uniformly in $\xi$. Continuing from \eqref{HS-norm} we deduce
\begin{equation}\label{HS-norm2}
\E\bigl\{\|\sqrt{R_0(\kappa)}q\phi\sqrt{R_0(\kappa)}\|_{\mathfrak{I}_2}^2\bigr\} = \tfrac{1}{4\kappa^2}\|\phi\|_{L^2}^2, 
\end{equation}
which settles $p=2$.  To extend this to $1\leq p<\infty$, we use the fact that the square of the Hilbert--Schmidt norm is a quadratic form in Gaussian random variables (cf. \eqref{ONB rep}) and Lemma~\ref{L3}. \end{proof}

Before we state the next lemma, we remind the reader of the notation $A(\kappa_0)$ from Definition~\ref{D:admissible}. 

\begin{lemma}\label{P:1}  There exists $c>0$ such that
\begin{equation}\label{E:P:1}
\PP\biggl\{\sup_{k\in A(\kappa_0)} |k|\kappa_0\|\sqrt{R_0(k)}q_L \sqrt{R_0(k)}\|_{L^2\to L^2}^2\geq\lambda\biggr\} \lesssim e^{-c\lambda}
\end{equation}
uniformly for $\kappa_0\geq 1$, $2\leq L\leq L_0$, and $\lambda\geq c^{-1} \log L$.

Consequently, for $1\leq p<\infty$ and $\kappa\geq 1$,
\begin{align}\label{E:C:q op}
\E\bigl\{ \|q_L\|_{H^1_\kappa\to H^{-1}_\kappa}^p \bigr\}\lesssim_p [\log L]^{\frac p2} \kappa^{-p}.
\end{align}
\end{lemma}

\begin{proof} We begin by choosing $\varphi\in C_c^\infty(\R)$ satisfying
\[
\sum_{n\in\mathbb{Z}} \varphi_n^2(x)=1,\qtq{where}\varphi_n(x)=\varphi(x-n),
\]
as well as $\varphi_n\varphi_m\equiv 0$ for $|n-m|>1$.  Consider
\[
Q_n := \sum_{\kappa\in \kappa_0 2^{\mathbb{N}_0}}\kappa\kappa_0\|\sqrt{R_0(\kappa)}q_L\varphi_n^2\sqrt{R_0(\kappa)}\|_{\mathfrak{I}_2}^2
\]
and observe the following large deviation bound
\begin{equation}\label{large-deviation}
\PP\biggl\{ Q_{n}\geq\mu\biggr\} \lesssim e^{-c_0\mu},
\end{equation}
which holds for some fixed $c_0>0$, uniformly over $n\in\mathbb{Z}$.  Indeed, arguing as in the proof of Lemma~\ref{L:HS} (see \eqref{HS-norm2}) and performing a change of variables, we may write
\begin{align*}
\E\biggl\{ \sum_{\kappa\in \kappa_0 2^{\mathbb{N}_0}}\kappa\kappa_0\|\sqrt{R_0(\kappa)}q_L \varphi_n^2\sqrt{R_0(\kappa)}\|_{\mathfrak{I}_2}^2 \biggr\} & \approx  \sum_{\kappa\in \kappa_02^{\mathbb{N}_0}} \!\!\tfrac{\kappa_0}{\kappa}\|\varphi_n^2\|_{L^2}^2 
 \approx \|\varphi^2\|_{L^2}^2 \lesssim 1. 
\end{align*}
Thus, by Lemma~\ref{L3}, we have
\[
\E\bigl\{ e^{c_0Q_n}\bigr\} \lesssim 1
\]
for some $c_0>0$, which guarantees \eqref{large-deviation}. 

Next, we observe that $\{Q_{n}:n$ odd$\}$ and $\{Q_{n}:n$ even$\}$ are sets of only $\mathcal{O}(L)$-many distinct random variables.  Moreover, each set is comprised of independent random variables; see Remark~\ref{R:cov}.  Thus, by Lemma~\ref{L:sup},
\begin{equation}\label{large-deviation2}
\PP\biggl\{\sup_{n\in\mathbb{Z}} Q_{n}\geq\lambda\biggr\}\lesssim\begin{cases} 1 & \lambda<c^{-1}\log L \\ e^{-c\lambda}& 
\lambda\geq c^{-1}\log L\end{cases}
\end{equation}
for some absolute $c>0$.  

Next, fix $f\in L^2$ and define
\[
f_n:=\sqrt{-\partial^2+\kappa^2}\varphi_n\sqrt{R_0(\kappa)}f,
\]
which satisfy
\[
\sum_n \|f_n\|_{L^2}^2 \lesssim \|f\|_{L^2}^2. 
\]
Exploiting self-adjointness and writing
\[
1=\sum_n \varphi_n^2 = \sum_{|n-m|\leq 1}\varphi_n^2 \varphi_m^2,
\]
we get
\begin{align*}
\langle f,\sqrt{R_0(\kappa)}q_L\sqrt{R_0(\kappa)}f\rangle & = \sum_{|n-m|\leq 1}\langle f_m,\sqrt{R_0(\kappa)}q_L\varphi_n^2\sqrt{R_0(\kappa)}f_m\rangle. 
\end{align*}
We can now apply Cauchy--Schwarz to deduce
\begin{equation}\label{P:1-wws2}
\bigl|\langle f,\sqrt{R_0(\kappa)}q_L\sqrt{R_0(\kappa)}f\rangle\bigr|\lesssim \sum_m \|f_m\|_{L^2}^2\sup_n\|\sqrt{R_0(\kappa)}q_L \varphi_n^2\sqrt{R_0(\kappa)}\|_{L^2\to L^2}. 
\end{equation}

By duality and polarization, \eqref{P:1-wws2} implies
\begin{equation}\label{P:1-wws}
\|\sqrt{R_0(\kappa)}q_L\sqrt{R_0(\kappa)}\|_{L^2\to L^2} \lesssim \sup_{n} \|\sqrt{R_0(\kappa)}q_L\varphi_n^2\sqrt{R_0(\kappa)}\|_{L^2\to L^2}. 
\end{equation}
Using this together with
\[
\sup_{\kappa\in \kappa_0 2^{\mathbb{N}_0}}\kappa\kappa_0  \sup_{n}\|\sqrt{R_0(\kappa)}q_L\varphi_n^2\sqrt{R_0(\kappa)}\|_{L^2\to L^2}^2 \lesssim \sup_{n} Q_{n}
\]
and \eqref{large-deviation2}, we deduce
\[
\PP\biggl\{ \sup_{\kappa\in \kappa_0 2^{\mathbb N}} \kappa\kappa_0 \|\sqrt{R_0(\kappa)}q_L\sqrt{R_0(\kappa)}\|_{L^2\to L^2}^2 \geq \lambda\biggr\} \lesssim e^{-c\lambda},
\]
uniformly for $\kappa_0\geq 1$, $L\geq 2$, and $\lambda\geq c^{-1} \log L$.  This finally implies \eqref{E:P:1}, as we have the bound
\[
\|\sqrt{R_0(k)}(-\partial^2+\kappa^2)^{\frac12}\|_{L^2\to L^2} \lesssim 1, 
\]
uniformly for $k\in A(\kappa_0)$ with $\kappa\leq |k|\leq 2\kappa$. This then yields \eqref{E:C:q op} by the standard argument. \end{proof}

We now turn to our first key result of this section. 

\begin{proposition}\label{P:2} The following holds uniformly for finite $L$ satisfying  $2\leq L\leq L_0\in 2^{\mathbb N}\cup\{\infty\}$:  There exists $0<c_0<1$ such that for any $\kappa_1\geq c_0^{-1}\sqrt{\log L}$, there is an event $\Omega=\Omega(\kappa_1)$ on which
\begin{equation}\label{P2eqn}
\left.\begin{aligned}
\biggl\| \sqrt{R_L^*(k)R_L(k)}\biggr\|_{H_{\kappa}^{-1}\to H_{\kappa}^1} &\lesssim 1 \qtq{if} \kappa\geq \kappa_1, \\
\biggl\| \sqrt{R_L^*(k)R_L(k)}\biggr\|_{H_{\kappa_1}^{-1}\to H_{\kappa_1}^1} &\lesssim 1+\tfrac{\kappa_1^2}{\sigma} \qtq{if} \kappa\leq \kappa_1,
\end{aligned}
\qquad \right\}
\end{equation}
uniformly for strictly admissible $k$, where we used the notations \eqref{b=a}.  Moreover, 
\[
\PP(\Omega^c)\lesssim e^{-c_0\kappa_1^2}. 
\]
\end{proposition}

Before proving this proposition, let us record a useful corollary.

\begin{corollary} \label{C:2} Define
\begin{equation}\label{XL}
X_L(k):=\bigl\|\sqrt{R_L(k)}\bigr\|_{L^2\to H_\kappa^1}^2 = \bigl\| \sqrt{R_L(k)}\bigr\|_{H_\kappa^{-1}\to L^2}^2 \geq \|R_L(k)\|_{H_{\kappa}^{-1}\to H_{\kappa}^1}.  
\end{equation}
There exists $c>0$ such that
\[
\PP\bigl\{X_L(k)\geq \lambda\bigr\} \lesssim e^{-c\lambda} \qtq{for all} \lambda \geq c^{-1}\log L,
\]
uniformly for finite $L$ satisfying  $2\leq L\leq L_0\in 2^{\mathbb N}\cup\{\infty\}$ and strictly admissible $k$.  Consequently, for $0<p<\infty$, we have
\[
\E\biggl\{\bigl\| \sqrt{R_L(k)}\bigr\|_{L^2\to H_{\kappa}^1}^p\biggr\} = \E\biggl\{ \bigl\| \sqrt{R_L(k)}\bigr\|_{H_{\kappa}^{-1}\to L^2}^p\biggr\} \lesssim_p [\log L]^{\frac{p}{2}},
\]
uniformly in $2\leq L\leq L_0$ and strictly admissible $k$.  In particular,
\begin{equation}\label{RLk-itself}
\E\biggl\{\|R_L(k)\|_{H_\kappa^{-1}\to H_\kappa^1}^p\biggr\} \lesssim_p [\log L]^{p}. 
\end{equation}
\end{corollary}

\begin{proof} Let $c_0$ be the constant appearing in Proposition~\ref{P:2}.  Given $\lambda\geq c_0^{-2}\log L$, we may choose $\kappa_1^2=a\lambda$ for some small $a>0$ (meant to defeat the implicit constants appearing in \eqref{P2eqn}) and apply Proposition~\ref{P:2} to deduce
\[
\PP\bigl\{ X_L(k)\geq \lambda\bigr\} \lesssim e^{-ac_0\lambda}
\]
uniformly in $L$ and over strictly admissible $k$. 

For $0<p<\infty$ we estimate as follows: 
\begin{align*}
\E\bigl\{|X_L(k)|^{\frac{p}{2}}\bigr\} & = \tfrac{p}{2}\int_0^\infty \lambda^{\frac{p}{2}} \PP\bigl\{|X_L(k)|\geq\lambda\bigr\}\tfrac{d\lambda}{\lambda} \\
& \lesssim_p \int_0^{c_0^{-2}\log L}\lambda^{\frac{p}{2}}\tfrac{d\lambda}{\lambda} + \int_{c_0^{-2}\log L}^\infty \lambda^{\frac{p}{2}}e^{-ac_0\lambda}\tfrac{d\lambda}{\lambda}  \lesssim_p \bigl[ \log L\bigr]^{\frac{p}{2}}.
\end{align*}
The result follows by choosing $c\leq c_0\min\{a, c_0\}$.  \end{proof}

\begin{proof}[Proof of Proposition~\ref{P:2}]  Our task is to estimate the random variables 
\[
X_L(k)=\bigl\|\sqrt{R_L(k)}\bigr\|_{L^2\to H_\kappa^1}^2 = \bigl\| \sqrt{R_L(k)}\bigr\|_{H_\kappa^{-1}\to L^2}^2=\bigl\|\sqrt{R_L^*(k)R_L(k)}\|_{H_{\kappa}^{-1}\to H_{\kappa}^1}.
\]

Given $\kappa_1>0$, let $\Omega=\Omega(\kappa_1)$ denote the event
\begin{equation}
\| \sqrt{R_0(k)}q_L\sqrt{R_0(k)}\|_{L^2\to L^2}^2 \leq \tfrac12\tfrac{\kappa_1}{|k|}\qtq{for all}k\in A(\kappa_1). 
\end{equation}
Note that the event $\Omega(\kappa_1)$ grows as $\kappa_1$ grows.

By Lemma~\ref{P:1} we have the estimate
\begin{align}\label{p.17*}
\PP(\Omega^c)\lesssim e^{-c\kappa_1^2}\qtq{provided}\kappa_1^2\geq 2c^{-1}\log L. 
\end{align}

In the following, we work on the set $\Omega$.  On this set and for $k\in A(\kappa_1)$, we may construct $R_L(k)$ via the series expansion
\[
R_L(k)=R_0(k) + \sum_{\ell\geq 1}(-1)^\ell\sqrt{R_0(k)}\left( \sqrt{R_0(k)}q_L\sqrt{R_0(k)}\right)^\ell \sqrt{R_0(k)}. 
\]
In particular, we have the estimate
\begin{equation}\label{RLR0}
\|R_L(k)-R_0(k)\|_{H_{\kappa}^{-1}\to H_{\kappa}^1} \lesssim \sqrt{\tfrac{\kappa_1}{\kappa}} \qtq{for all} k\in A(\kappa_1).
\end{equation}

We next observe that
\begin{equation}\label{st-equiv2}
X_L(k) \approx \biggl\| \int_0^\infty R_L^*(k_\tau)R_L(k_\tau)\,d\tau\biggr\|_{H_{\kappa}^{-1}\to H_{\kappa}^1},
\end{equation}
where
\[
k_\tau:=\sqrt{k^2+i\tau} \qtq{with}\tau\in(0,\infty). 
\]
Writing $k^2=E+i\sigma$, this is a consequence of the fact that
\[
\int_0^\infty \bigl| (\lambda+E)+i(\sigma+\tau)\bigr|^{-2}\,d\tau = \sqrt{\lambda +E -i\sigma} \, m(\lambda) \sqrt{\lambda +E + i\sigma} \qtq{with} m(\lambda)\approx 1
\]
along with the spectral theorem. 

We first consider the case $\kappa\geq \kappa_1$ and seek to prove the estimate
\begin{equation}\label{1106-1}
\biggl\|\int_0^\infty R_L^*(k_\tau)R_L(k_\tau)\,d\tau\biggr\|_{H_{\kappa}^{-1}\to H_{\kappa}^1} \lesssim 1.
\end{equation}
Note that in this case, $\kappa_\tau\in A(\kappa_1)$.

We begin with the identity
\begin{equation}\label{RLRLstar}
R_L^* R_L^{ }  = -(R_L^*-R_0^*)(R_L-R_0)+R_L^*(R_L^{ }-R_0)+(R_L^*-R_0^*)R_L + R_0^* R_0^{ }, 
\end{equation}
where each resolvent is evaluated at $k_\tau$.  The contribution to \eqref{1106-1} of the last term on the right-hand side of \eqref{RLRLstar} is straightforward to estimate: applying \eqref{st-equiv2} again, we get
\[
\biggl\| \int_0^\infty R_0^*(k_\tau)R_0(k_\tau)\,d\tau\biggr\|_{H_{\kappa}^{-1}\to H_\kappa^1} \lesssim \bigl\|\sqrt{R_0^*(k)R_0(k)}\bigr\|_{H_\kappa^{-1}\to H_\kappa^1} \lesssim 1,
\]
which is acceptable.

The contributions of the second and third terms on the right-hand side of \eqref{RLRLstar} can be handled similarly; we present here the details for the second term. We begin with the estimate
\begin{align*}
\bigl\| R_L^*(k_\tau)[R_L(k_\tau)&-R_0(k_\tau)]\bigr\|_{H_\kappa^{-1}\to H_\kappa^1} \\
&\lesssim \|1\|_{H_{|k_\tau|}^1\to H_{\kappa}^1}\|R_L^*(k_\tau)\|_{H_{|k_\tau|}^{-1}\to H_{|k_\tau|}^1}\|1\|_{H_{|k_\tau|}^1\to H_{|k_\tau|}^{-1}} \\
&\quad\times \|R_L(k_\tau)-R_0(k_\tau)\|_{H_{|k_\tau|}^{-1}\to H_{|k_\tau|}^1}\|1\|_{H_\kappa^{-1}\to H_{|k_\tau|}^{-1}}.
\end{align*}
To estimate these norms we first observe that $|k_\tau|\geq \kappa$, so that
\[
\|1\|_{H_\kappa^{-1}\to H_{|k_\tau|}^{-1}} + \|1\|_{H_{|k_\tau|}^1\to H_{\kappa}^1} \lesssim1.
\]
By \eqref{829-gain}, we also have the estimate
\begin{equation}\label{ktaugain}
 \|1\|_{H_{|k_\tau|}^1\to H_{|k_\tau|}^{-1}} \lesssim \tfrac{1}{\kappa^2+\tau}. 
\end{equation}
Thus, as $k_\tau\in A(\kappa_1)$, we may continue from above and use \eqref{RLR0} to deduce
\[
\bigl\| R_L^*(k_\tau)[R_L(k_\tau)-R_0(k_\tau)]\bigr\|_{H_\kappa^{-1}\to H_\kappa^1} \lesssim \frac{\sqrt{\kappa_1}}{(\kappa^2+\tau)^{\frac54}}.
\]
By the triangle inequality, this gives the acceptable contribution 
\[
\biggl\| \int_0^\infty R_L^*(k_\tau)[R_L(k_\tau)-R_0(k_\tau)]\,d\tau\biggr\|_{H_{\kappa}^{-1}\to H_{\kappa}^1} \lesssim\sqrt{\tfrac{\kappa_1}{\kappa}}.
\]

Finally, for the first term on the right-hand side of \eqref{RLRLstar}, an analogous argument yields the estimate
\[
\bigl\| [R_L^*(k_\tau)-R_0^*(k_\tau)][R_L(k_\tau)-R_0(k_\tau)]\bigr\|_{H_{\kappa}^{-1}\to H_{\kappa}^1} \lesssim \tfrac{\kappa_1}{(\kappa^2+\tau)^{\frac32}},
\]
whence
\[
\biggl\| \int_0^\infty [R_L^*(k_\tau)-R_0^*(k_\tau)][R_L(k_\tau)-R_0(k_\tau)]\,d\tau\biggr\|_{H_{\kappa}^{-1}\to H_{\kappa}^1} \lesssim \tfrac{\kappa_1}{\kappa},
\]
which is acceptable. 

It remains to consider the case $\kappa\leq \kappa_1$, where \eqref{P2eqn} will follow from
\begin{align}\label{1106-2}
\biggl\| \int_0^\infty R_L^*(k_\tau)R_L(k_\tau)\,d\tau\biggr\|_{H_{\kappa_1}^{-1}\to H_{\kappa_1}^1} \lesssim 1+\tfrac{\kappa_1^2}{\sigma}.
\end{align}
We will split the integral in $\tau$ into the regions $[0,\kappa_1^2]$ and $[\kappa_1^2,\infty)$.  

We begin with the region $[\kappa_1^2,\infty)$.  Defining 
\[
\tilde k = \sqrt{k^2+i\kappa_1^2},\quad \tilde\kappa=|\tilde k|,\qtq{and} \tilde k_\tau=\sqrt{k^2+i\kappa_1^2+i\tau}
\]
and observing that $\kappa_1\leq \tilde \kappa$, we can change variables and bound
\[
\biggl\| \int_{\kappa_1^2}^\infty R_L^*(k_\tau)R_L(k_\tau)\,d\tau\biggr\|_{H_{\kappa_1}^{-1}\to H_{\kappa_1}^1} \leq \biggl\| \int_0^\infty R_L^*(\tilde k_\tau)R_L(\tilde k_\tau)\,d\tau\biggr\|_{H_{\tilde \kappa}^{-1}\to H_{\tilde \kappa}^1}. 
\]
As $\tilde k_\tau\in A(\kappa_1)$ for $\tau\in(0,\infty)$,  we are in a position to estimate the term on the right-hand side exactly as we did for \eqref{1106-1}. In particular, we deduce
\[
\biggl\| \int_{\kappa_1^2}^\infty R_L^*(k_\tau)R_L(k_\tau)\,d\tau\biggr\|_{H_{\kappa_1}^{-1}\to H_{\kappa_1}^1} \lesssim 1. 
\]

We turn to the region $[0,\kappa_1^2]$. Using the resolvent identity 
\[
R_L(k_{\tau})=R_L(\kappa_1)-(k_{\tau}^2-\kappa_1^2)\sqrt{R_L(\kappa_1)}R_L(k_\tau)\sqrt{R_L(\kappa_1)},
\]
we find that it suffices to estimate the following three terms:
\begin{align}
&\int_0^{\kappa_1^2} \|R_L^*(\kappa_1)R_L(\kappa_1)\|_{H_{\kappa_1}^{-1} \to H_{\kappa_1}^1}\,d\tau, \label{829-1}\\
&\int_0^{\kappa_1^2}|k_\tau^2-\kappa_1^2|\|R_L^*(\kappa_1)\sqrt{R_L(\kappa_1)}R_L(k_\tau)\sqrt{R_L(\kappa_1)}\|_{H_{\kappa_1}^{-1}\to H_{\kappa_1}^1}\,d\tau, \label{829-2} \\
&\int_0^{\kappa_1^2} |k_\tau^2-\kappa_1^2|^2\| \sqrt{R_L^*(\kappa_1)}R_L^*(k_\tau) \sqrt{R_L^*(\kappa_1)}\sqrt{R_L(\kappa_1)} R_L(k_\tau)\sqrt{R_L(\kappa_1)}\|_{ H_{\kappa_1}^{-1}\to H_{\kappa_1}^1}\,d\tau. \label{829-3}
\end{align}
 
For \eqref{829-1}, we observe that \eqref{RLR0} implies 
\[
\|R_L(\kappa_1)\|_{H_{\kappa_1}^{-1}\to H_{\kappa_1}^1} + \|R_L^*(\kappa_1)\|_{H_{\kappa_1}^{-1}\to H_{\kappa_1}^1} \lesssim 1.
\] 
Combining this with \eqref{829-gain}, we estimate $\eqref{829-1} \lesssim 1,$  which is acceptable.
 
Next, we have
\begin{align*}
\eqref{829-2}  \lesssim \int_0^{\kappa_1^2}& \kappa_1^2\bigl\{ \|R_L^*(\kappa_1)\|_{H_{\kappa_1}^{-1}\to H_{\kappa_1}^1} \|1\|_{H_{\kappa_1}^1\to H_{\kappa_1}^{-1}} \|\sqrt{R_L(\kappa_1)}\|_{L^2\to H_{\kappa_1}^1} \\
& \quad \times \|R_L(k_\tau)\|_{L^2\to L^2} \|\sqrt{R_L(\kappa_1)}\|_{H_{\kappa_1}^{-1}\to L^2}\bigr\}\,d\tau.
\end{align*}
Using \eqref{RL2L2} with  $\Im(k_\tau^2)= \tau+\sigma$, \eqref{829-gain}, and the fact that \eqref{P2eqn} is already proved for $\kappa=\kappa_1$, we get
\[
\eqref{829-2} \lesssim \int_0^{\kappa_1^2} \tfrac{d\tau}{\tau+\sigma} \lesssim \tfrac{\kappa_1^2}{\sigma},
\]
which is an acceptable contribution to \eqref{1106-2}. 

Finally we turn to \eqref{829-3}.  Estimating in the same way,
\begin{align*}
\eqref{829-3} &\lesssim \int_0^{\kappa_1^2}  \kappa_1^4\big\{ \|\sqrt{R_L^*(\kappa_1)}\|_{L^2\to H_{\kappa_1}^1} \|R_L^*(k_\tau)\|_{L^2\to L^2} \|\sqrt{R_L^*(\kappa_1)}\|_{H_{\kappa_1}^{-1}\to L^2} \\
& \quad\quad\quad \times \|1\|_{H_{\kappa_1}^{1}\to H_{\kappa_1}^{-1}}\|\sqrt{R_L(\kappa_1)}\|_{L^2\to H_{\kappa_1}^1}\|R_L(k_\tau)\|_{L^2\to L^2} \\
& \quad\quad\quad \times \|\sqrt{R_L(\kappa_1)}\|_{H_{\kappa_1}^{-1}\to L^2}\bigr\}\,d\tau \\
& \lesssim \int_0^{\kappa_1^2} \tfrac{\kappa_1^2}{(\tau+\sigma)^2}\,d\tau \lesssim \tfrac{\kappa_1^2}{\sigma},
\end{align*} 
which is acceptable. 

This completes the proof of \eqref{1106-2} and hence the proof of Proposition~\ref{P:2}.\end{proof}

The remaining three propositions in this section are devoted to exhibiting decay of the resolvent away from the diagonal.   We begin by running the Combes--Thomas argument with a generic weight $e^\rho$ for $\rho$ belonging to the class
\begin{equation}\label{def:W}
W(\lambda):=\bigl\{\rho:\R\to\R\qtq{such that} \|\rho'\|_{L^\infty}\leq \lambda^{-\frac12}\bigr\}
\end{equation}
for $\lambda>0$.  We then employ certain concrete choices of weight to obtain our main results, namely, Propositions~\ref{P:4} and~\ref{P:alpha}.  Recall that we employed the Combes--Thomas argument earlier when proving Lemma~\ref{P:11}.

\begin{proposition}\label{P:3} There exists $c>0$ such that
\begin{equation}\label{E:P:3}
\PP\bigl\{ \| e^\rho R_L(k) e^{-\rho}\|_{H_\kappa^{-1}\to H_\kappa^1} \geq \lambda\bigr\} \lesssim e^{-c\lambda} \qtq{for all} \lambda\geq c^{-1} \log L,
\end{equation}
uniformly over finite $L$ satisfying $2 \leq L \leq L_0\in 2^{\mathbb N_0}\cup\{\infty\}$, strictly admissible $k$, and $\rho\in W(\lambda)$. 
\end{proposition}

\begin{proof}  Let $\kappa_1 \gg \sqrt{\log L}$ and take $\Omega=\Omega(\kappa_1)$ as in Proposition~\ref{P:2}.  In view of \eqref{p.17*}, it suffices to show that for $\rho\in W(\lambda)$ we have
\[
\| e^\rho R_L(k) e^{-\rho}\|_{H_\kappa^{-1}\to H_\kappa^1} \leq \lambda \quad \text{on $\Omega$,}
\]
provided $\lambda$ and $\kappa_1$ are chosen appropriately. 

Given $\rho\in W(\lambda)$, we may write
\begin{equation}\label{erhoinv}
e^{\rho(x)}\bigl(-\partial_x^2+q_L+k^2\bigr)e^{-\rho(x)}=-\partial_x^2+q_L+k^2+B_\rho,
\end{equation}
where $B_\rho$ is the differential operator 
\[
B_\rho:=\rho'(x)\partial_x + \partial_x\rho'(x)-\rho'(x)^2.
\]
Note that
\[
e^{\rho}R_L(k)e^{-\rho} = \sqrt{R_L(k)}\bigl(1+\sqrt{R_L(k)}B_\rho \sqrt{R_L(k)}\bigr)^{-1}\sqrt{R_L(k)}.
\]

By duality, one can readily show that
\[
\|B_\rho\|_{H_\zeta^{1}\to H_{\zeta}^{-1}}\leq 2\zeta^{-1}\|\rho'\|_{L^\infty} + \zeta^{-2}\|\rho'\|_{L^\infty}^2
\]
for any $\zeta>0$.  Thus for $\rho\in W(\lambda)$ and $\zeta=\max\{\kappa,\kappa_1\}$, Proposition~\ref{P:2} yields
\begin{align*}
\| \sqrt{R_L(k)}B_\rho \sqrt{R_L(k)}\|_{L^2\to L^2} 
&\leq \|\sqrt{R_L(k)}\|_{H_{\zeta}^{-1}\to L^2} \|B_\rho\|_{H_{\zeta}^1\to H_{\zeta}^{-1}}\|\sqrt{R_L(k)}\|_{L^2\to H_{\zeta}^1}\\
& \lesssim \bigl(1+\tfrac{\kappa_1^2}{\sigma}\bigr)\bigl(\tfrac{1}{\sqrt{\lambda}\kappa_1}+\tfrac{1}{\lambda\kappa_1^2}\bigr) \quad\text{on $\Omega$}.
\end{align*}
Hence, choosing $\kappa_1=a\sqrt{\lambda}$ for a small $a>0$, which is consistent with $\kappa_1\gg \sqrt{\log L}$ if we choose $c$ appearing in \eqref{E:P:3} sufficiently small, we deduce
\[
\| \sqrt{R_L(k)}B_\rho \sqrt{R_L(k)}\|_{L^2\to L^2}\leq \tfrac12
\]
on $\Omega$ for all strictly admissible $k$ and for all $\rho\in W(\lambda)$. Thus, with $\zeta=\max\{\kappa,\kappa_1\}$,
\begin{align*}
\| e^{\rho}R_L(k) e^{-\rho}\|_{H_{\kappa}^{-1}\to H_{\kappa}^1} \lesssim \|\sqrt{R_L(k)}\|_{L^2\to H_{\zeta}^1} \|\sqrt{R_L(k)}\|_{H_{\zeta}^{-1}\to L^2}\lesssim 1+ \tfrac{\kappa_1^2}\sigma \lesssim 1+a^2 \lambda 
\end{align*}
on $\Omega$.  The result follows by choosing $a$ and $c$ suitably small. \end{proof}

To continue, let $\chi_1\in C_c^\infty(\R)$ be a bump function satisfying
\[
1_{[-1,1]}\leq \chi_1 \leq 1_{[-2,2]}.
\]
For $n\in 2^{\mathbb{N}}$ we define
\begin{align}\label{chin}
\chi_n(x) := \chi_1(\tfrac{x}{n})-\chi_1(\tfrac{2x}{n}),
\end{align}
which yields the smooth partition of unity
\[
\sum_{n\in 2^{\mathbb{N}_0}} \chi_n \equiv 1.
\]
Note that for $n>1$, $\chi_n$ is supported in $\{\tfrac{n}{2}\leq |x|\leq 2n\}$. In particular, $\chi_n\chi_m\equiv 0$ for $1\leq n\leq \frac m4$. 

By construction we have
\[
\|\chi_n'\|_{L^\infty}\lesssim 1\quad \text{uniformly in $n\in 2^{\mathbb{N}_0}$},
\]
which by \eqref{multiplier1} implies
\begin{equation}\label{multiplier2}
\|\chi_n\|_{H_\kappa^1\to H_\kappa^1} = \|\chi_n\|_{H_\kappa^{-1}\to H_\kappa^{-1}}\lesssim 1\quad \text{uniformly in $n\in 2^{\mathbb{N}_0}$ and $\kappa\geq 1$}. 
\end{equation}

\begin{proposition}\label{P:4} Fix $0<\alpha<\frac12$.  For $1\leq p<\infty$, we have
\begin{equation*}
\E\bigl\{ \bigl\| \chi_m R_L(k)\chi_n\bigr\|_{H_{\kappa}^{-1}\to H_{\kappa}^1}^p \bigr\} = \E\bigl\{\bigl\|\chi_n R_L(k)\chi_m \bigr\|_{H_{\kappa}^{-1}\to H_{\kappa}^1}^p\bigr\} \lesssim_{p,\alpha} [\log(L)]^p e^{-8p\langle m\rangle^{\alpha}}
\end{equation*} 
uniformly for $n,m\in 2^{\mathbb{N}_0}$ with $n\leq \tfrac18 m$, finite $L$ satisfying $2 \leq L \leq L_0\in 2^{\mathbb N} \cup \{\infty\}$, and strictly admissible $k$.
\end{proposition}

\begin{proof} The two operators under consideration have the same norm since they are transposes of one another.  Thus we only consider $\chi_m R_L(k)\chi_n$ in what follows.

Fix $n\leq \tfrac18m$ and strictly admissible $k$.  We define the weight $\rho$ via
\[
\rho(x)=\begin{cases} 20\langle x\rangle^{\alpha} & |x| \geq \tfrac{m}{2} \\ 0 & |x|\leq 2n \end{cases},
\]
and by linear interpolation in the remaining intervals.  By construction,
\[
\|\rho'\|_{L^\infty} \lesssim \langle m\rangle^{\alpha-1},
\]
and thus $\rho\in W(\lambda)$ with $\lambda\approx \langle m\rangle^{2(1-\alpha)}$ (cf. \eqref{def:W}).  By \eqref{multiplier1},  
\[
\|e^{-\rho}\chi_m\|_{H_{\kappa}^1\to H_{\kappa}^1}\lesssim e^{-10\langle m\rangle^{\alpha}}\qtq{and} \|e^{\rho}\chi_n\|_{H_{\kappa}^{-1}\to H_{\kappa}^{-1}} \lesssim 1. 
\]

We define $\Omega_{lo}$ to be the event
\begin{align*}
\|e^{\rho}R_L(k)e^{-\rho}\|_{H_\kappa^{-1}\to H_{\kappa}^1}&\leq \langle m\rangle^{2(1-\alpha)}+ c^{-1} \log L,
\end{align*}
where $c$ is as in Proposition~\ref{P:3}.

We estimate
\begin{align*}
\E\bigl\{ 1_{\Omega_{lo}}\|\chi_m R_L \chi_n \|_{H_{\kappa}^{-1}\to H_\kappa^1}^p\bigr\}
& \lesssim_p \|e^{-\rho}\chi_m\|_{H_{\kappa}^1\to H_{\kappa}^1}^p \|e^{\rho}\chi_n\|_{H_{\kappa}^{-1}\to H_{\kappa}^{-1}}^p [\langle m\rangle^{2(1-\alpha)}+ \log L ]^p \\
&\lesssim_{p,\alpha} [\log L]^p e^{-8p\langle m\rangle^\alpha},
\end{align*}
which is an acceptable contribution.

Next, using Proposition~\ref{P:3}, \eqref{RLk-itself}, and \eqref{multiplier2}, we estimate
\begin{align*}
\E&\bigl\{1_{\Omega_{lo}^c}\|\chi_m R_L \chi_n\|_{H_{\kappa}^{-1}\to H_{\kappa}^1}^p \bigr\}  \lesssim  [\PP\bigl(\Omega_{lo}^c\bigr)]^{\frac12}\sqrt{\E\bigl\{\|\chi_mR_L\chi_n\|_{H_{\kappa}^{-1}\to H_{\kappa}^1}^{2p}\bigr\}} \\
& \lesssim e^{-\frac{c}{2}\langle m\rangle^{2(1-\alpha)}}\sqrt{\E\bigl\{\|R_L\|_{H_{\kappa}^{-1}\to H_{\kappa}^1}^{2p}\bigr\}} \lesssim_{p,\alpha} e^{-8p\langle m\rangle^{\alpha}}[\log L]^p,
\end{align*}
which is also acceptable.  This completes the proof of Proposition~\ref{P:4}. \end{proof}

At this point a word of explanation is warranted regarding the sub-exponential decay appearing in Proposition~\ref{P:4} and recurrently through the remainder of the paper.  First, and most important, is the fact that this type of decay suffices to prove all the results we need.  On the other hand, polynomial weights do not suffice.  This is most evident in the dynamical arguments of Section~\ref{S:L} where Gronwall arguments need to be applied.  In this setting, unavoidable logarithmic losses in the volume (cf. Lemma~\ref{L:sup}) raised to a power greater than one must be exponentiated. Thus one needs decay that is stronger than polynomial in order to defeat them.

On the other hand, pure exponential decay of the resolvent would be a much more difficult goal to pursue, and ultimately unnecessary.  In particular, when concatenating exponential decay from one cube to a second, and then on to a third (which is an essential technique in multi-scale analysis), one has to be extremely vigilant to avoid continually accumulating losses in the rate of decay.  The strict subadditivity of fractional powers, makes this a breeze in the  sub-exponential regime.   

\begin{proposition}\label{P:alpha} Fix $0<\alpha<\tfrac12$, $0<\beta\leq 4$, and $1\leq p<\infty$.  Then
\[
\E\bigl\{ \|e^{\pm\beta\langle x\rangle^\alpha}\chi_m R_L(k) \chi_n e^{\mp\beta\langle x\rangle^{\alpha}}\|_{H_{\kappa}^{-1}\to H_{\kappa}^1}^p \bigr\} \lesssim_{p,\alpha, \beta} [\log L]^p,
\]
uniformly in finite $L$ satisfying $2 \leq L \leq L_0\in 2^{\mathbb N}\cup \{\infty\}$, $m,n\in 2^{\mathbb{N}}$ with $m,n\gg [\log L]^{\frac{1}{2(1-\alpha)}}$, and strictly admissible $k$. 
\end{proposition}

\begin{proof} By transpose symmetry, it suffices to consider only the upper signs.  Moreover, the preceding proposition allows us to restrict our attention to the case $\frac n4\leq m\leq 4n$.  We define
\[
\rho(x) = \beta\langle x\rangle^\alpha \qtq{for}\tfrac{n}{8}\leq |x|\leq 8n
\]
with $\rho$ continuous and constant on each remaining interval. Note that
\[
\|\rho'\|_{L^\infty} \lesssim n^{\alpha-1}.
\]
In particular, $\rho\in W(\lambda)$ for $\lambda \approx n^{2(1-\alpha)}$.  

Let $\Omega$ denote the event
\[
\|e^{\rho} R_L(k) e^{-\rho}\|_{H_\kappa^{-1}\to H_\kappa^1}\leq n^{2(1-\alpha)}.
\]
The restriction $n\gg [\log L]^{\frac{1}{2(1-\alpha)}}$ allows us to apply Proposition~\ref{P:3} and obtain
\[
\PP\{\Omega^c\} \lesssim e^{-cn^{2(1-\alpha)}}. 
\] 
Using Proposition~\ref{P:3} and \eqref{multiplier2}, we now estimate
\begin{align*}
\E\bigl\{&1_{\Omega} \|e^{\beta\langle x\rangle^\alpha}\chi_m R_L(k) \chi_n e^{-\beta\langle x\rangle^\alpha}\|_{H_\kappa^{-1}\to H_\kappa^1}^p\bigr\}\\
&\lesssim \E\bigl\{1_{\Omega}\|e^{\rho}R_L(k)e^{-\rho}\|_{H_\kappa^{-1}\to H_\kappa^1}^p \bigr\} \lesssim_p\ \int_0^{c^{-1}\log L}\mu^{p-1}d\mu + \int_{c^{-1}\log L}^{n^{2(1-\alpha)}} \mu^{p-1}e^{-c\mu}\,d\mu \\
&\phantom{ \lesssim \E\bigl\{1_{\Omega}\|e^{\rho}R_L(k)e^{-\rho}\|_{H_\kappa^{-1}\to H_\kappa^1}^p \bigr\}\,} \lesssim_p [\log L]^p.
\end{align*}
Applying Cauchy--Schwarz and using Corollary~\ref{C:2} as well, we estimate
\begin{align*}
\E\bigl\{ 1_{\Omega^c} \|e^{\beta\langle x\rangle^\alpha}\chi_m &  R_L(k) \chi_n e^{-\beta\langle x\rangle^\alpha}\|_{H_\kappa^{-1}\to H_\kappa^1}^p\bigr\} \\
& \lesssim [\PP\{\Omega^c\}]^{\frac12}e^{p\beta\langle 2m\rangle^\alpha} \bigl[\E\{\| R_L(k)\|_{H_\kappa^{-1}\to H_\kappa^1}^{2p}\bigr\}\bigr]^{\frac12}  \lesssim_{p,\alpha,\beta} [\log L]^p.
\end{align*}
This completes the proof.  \end{proof}

\section{Multiscale analysis}\label{S:multiscale}

The key idea of the multiscale analysis is to tame bad behavior in $q$ by exploiting decay of the resolvent away from the diagonal.  Note that ergodicity of white noise under translation on the line guarantees that such bad behavior is inevitable.  Our approach to multiscale analysis is to successively reveal the potential on dyadic shells using the resolvent identity; archetypal examples appear in \eqref{RL-reveal} and \eqref{dyadic-reveal-2} below.  In both cases, the sum in $L$ will be controlled using decay of the resolvent in the form exhibited in Proposition~\ref{P:4}.

The ultimate goal of the multiscale analysis is to obtain bounds independent of the volume $L_0$, such as those appearing in Proposition~\ref{P:9'}.   These are evidently essential if we are to send $L_0\to\infty$, which is what will be done in  Section~\ref{S:L}.  However, our first application of this technique is Proposition~\ref{P:8}, which removes the limitations on $m$ and $n$ in the formulation of Proposition~\ref{P:alpha}.  In this setting, a logarithmic loss in $L$ is unavoidable, as it originates from contributions near the diagonal.

\begin{proposition}\label{P:8}
Fix $0<\alpha<\tfrac12$ and let $\varphi_L$ be a smooth cutoff satisfying
\begin{align}\label{phiL}
1_{[-4L,4L]}\leq \varphi_L \leq 1_{[-8L,8L]}.
\end{align}
For $1\leq p<\infty$ and $0<\beta\leq 4$, we have
\[
\E\bigl\{ \|e^{\beta \langle x\rangle^{\alpha}}R_{L}(k)\varphi_Le^{- \beta\langle x\rangle^{\alpha}}\|_{H_{\kappa}^{-1}\to H_{\kappa}^1}^p\bigr\} \lesssim_p [\log L]^{2p},
\]
uniformly for  finite $L$ satisfying $2\leq L \leq L_0\in 2^{\mathbb N}\cup\{\infty\}$ and strictly admissible $k$.
\end{proposition}

\begin{proof} Recall the cutoffs $\chi_n$ introduced in \eqref{chin}. We begin by using the triangle inequality to estimate
\begin{align*}
\bigl(&\E\{\|e^{\beta \langle x\rangle^\alpha}R_{L}(k) \varphi_Le^{-\beta\langle x\rangle^\alpha}\|_{H_\kappa^{-1}\to H_{\kappa}^1}^p\}\bigr)^{\frac{1}{p}} \\
& \lesssim \sum_{m\geq 1} \sum_{n=1}^{8L} \bigl( \E\{\|e^{\beta\langle x\rangle^\alpha}\chi_m R_{L}(k)\chi_n e^{- \beta \langle x\rangle^\alpha}\|_{H_\kappa^{-1}\to H_{\kappa}^1}^p \}\bigr)^{\frac{1}{p}} \\
& \lesssim \log L \sup_{1\leq n\leq 8L} \sum_{m\geq 1} \bigl(\E\{\|e^{\beta\langle x\rangle^\alpha}\chi_m R_{L}(k)\chi_n e^{-\beta \langle x\rangle^\alpha}\|_{H_\kappa^{-1}\to H_{\kappa}^1}^p \}\bigr)^{\frac{1}{p}}. 
\end{align*}

Using \eqref{multiplier1} and Proposition~\ref{P:4}, we can estimate the contributions of $m<\frac n4$ and $m>4n$ by
\begin{align*}
\lesssim_p \log L \sup_{1\leq n\leq 8L} \Bigl[ \sum_{m<\frac n4}  e^{ \beta\langle 2m\rangle^\alpha}e^{-8\langle n\rangle^\alpha} \log L  + \sum_{m> 4n}  e^{ \beta\langle 2m\rangle^\alpha}e^{-8\langle m\rangle^\alpha} \log L\Bigr] \lesssim_p [\log L]^2,
\end{align*}
which is acceptable. 

It remains to estimate the contribution near the diagonal, namely when $\frac n4\leq m\leq 4n\leq 32 L$.  We distinguish two cases; the case $L<64n$ will be discussed at the end of the proof.

If $L\geq 64n$, we use the decomposition
\begin{align}
R_{L}(k) &= R_{32n}(k) + \sum_{\ell=32n}^{L/2} [R_{2\ell}(k)-R_\ell(k)] \notag \\
&= R_{32n}(k) - \sum_{\ell=32n}^{L/2} R_{2\ell}(k)(q_{2\ell}-q_\ell)R_\ell(k). \label{RL-reveal}
\end{align}

We first consider the contribution of $R_{32n}(k)$.  If $n\gg 1$ so that $n\gg (\log n)^{\frac{1}{2(1-\alpha)}}$, we use Proposition~\ref{P:alpha} to estimate
\[
\E\{\| e^{\beta\langle x\rangle^\alpha} \chi_m R_{32n}(k) \chi_n e^{-\beta\langle x\rangle^\alpha}\|_{H_\kappa^{-1}\to H_{\kappa}^1}^p\} \lesssim_p [\log n]^p \lesssim_p [\log L]^p, 
\]
which is acceptable.  In the remaining case $n\lesssim 1$, we have by \eqref{multiplier1} and \eqref{RLk-itself} that
\[
\E\{\| e^{\beta\langle x\rangle^\alpha} \chi_m R_{32n}(k)\chi_n e^{-\beta\langle x\rangle^\alpha}\|_{H_{\kappa}^{-1}\to H_{\kappa}^1}^p\} \lesssim_p 1,
\]
which is acceptable, as well. 

We treat an individual summand in \eqref{RL-reveal} using \eqref{multiplier1}:
\begin{align*}
\E&\bigl\{ \| e^{ \beta\langle x\rangle^\alpha}\chi_m R_{2\ell}(k)\tilde\chi_\ell(q_{2\ell}-q_\ell)\tilde\chi_\ell R_\ell(k)\chi_n e^{- \beta\langle x\rangle^\alpha}\|_{H_\kappa^{-1}\to H_\kappa^1}^p\bigr\} \\
& \lesssim e^{p\beta\langle 2m\rangle^\alpha}\E\{\|\chi_m R_{2\ell}(k)\tilde\chi_{\ell}\|_{H_{\kappa}^{-1}\to H_{\kappa}^1}^p
	\|\tilde\chi_\ell R_\ell(k) \chi_n\|_{H_{\kappa}^{-1}\to H_\kappa^1}^p \|q_{2\ell}-q_\ell\|_{H_{\kappa}^{1}\to H_{\kappa}^{-1}}^p \},
\end{align*}
where we have employed the following cutoff to the support of $q_{2\ell}-q_\ell$:
\begin{equation}\label{chitilde}
\begin{aligned}
\tilde\chi_\ell(x):=
\begin{cases} \sum_{h\in \Z}\bigl(\chi_\ell+\chi_{2\ell}\bigr)(x-2hL_0) &\quad \text{if $L_0\in 2^{\mathbb N}$,}\\
\bigl(\chi_\ell+\chi_{2\ell}\bigr)(x) &\quad \text{if $L_0=\infty$}.
\end{cases}
\end{aligned}
\end{equation}
Recalling that $\ell\geq 32n\geq 8m$ and using H\"older's inequality, \eqref{E:C:q op}, and Proposition~\ref{P:4}, for $L_0\in 2^{\mathbb N}$ we estimate the above by
\begin{align*}
\lesssim_p e^{p\beta \langle 2m\rangle^\alpha}\Bigl(\sum_{h\in \Z} e^{-8p\langle \ell+2|h|L_0\rangle^\alpha}[\log(\ell+2|h|L_0)]^p\Bigr)^2 [\log \ell]^{p/2}  \lesssim_p e^{-8p\langle \ell\rangle^\alpha},
\end{align*}
which sums in $\ell$ with a final bound of order $1$.  If $L_0=\infty$, a similar argument yields the bound
$$
\lesssim_p e^{p\beta \langle 2m\rangle^\alpha}\bigl[e^{-8p\langle \ell\rangle^\alpha}[\log(\ell)]^p\bigr]^2 [\log \ell]^{p/2}  \lesssim_p e^{-8p\langle \ell\rangle^\alpha},
$$
which also sums in $\ell$ with a final bound of order $1$.

It remains to consider the contribution near the diagonal in the case $L<64n$.  This can be handled using \eqref{RLk-itself} and Proposition~\ref{P:alpha} in the same way as when controlling the contribution of $R_{32n}$ above.
\end{proof} 

\begin{proposition}\label{P:9'} For $1\leq p<\infty$, we have the following:
\[
\E\bigl\{ \|\chi_m R_{L_0}(k) \chi_n\|_{H_\kappa^{-1}\to H_\kappa^1}^p\bigr\} \lesssim_p \begin{cases} [\log m]^p & \text{if }m\approx n \\ e^{-6p\langle m\rangle^\alpha} & \text{if }m>4n \\
e^{-6p\langle n\rangle^\alpha} & \text{if }n>4m,\end{cases}
\]
uniformly over $m,n\in 2^{\mathbb{N}}$, $2\leq L_0\in 2^{\mathbb{N}}\cup \{\infty\}$, and strictly admissible $k$.
\end{proposition}

\begin{proof} By transpose symmetry, it suffices to consider only the cases $m\approx n$ or $m>4n$.  At the outset, we assume that $16m\leq L_0$.  The remaining case will be discussed at the end of the proof.  As in \eqref{RL-reveal}, we write
\begin{equation}\label{dyadic-reveal-2}
R_{L_0}(k) =  R_{8m}(k) \ -\sum_{8m\leq L\leq L_0/2} R_{2L}(k)[q_{2L}-q_L]R_L(k).
\end{equation}

If $m\approx n$ we use \eqref{RLk-itself} to bound
\[
\E\bigl\{\|\chi_m R_{8m}(k)\chi_n\|_{H_{\kappa}^{-1}\to H_\kappa^1}^p\bigr\} \lesssim_p [\log m]^p, 
\]
which is acceptable.  If instead $m> 4n$, then Proposition~\ref{P:4} yields
\[
\E\bigl\{ \|\chi_m R_{8m}(k) \chi_n\|_{H_\kappa^{-1}\to H_\kappa^1}^p\bigr\} \lesssim_p [\log m]^pe^{-8p\langle m\rangle^\alpha},
\]
which is also acceptable.  

For $L\geq 8m$, we use Proposition~\ref{P:4}, \eqref{E:C:q op}, \eqref{RLk-itself}, and H\"older's inequality to estimate 
\begin{align*}
\E\bigl\{&  \| \chi_m R_{2L}(k)[q_{2L}-q_L]R_L(k)\chi_n\|_{H_\kappa^{-1}\to H_\kappa^1}^p \bigr\} \\
& \lesssim \E\bigl\{  \| \chi_m R_{2L}(k)[q_{2L}-q_L]\tilde\chi_L^2 R_L(k)\chi_n\|_{H_\kappa^{-1}\to H_\kappa^1}^p \bigr\} \\
& \lesssim \E\bigl\{ \|\chi_m R_{2L}(k)\tilde\chi_L\|_{H_\kappa^{-1}\to H_\kappa^1}^p\|\tilde\chi_L R_L(k)\chi_n\|_{H_\kappa^{-1}\to H_\kappa^1}^p
	\|q_{2L}-q_L\|_{H_\kappa^1\to H_\kappa^{-1}}^p \bigr\} \\
& \lesssim_p [\log L]^{p/2}\Bigl(\sum_{h\in\Z} e^{-8p\langle L+2|h|L_0\rangle^\alpha}[\log(L+2|h|L_0)]^p\Bigr)^2  \lesssim_p e^{-8p\langle L\rangle^\alpha},
\end{align*}
where $\tilde\chi_L$ is as in \eqref{chitilde}.  Summing over $L\geq 8m$ yields an acceptable contribution.

It remains to consider the case $16m>L_0$.  In this case, the claim follows readily from \eqref{RLk-itself} and Proposition~\ref{P:4}, in the same way that we treated the contribution of $R_{8m}(k)$ above. 
\end{proof}

Proposition~\ref{P:9'} leads to several useful estimates, which we record here. 

\begin{corollary}\label{C:9} Fix $1\leq p<\infty$.  The following holds uniformly over $L_0\in 2^{\mathbb N}\cup\{\infty\}$ and strictly admissible $k$:
\begin{gather}
\E\bigl\{\|\chi_m R_{L_0}(k)\|_{H_{\kappa}^{-1}\to H_\kappa^1}^p\bigr\} \lesssim_p [\log m]^p \quad\text{uniformly for $m\in 2^{\mathbb{N}}$,} \label{E:9-1}\\
\E\bigl\{ \| R_{L_0}(k)\langle x\rangle^{-\beta}\|_{H_\kappa^{-1}\to H_\kappa^1}^p\bigr\}  \lesssim_{\beta,p} 1 \quad\text{for any $\beta>0$,} \label{E:9-2}\\
\E\bigl\{ \| e^{\langle x\rangle^\alpha} R_{L_0}(k)e^{-2\langle x\rangle^\alpha}\|_{H_\kappa^{-1}\to H_\kappa^1}^p\bigr\}  \lesssim_{\alpha,p} 1 \quad\text{for any $0<\alpha<\tfrac14$,}\label{E:9-3}\\
\E\bigl\{ \| e^{2\langle x\rangle^\alpha}\varphi_L R_{L_0}(k)e^{-2\langle x\rangle^\alpha} \|_{H_\kappa^{-1}\to H_\kappa^1}^p\bigr\} \lesssim_{\alpha, p} [\log L]^{p} \quad\text{for any $0<\alpha<\tfrac12$,}\label{E:9-4}
\end{gather}
where the implicit constant in \eqref{E:9-4} is uniform over $L\in 2^{\mathbb N}$ and $\varphi_L$ is as in \eqref{phiL}. 
\end{corollary}

\begin{proof} The estimate \eqref{E:9-1} follows from Proposition~\ref{P:9'} by writing $1=\sum_n \chi_n$ and then splitting into the regimes $n<\tfrac14m$, $\tfrac14m\leq n\leq 4m$, and $n>4m$. 

The estimate \eqref{E:9-2} follows from \eqref{E:9-1}.  Indeed, writing $1=\sum_m \chi_m$, we estimate
\begin{align*}
\bigl[\E\{\|R_{L_0}(k)\langle x\rangle^{-\beta}\|_{H_\kappa^{-1}\to H_\kappa^1}^p\bigr\}\bigr]^{\frac1p}
\lesssim \sum_{m\in 2^{\mathbb{N}_0}}\langle m\rangle^{-\beta} \bigl[\E\{\|R_{L_0}(k)\langle x\rangle^{-\beta}\langle m\rangle^\beta \chi_m\|_{H_\kappa^{-1}\to H_\kappa^1}^p\bigr\}\bigr]^{\frac{1}{p}}.
\end{align*}
Observing that $\langle x\rangle^{-\beta}\langle m\rangle^\beta\chi_m$ defines a bounded operator on $H_\kappa^{\pm1}$, the result follows from \eqref{E:9-1} and the fact that  $\langle m\rangle^{-\beta}\log m$ is summable over $m\in 2^{\mathbb{N}}$.

To obtain \eqref{E:9-3} we insert $1=\sum_n \chi_n=\sum_m \chi_m$ and use that multiplication by
\[
e^{c\langle x\rangle^\alpha}e^{- c\langle 2\ell\rangle^\alpha}\chi_\ell(x) \qtq{and} e^{-c\langle x\rangle^\alpha}e^{c\langle \ell/2\rangle^\alpha}\chi_\ell(x)
\]
are bounded operators on $H_\kappa^{\pm 1}$ uniformly for $\ell\in 2^{\mathbb{N}_0}$.  Thus, invoking Proposition~\ref{P:9'} we get
\begin{align*}
\bigl[\text{LHS\eqref{E:9-3}}\bigr]^{\frac1p}
&\lesssim \sum_{n,m\in 2^{\mathbb{N}_0}} e^{\langle 2m\rangle^\alpha} e^{-2\langle n/2\rangle^\alpha}\bigl[\E\bigl\{ \| \chi_m R_{L_0}(k) \chi_n\|_{H_\kappa^{-1}\to H_\kappa^1}^p\bigr\}\bigr]^{\frac{1}{p}}\\
& \lesssim_p \sum_{m>4n}e^{\langle 2m\rangle^\alpha} e^{-2\langle n/2\rangle^\alpha}e^{-6\langle m\rangle^\alpha} + \sum_{n>4m}e^{\langle 2m\rangle^\alpha} e^{-2\langle n/2\rangle^\alpha}e^{-6\langle n\rangle^\alpha} \\
&\quad +\sum_{\frac n4\leq m\leq 4n}e^{\langle 2m\rangle^\alpha} e^{-2\langle n/2\rangle^\alpha}\log(m)\lesssim_p1,
\end{align*}
where we use that $0<\alpha<\tfrac14$ in order to control the contribution near the diagonal.

We turn now to \eqref{E:9-4}.  By the triangle inequality, we estimate
\begin{align*}
\bigl(&\E\{\|e^{2 \langle x\rangle^\alpha} \varphi_LR_{L_0}(k)e^{-2\langle x\rangle^\alpha}\|_{H_\kappa^{-1}\to H_{\kappa}^1}^p\}\bigr)^{\frac{1}{p}} \\
& \lesssim  \sum_{m=1}^{8L} \sum_{n\geq 1}\bigl( \E\{\|e^{2\langle x\rangle^\alpha}\chi_m R_{L}(k)\chi_n e^{-2 \langle x\rangle^\alpha}\|_{H_\kappa^{-1}\to H_{\kappa}^1}^p \}\bigr)^{\frac{1}{p}}.
\end{align*}
Using Proposition~\ref{P:9'}, we estimate the contributions away from the diagonal by 
$$
\sum_{m>4n}e^{2\langle 2m\rangle^\alpha} e^{-2\langle n/2\rangle^\alpha} e^{-6\langle m\rangle^\alpha} + \sum_{n>4m}e^{2\langle 2m\rangle^\alpha} e^{-2\langle n/2\rangle^\alpha} e^{-6\langle n\rangle^\alpha}\lesssim1.
$$
To estimate the contribution of $\frac m4\leq n\leq 4m\leq 32 L$, we repeat the argument covering this case in the proof of Proposition~\ref{P:8}, using the decomposition
\begin{align}\label{reveal}
R_{L_0} = R_{32m} - \sum_{\ell=32m}^{L_0/2} R_{2\ell}(q_{2\ell}-q_\ell)R_\ell
\end{align}
in place of \eqref{RL-reveal}.  
\end{proof}

As discussed in the introduction and manifested already in our treatment of the torus problem, the diagonal Green's function $g(x;q,k)$ and its reciprocal play essential roles in our analysis.  The two remaining propositions in this section provide the essential volume-independent bounds on these objects.  We begin with $g(x)$ itself: 

\begin{prop}\label{P:L53} Fix  $1\leq s<\frac32$ and $1\leq p<\infty$.  For $k$ strictly admissible,
\begin{alignat}{2}
\E\bigl\{ \| \langle x\rangle^{-\beta}[g_{L_0}(x;k)-\tfrac{1}{2k}]\|_{L^p}^p\bigr\} & \lesssim_{p,\beta} \kappa^{-\frac{5p}{2}} \quad&&\text{for $\beta>1$},\label{L53-1.5} \\
\E\bigl\{\| \langle x\rangle^{-\beta}[g_{L_0}(x;k)-\tfrac{1}{2k}]\|_{H^s_\kappa}^p \bigr\} & \lesssim_{p,\beta,s} \kappa^{(2s-4)p} \quad&&\text{for $\beta>1$},\label{L53-3.5}\\
\E\bigl\{\| \langle x\rangle^{-\beta}[g_{L_0}(x;k)-\tfrac{1}{2k}]\|_{H^{-1}}^p \bigr\} & \lesssim_{p,\beta} \kappa^{-3p} \quad&&\text{for $\beta>2$},\label{L53-5.5}
\end{alignat}
uniformly in $L=L_0\in 2^{\mathbb N}\cup\{\infty\}$.
\end{prop}

\begin{proof}  As all estimates will be uniform in $L_0$, it will be convenient to drop the subscript everywhere in what follows.

By translation invariance of white noise, \eqref{L53-1.5} will follow readily from
\begin{equation}
\E\bigl\{ \bigl| g(0;k) - \tfrac{1}{2k}\bigr|^p \bigr\} \lesssim \kappa^{-\frac{5p}{2}}.\label{L53-1}
\end{equation}
Recall from \eqref{g-to-quadratic} that 
\begin{align*}
g(0;k)-\tfrac{1}{2k} & = -\langle\delta_0, \tfrac{1}{k}R_0(2k)q\rangle + \langle \delta_0, R_0(k) qR(k)qR_0(k)\delta_0\rangle. 
\end{align*}

As $q$ is white noise distributed,
\begin{align*}
\E\bigl\{\bigl| \langle \delta_0,\tfrac{1}{k}R_0(2k)q\rangle\bigr|^2\bigr\}  \lesssim \kappa^{-2} \E\bigl\{|\langle R_0(2k)\delta_0,q\rangle|^2\bigr\} 
& \lesssim  \kappa^{-2}\|R_0(2k)\delta_0\|_{L^2}^2 \lesssim \kappa^{-5}. 
\end{align*}
Using Lemma~\ref{L3}, we deduce 
\[
\E\bigl\{ \bigl| \langle \delta_0, \tfrac{1}{k}R_0(2k)q\rangle\bigr|^p\bigr\} \lesssim_p \kappa^{-\frac{5p}{2}}.
\]

Next, we estimate
\begin{align*}
\bigl| \langle \delta_0,R_0(k)qR(k)qR_0(k)\delta_0\rangle\bigr|&\lesssim\| \langle x\rangle^2\delta_0\|_{H_\kappa^{-1}} \|\langle x\rangle^4\delta_0\|_{H_\kappa^{-1}} \\
& \quad\times \|\langle x\rangle^{-2}R_0(k)\langle x\rangle^2\|_{H_\kappa^{-1}\to H_\kappa^1} \|\langle x\rangle^4 R_0(k)\langle x\rangle^{-4}\|_{H_\kappa^{-1}\to H_\kappa^1} \\
& \quad \times \|\langle x\rangle^{-2}q\|_{H_\kappa^1\to H_\kappa^{-1}}^2 \|R(k)\langle x\rangle^{-2}\|_{H_\kappa^{-1}\to H_\kappa^1}.
\end{align*}
Thus, using Lemma~\ref{P:11}, \eqref{E:C:q op}, and \eqref{E:9-2}, we get
\begin{align*}
\E\bigl\{ \bigl| \langle \delta_0,R_0(k)qR(k)qR_0(k)\delta_0\rangle\bigr|^p\bigr\} \lesssim_p \kappa^{-3p}.
\end{align*}
This completes the proof of \eqref{L53-1}. 

We turn now to \eqref{L53-3.5}.  Using again \eqref{g-to-quadratic}, we write
\begin{align}\label{200}
g(x;k)-\tfrac{1}{2k} & = -\langle\delta_x, \tfrac{1}{k}R_0(2k)q\rangle + \langle \delta_x, R_0(k) qR(k)qR_0(k)\delta_x\rangle. 
\end{align}

The contribution of the first term on the right-hand side of \eqref{200} is straightforward to estimate.  Indeed, 
\begin{align*}
\E\bigl\{\| \langle x\rangle^{-\beta}\tfrac{1}{k}R_0(2k)q\|_{H^s_\kappa}^2 \bigr\}
&= \kappa^{-2} \tr \bigl\{R_0(2k)^*  \langle x\rangle^{-\beta} (-\Delta + 4\kappa^2)^s  \langle x\rangle^{-\beta} R_0(2k)  \bigr\}\\
&=\kappa^{-2} \iint \frac{|m_\beta(\xi-\eta)|^2 (\eta^2+4\kappa^2)^s}{|\xi^2+ 4k^2|^2}\, d\xi\, d\eta,
\end{align*}
where $m_\beta$ denotes the Fourier transform of the function $\langle x\rangle^{-\beta}$.  As $\beta>1$, we have 
$$
\bigl|m_\beta(\xi)\bigr|\lesssim_n \langle \xi\rangle^{-n} \quad\text{for any $n\in \mathbb N$},
$$
which readily implies that
$$
\E\bigl\{\| \langle x\rangle^{-\beta}\tfrac{1}{k}R_0(2k)q\|_{H^s_\kappa}^2 \bigr\}\lesssim_s  \kappa^{-2} \kappa^{2s-3}\lesssim_s \kappa^{2s-5}.
$$
Invoking Lemma~\ref{L3}, this gives
$$
\E\bigl\{\| \langle x\rangle^{-\beta}\tfrac{1}{k}R_0(2k)q\|_{H^s_\kappa}^p \bigr\}\lesssim_{p,s}  \kappa^{(s-\frac52)p},
$$
which is an acceptable contribution to \eqref{L53-3.5}.

To estimate the contribution of the second term on the right-hand side of \eqref{200}, we argue by duality.  For $f\in H^{-s}_\kappa$ with $s>\frac12$ and $\gamma>0$,
\begin{align}\label{943}
&\bigl|\bigl\langle f, \langle x\rangle^{-\beta}\langle \delta_x, R_0(k) qR(k)qR_0(k)\delta_x\rangle \bigr\rangle\bigr|\\
&\leq \bigl|\tr \bigl\{ f \langle x\rangle^{-\beta}R_0(k) qR(k)qR_0(k)\bigr\}\bigr|\notag\\
&\lesssim \|R_0(\kappa)^{\frac s2} f R_0(\kappa)^{\frac s2}\|_{L^2\to L^2} \|\langle x\rangle^{-\gamma} R(k)\langle x\rangle^{-\gamma}\|_{H^{-1}_\kappa\to H^1_\kappa}\notag\\
&\qquad\times\|(-\Delta + \kappa^2)^{\frac s2}\langle x\rangle^{-\beta/2}R_0(k) q\langle x\rangle^{\gamma}\sqrt{R_0(\kappa)}\|_{\mathfrak I_2}^2\notag\\
&\lesssim \|f\|_{H^{-s}_\kappa} \|\langle x\rangle^{-\gamma} R(k)\langle x\rangle^{-\gamma}\|_{H^{-1}_\kappa\to H^1_\kappa}
\|(-\Delta + \kappa^2)^{\frac s2}\langle x\rangle^{-\beta/2}R_0(k) q\langle x\rangle^{\gamma}\sqrt{R_0(\kappa)}\|_{\mathfrak I_2}^2.\notag
\end{align}

To continue, we write
$$
\langle x\rangle^{-\frac\beta 2}R_0(k)= BR_0(k)\langle x\rangle^{-\frac\beta2} \qtq{with} B=I + R_0(k)A \langle x\rangle^{-\frac\beta2}R_0(k)\langle x\rangle^{\frac\beta2} (-\Delta+k^2)
$$
and 
$$
A= \beta\partial_x \tfrac{x}{\langle x\rangle^{2}} + \beta\tfrac{(\beta+2)x^2-2}{4\langle x\rangle^{4}}.
$$
For $1\leq s\leq 2$, Lemma~\ref{P:11} yields
\begin{align*}
\|B\|_{H^s_\kappa\to H^s_\kappa}
&\leq 1 + \|R_0(k)A\|_{H^1_\kappa\to H^s_\kappa} \|\langle x\rangle^{-\frac\beta2}R_0(k)\langle x\rangle^{\frac\beta2}\|_{H^{-1}_\kappa\to H^1_\kappa}\|-\Delta+k^2\|_{H^s_\kappa\to H^{-1}_\kappa}\\
&\lesssim 1.
\end{align*}
Thus, using \eqref{931} we get
\begin{align*}
\E\bigl\{ \|(-\Delta + \kappa^2)^{\frac s2}\langle x\rangle^{-\frac\beta2}&R_0(k) q\langle x\rangle^{\gamma}\sqrt{R_0(\kappa)}\|_{\mathfrak I_2}^2 \bigr\}\\
&\lesssim \|B\|_{H^s_\kappa\to H^s_\kappa}^2 \E\bigl\{\|(-\Delta + \kappa^2)^{\frac s2}R_0(k) q\langle x\rangle^{\gamma-\frac\beta2}\sqrt{R_0(\kappa)}\|_{\mathfrak I_2}^2\bigr\}\\
&\lesssim \iint\frac{|\xi^2+\kappa^2|^{s}}{(\eta^2+\kappa^2)|\xi^2+k^2|^2} \E\bigl\{\bigl| \mathcal F(q\langle x\rangle^{\gamma-\frac\beta2}) (\xi-\eta)\bigr|^2\bigr\}\, d\xi\, d\eta\\
&\lesssim \|\langle x\rangle^{\gamma-\frac\beta2}\|_{L^2}^2  \iint\frac{|\xi^2+\kappa^2|^{s}}{(\eta^2+\kappa^2)|\xi^2+k^2|^2}\, d\xi\, d\eta\lesssim_s \kappa^{2s-4},
\end{align*}
provided $\frac\beta2>\gamma+\frac12$ and $s<\frac32$.  By Lemma~\ref{L3}, this gives
\begin{align*}
\E\bigl\{ \|(-\Delta + \kappa^2)^{\frac s2}\langle x\rangle^{-\frac\beta2}R_0(k) q\langle x\rangle^{\gamma}\sqrt{R_0(\kappa)}\|_{\mathfrak I_2}^p \bigr\}
\lesssim_{p,\beta,s} \kappa^{p(s-2)} \quad\text{for all $1\leq p<\infty$},
\end{align*}
which combined with \eqref{943} and \eqref{E:9-2} yields
$$
\E\bigl\{ \|\langle x\rangle^{-\beta}\langle \delta_x, R_0(k) qR(k)qR_0(k)\delta_x\rangle\|_{H^s_\kappa}^p \bigr\}\lesssim_{p,\beta,s} \kappa^{p(2s-4)}
$$
for all $\beta>1$ and $1\leq s<\frac32$.  This completes the proof of \eqref{L53-3.5}.

To prove \eqref{L53-5.5}, we argue as for \eqref{L53-3.5}.  Starting from \eqref{200}, we estimate
\begin{align*}
\E\bigl\{\| \langle x\rangle^{-\beta}\tfrac{1}{k}R_0(2k)q\|_{H^{-1}}^2 \bigr\}
&= \kappa^{-2} \tr \bigl\{R_0(2k)^*  \langle x\rangle^{-\beta} (-\Delta + 4)^{-1}  \langle x\rangle^{-\beta} R_0(2k)  \bigr\}\\
&=\kappa^{-2} \iint \frac{|m_\beta(\xi-\eta)|^2 }{|\xi^2+ 4k^2|^2(\eta^2+4)}\, d\xi\, d\eta\lesssim \kappa^{-6}.
\end{align*}
Invoking Lemma~\ref{L3}, we thus obtain
$$
\E\bigl\{\| \langle x\rangle^{-\beta}\tfrac{1}{k}R_0(2k)q\|_{H^{-1}}^p \bigr\}\lesssim_{p}  \kappa^{-3p}.
$$
To estimate the contribution to \eqref{L53-5.5} of the second term on the right-hand side of \eqref{200}, we argue by duality.  For $f\in H^{1}$,
\begin{align*}
&\bigl|\bigl\langle f, \langle x\rangle^{-\beta}\langle \delta_x, R_0(k) qR(k)qR_0(k)\delta_x\rangle \bigr\rangle\bigr|\\
&\leq \bigl|\tr \bigl\{ f \langle x\rangle^{-\beta}R_0(k) qR(k)qR_0(k)\bigr\}\bigr|\\
&\lesssim  \|f\|_{H^1_\kappa\to H^{-1}_\kappa} \|\langle x\rangle^{\frac\beta2} R_0(\kappa)\langle x\rangle^{-\frac\beta2}\|_{H^{-1}_\kappa\to H^1_\kappa} \|\langle x\rangle^{-\frac\beta2}  R_0(\kappa)\langle x\rangle^{\frac\beta2} \|_{H^{-1}_\kappa\to H^1_\kappa}\\
&\quad  \times \|\langle x\rangle^{-\frac\beta4} R(k)\langle x\rangle^{-\frac\beta4}\|_{H^{-1}_\kappa\to H^1_\kappa}\|\sqrt{R_0(\kappa)}\langle x\rangle^{-\frac\beta4}q\sqrt{R_0(\kappa)}\|_{\mathfrak I_2}^2\\
&\lesssim \kappa^{-2}\|f\|_{H^1}\|R(k)\langle x\rangle^{-\frac\beta4}\|_{H^{-1}_\kappa\to H^1_\kappa}\|\sqrt{R_0(\kappa)}\langle x\rangle^{-\frac\beta4}q\sqrt{R_0(\kappa)}\|_{\mathfrak I_2}^2,
\end{align*}
where we used Lemma~\ref{P:11} in the last step.  Using also Lemma~\ref{L:HS} together with \eqref{E:9-2}, we obtain
$$
\E\bigl\{ \|\langle x\rangle^{-\beta}\langle \delta_x, R_0(k) qR(k)qR_0(k)\delta_x\rangle\|_{H^{-1}}^p \bigr\}\lesssim_{p, \beta} \kappa^{-4p},
$$
provided $\beta>2$.  This completes the proof of \eqref{L53-5.5}.
\end{proof}

We turn now to obtaining bounds on $1/g(x;q,\kappa)$ that are independent of the volume.  Evidently, the new difficulty is obtaining \emph{lower} bounds on the diagonal of the resolvent.   This is quite at odds with what is usually investigated in the Anderson model.  Regions where the potential is extremely large (and positive) improve localization; indeed, they act as de facto Dirichlet boundary conditions.  This is advantageous for localization, but an enemy in our setting, since this drives the Green's function to zero!

As our bounds on $1/g(x;q,\kappa)$ need to be volume-independent, multiscale analysis again plays an essential role.  We encapsulate precisely what we need in the following lemma.

\begin{lemma}\label{L:18} Let $\varphi\in C_c^\infty$. For $1\leq p<\infty$, we have
\[
\E\bigl\{ \|\varphi \sqrt{R_{L_0}^*(k)R_{L_0}(k)}\varphi\|_{H_{\kappa}^{-1}\to H_{\kappa}^1}^p\bigr\}\lesssim_{p,\varphi} 1
\]
uniformly over $L_0\in 2^{\mathbb N}\cup\{\infty\}$ and strictly admissible $k$. 
\end{lemma}

\begin{proof} We recall from \eqref{st-equiv2} that
\[
\|\varphi\sqrt{R_{L_0}^*(k)R_{L_0}(k)}\varphi\|_{H_{\kappa}^{-1}\to H_\kappa^1} \approx \biggl\| \int_0^\infty \varphi R_{L_0}^*(k_\tau)R_{L_0}(k_\tau)\varphi\,d\tau\biggr\|_{H_\kappa^{-1}\to H_\kappa^1}, 
\]
where once again we write $k_\tau = \sqrt{k^2+i\tau}$ and $\kappa_\tau=|k_\tau|$.

To proceed, we write
\[
R_{L_0}=R_\ell+ \sum_{L= \ell}^{L_0/2}(R_{2L}-R_L) = R_\ell - \sum_{L=\ell}^{L_0/2} R_{2L}(q_{2L}-q_L)R_L 
\]
where $\ell$ is chosen so that $\supp(\varphi)\subset (-\ell/4, \ell/4)$. As $\kappa_\tau\geq \kappa$, we have
\begin{align}
\|&\varphi\sqrt{R_{L_0}^*(k)R_{L_0}(k)}\varphi\|_{H_\kappa^{-1}\to H_\kappa^1}\nonumber \\
  & \leq \biggl\| \int_0^\infty \varphi R_\ell^* R_\ell\varphi\,d\tau \biggr\|_{H_\kappa^{-1}\to H_\kappa^1} \label{P101} \\
& \quad + 2\sum_{L=\ell}^{\frac12 L_0} \int_0^\infty \|\varphi R_\ell^*R_{2L}(q_{2L}-q_L)R_L\varphi\|_{H_{\kappa_\tau}^{-1}\to H_{\kappa_\tau}^1} \,d\tau \label{P102} \\
& \quad + \sum_{L,L'=\ell}^{\frac12 L_0}\int_0^\infty \|\varphi R_{L'}^*(q_{2L'}-q_{L'})R_{2L'}^*R_{2L}(q_{2L}-q_L)R_L\varphi\|_{H_{\kappa_\tau}^{-1}\to H_{\kappa_\tau}^1}\, d\tau, \label{P103}
\end{align}
where the resolvents are all evaluated at $k_\tau$. 

We first consider the contribution of \eqref{P101}.  Applying \eqref{st-equiv2} once again, followed by \eqref{multiplier1} and Corollary~\ref{C:2}, we estimate the contribution of \eqref{P101} by
\[
\E\bigl\{ \|\varphi\sqrt{R_\ell^*(k) R_\ell(k)}\varphi\|_{H_{\kappa}^{-1}\to H_\kappa^1}^p\bigr\} \lesssim_{p, \varphi} (\log \ell)^p,
\]
which is acceptable. 

By \eqref{multiplier1}, Corollary~\ref{C:2}, \eqref{E:C:q op}, Proposition~\ref{P:4}, \eqref{ktaugain}, and H\"older's inequality, we estimate the contribution of \eqref{P102} by
\begin{align*}
\sum_{L=\ell}^{\frac12 L_0}& \int_0^\infty \Bigl[\E\bigl\{\| \varphi R_\ell^* R_{2L}(q_{2L}-q_L)\tilde \chi_L R_L\varphi\|_{H_{\kappa_\tau}^{-1}\to H_{\kappa_\tau}^1}^p\bigr\}\Bigr]^{\frac{1}{p}}\,d\tau \\
& \lesssim_\varphi \sum_{L=\ell}^{\frac12 L_0} \int_0^\infty \Bigl[\E\bigl\{\|R_\ell^*\|_{H_{\kappa_\tau}^{-1}\to H_{\kappa_\tau}^1}^p \|1\|_{H_{\kappa_\tau}^1 \to H_{\kappa_\tau}^{-1}}^p \| R_{2L}\|_{H_{\kappa_\tau}^{-1}\to H_{\kappa_\tau}^1}^p \\
& \quad\quad\quad\quad \times \|q_{2L}-q_L\|_{H_{\kappa_\tau}^1\to H_{\kappa_\tau}^{-1}}^p \|\tilde\chi_L R_L\varphi \|_{H_{\kappa_\tau}^{-1}\to H_{\kappa_\tau}^1}^p\bigr\}\Bigr]^{\frac1p}\,d\tau \\
& \lesssim_{p, \varphi} \sum_{L=\ell}^{\frac12 L_0} \int_0^\infty \frac{\log \ell \cdot [\log L]^{2+\frac12} \cdot e^{-8\langle L\rangle ^{\alpha}}}{(\kappa^2+\tau)^{\frac32}}\,d\tau \lesssim_{p,\varphi} \tfrac{1}{\kappa},
\end{align*}
where $\tilde\chi_L$ is a cutoff to the support of $q_{2L}-q_L$ as defined in \eqref{chitilde}.  The penultimate step is immediate when $L_0=\infty$.  When $L_0$ is finite, one must first introduce an additional sum over $h\in \Z$ as in the proof of Propositions~\ref{P:8} and \ref{P:9'}, which is then easily seen to be controlled by the $h=0$ summand.

Arguing similarly, we estimate the contribution of \eqref{P103} by
\begin{align*}
\sum_{L,L'=\ell}^{\frac12 L_0}& \int_0^\infty  \|1\|_{H_{\kappa_\tau}^1\to H_{\kappa_\tau}^{-1}} \E\Bigl\{ \|\varphi R_{L'}^*\tilde\chi_{L'}\|_{H_{\kappa_\tau}^{-1}\to H_{\kappa_\tau}^1}^p \|q_{2L'}-q_{L'}\|_{H_{\kappa_\tau}^1\to H_{\kappa_\tau}^{-1}}^p \\
&\quad\quad\quad\times \|R_{2L'}^*\|_{H_{\kappa_\tau}^{-1}\to H_{\kappa_\tau}^1}^p \|R_{2L}\|_{H_{\kappa_\tau}^{-1} \to H_{\kappa_\tau}^1}^p \|q_{2L}-q_{L}\|_{H_{\kappa_\tau}^1\to H_{\kappa_\tau}^{-1}}^p \\
& \quad\quad\quad\times \|\tilde \chi_L R_L \varphi\|_{H_{\kappa_\tau}^{-1}\to H_{\kappa_\tau}^1}^p\Bigr\}^{\frac{1}{p}}\,d\tau \\
& \lesssim_p \sum_{L,L'=\ell}^{\frac12 L_0} \int_0^\infty \frac{[\log L']^{2+\frac12}e^{-8\langle L'\rangle^\alpha}\cdot [\log L]^{2+\frac12} e^{-8\langle L\rangle ^\alpha}}{(k^2+\tau)^2}\,d\tau \lesssim_{p,\varphi} \tfrac{1}{\kappa^2},
\end{align*}
which is acceptable. This completes the proof. \end{proof}

We are now ready to prove the bounds we need on the reciprocal of the diagonal Green's function; these are optimal in $\kappa$, as we can see already from the case $q\equiv 0$.

\begin{prop}\label{P:L56} Fix $1\leq p<\infty$.  For $k$ strictly admissible, we have
\begin{alignat}{2}
\E\bigl\{ \bigl\| \langle x\rangle^{-\beta}\tfrac{1}{g_{L_0}(x;k)}\bigr\|_{L^p}^p \bigr\} & \lesssim_{p,\beta} \kappa^p
	\quad &&\text{for $\beta>\tfrac1p$},\label{L56-3.5}\\
\E\bigl\{ \bigl\| \langle x\rangle^{-\beta}\tfrac{1}{g_{L_0}(x;k)}\bigr\|_{H^s}^p\bigr\} & \lesssim_{p, \beta} \kappa^{p}
	\quad &&\text{for $\beta>5$ and $0\leq s<\tfrac32$}, \label{L56-1.5} 
\end{alignat}
uniformly in $L=L_0\in 2^{\mathbb N}\cup\{\infty\}$.
\end{prop}

\begin{proof}  As the arguments to be presented do not depend on $L_0$, we suppress it in the notation below.

By translation invariance of white noise, \eqref{L56-3.5} will follow readily from
\begin{equation}\label{L53-3}
\E\bigl\{  \bigl|\tfrac{1}{g(0;k)}\bigr|^p\}  \lesssim_p \kappa^p  
\end{equation}
To prove \eqref{L53-3}, we use
\[
\tfrac{1}{|g(0;k)|} \leq \tfrac{1}{| \Im g(0;k)|}
\]
and (recalling that $\sigma=\Im k^2$)
\begin{equation}\label{eqnforImg0}
\Im g(0;k) = \tfrac{1}{2i} \langle \delta_0,[R(k)-R^*(k)]\delta_0\rangle = -\sigma \langle \delta_0, R^*(k)R(k)\delta_0\rangle. 
\end{equation}
Using H\"older (with respect to the spectral measure), we estimate
\begin{align*}
|\langle\delta_0, \sqrt{R^*(k)R(k)}\delta_0\rangle| & = \|[R^*(k)R(k)]^{\frac14}\delta_0\|_{L^2}^2 \\
& \lesssim | \langle \delta_0, R^*(k)R(k)\delta_0\rangle|^{\frac15} |\langle \delta_0,[R^*(k)R(k)]^{\frac38}\delta_0\rangle|^{\frac45},
\end{align*}
and so
\begin{equation}\label{Imgk1}
|\Im g(0;k)|^{-1} \lesssim \sigma^{-1} |\langle \delta_0,\sqrt{R^*(k)R(k)}\delta_0\rangle|^{-5} |\langle\delta_0, [R^*(k)R(k)]^{\frac38}\delta_0\rangle|^4. 
\end{equation}

We now let $\psi\in C_c^\infty$ satisfy
\begin{equation}\label{Imgk-psi}
\psi(x)=\begin{cases} 1, &\quad |x|\leq \tfrac{1}{2\kappa} \\ 0, & \quad |x|\geq \tfrac{1}{\kappa}.\end{cases}
\end{equation}
By Cauchy--Schwarz,
\begin{align*}
1= |\langle \delta_0,\psi\rangle|^2 & = |\langle[R(k)^*R(k)]^{\frac14}\delta_0,[(H+k^2)^*(H+k^2)]^{\frac14}\psi\rangle|^2 \\ 
& \leq \bigl|\langle\delta_0,\sqrt{R^*(k)R(k)}\delta_0\rangle\bigr|\, \bigl| \langle \psi,|H+k^2|\psi\rangle\bigr|. 
\end{align*}
In particular, 
\[
|\langle\delta_0,\sqrt{R^*(k)R(k)}\delta_0\rangle|^{-1}\leq |\langle\psi,|H+k^2|\psi\rangle|.
\]

We now let $\tilde\psi\in C_c^\infty$ be a bump function with $\tilde \psi \psi=\psi$.  As $H+k^2$ is a local operator and all operators involved are normal, 
\begin{align*}
|\langle \psi,|H+k^2|\psi\rangle| & = |\langle (H+k)^2\psi, \sqrt{R^*(k)R(k)}(H+k^2)\psi\rangle| \\
 & = |\langle (H+k^2)\psi,\tilde\psi\sqrt{R^*(k)R(k)}\tilde\psi(H+k^2)\psi\rangle|   \\
& \lesssim \|(H+k^2)\psi\|_{H_{\kappa}^{-1}}^2 \|\tilde\psi\sqrt{R^*(k)R(k)}\tilde\psi\|_{H_\kappa^{-1}\to H_\kappa^1} \\
& \lesssim\bigl( \|\psi\|_{H_\kappa^1}^2 + \|q\psi\|_{H_\kappa^{-1}}^2\bigr) \|\tilde\psi\sqrt{R^*(k)R(k)}\tilde\psi\|_{H_\kappa^{-1}\to H_\kappa^1} \\
& \lesssim\bigl(\kappa + \|q\psi\|_{H_\kappa^{-1}}^2\bigr) \|\tilde\psi\sqrt{R^*(k)R(k)}\tilde\psi\|_{H_\kappa^{-1}\to H_\kappa^1}. 
\end{align*}
Thus,
\begin{equation}\label{ImgkUB3}
|\langle\delta_0,\sqrt{R^*(k)R(k)}\delta_0\rangle|^{-1} \lesssim \bigl(\kappa+\|q\psi\|_{H_{\kappa}^{-1}}^2\bigr)\|\tilde\psi\sqrt{R^*(k)R(k)}\tilde\psi\|_{H_\kappa^{-1}\to H_\kappa^1}.
\end{equation}
By \eqref{weighted-wn}, 
\[
\E\{\|q\psi\|_{H_{\kappa}^{-1}}^2\bigr\}= \tfrac1{4\kappa}\|\psi\|_{L^2}^2 \lesssim \kappa^{-2},
\]
which (appealing to Lemma~\ref{L3}) then implies
\[
\E\{\|q\psi\|_{H_\kappa^{-1}}^p\}\lesssim_p \kappa^{-p}
\]
for all $1\leq p<\infty$.  Using this together with Lemma~\ref{L:18}, we get
\begin{equation}\label{ImgkUB4}
\E\bigl\{ |\langle \delta_0,\sqrt{R^*(k)R(k)}\delta_0\rangle|^{-p}\} \lesssim_p \kappa^{p} 
\end{equation}
for any $1\leq p<\infty$. 

We turn to the remaining term in \eqref{Imgk1}. We recall the notation from the proof of Proposition~\ref{P:2} and write $k_\tau:=\sqrt{k^2+i\tau}$ for $\tau\in(0,\infty)$.  We begin by observing
\begin{align*}
\langle  \delta_0, [R^*(k)R(k)]^{\frac38}\delta_0\rangle
& = \int \bigl(|\lambda+E|^2+\sigma^2\bigr)^{-\frac34}\,d\mu(\lambda) \\
& \approx \int_0^\infty \int \frac{\tau^{\frac14}}{|\lambda+E|^2+\tau^2+\sigma^2}\,d\mu(\lambda)\,d\tau \\
& \approx \int_0^\infty \int \frac{\tau^{\frac14}}{|\lambda+k^2+i\tau|^2}d\mu(\lambda)\,d\tau \\
& = \biggl\langle \delta_0,\biggl(\int_0^\infty R^*(k_\tau)R(k_\tau)\tau^{\frac14}\,d\tau\biggr)\delta_0\biggr\rangle \\
& =\int_0^\infty \langle\delta_0, R^*(k_\tau)R(k_\tau)\delta_0\rangle\tau^{\frac14}\,d\tau=-\int_0^\infty \tfrac{\tau^{\frac14}}{\sigma+\tau}\Im g(0;k_\tau)\,d\tau,
\end{align*}
where $d\mu$ denotes the spectral measure for $(H, \delta_0)$. Thus, using \eqref{L53-1}, we get
\begin{align}
\bigl[\E\bigl\{|\langle\delta_0,\bigl[R^*(k)R(k)\bigr]^{\frac38}\delta_0\rangle|^p\bigr\}\bigr]^{\frac{1}{p}} 
& \lesssim \int_0^\infty \bigl[\E\bigl\{ |\Im g(0;k_\tau)|^p\bigr\}\bigr]^{\frac{1}{p}}\tfrac{\tau^{\frac14}}{\sigma+\tau}\,d\tau \nonumber \\
& \lesssim \int_0^\infty \tfrac{\tau^{\frac14}}{\sigma+\tau}|k_\tau|^{-1}\,d\tau   \lesssim \sigma^{-\frac14} \label{ImgkUB5}.
\end{align}
Continuing from \eqref{Imgk1} and using \eqref{ImgkUB4} and \eqref{ImgkUB5}, we finally derive 
\begin{align*}
\E\bigl\{ \bigl| \tfrac{1}{g(0;k)} \bigr|^p\bigr\}  &\lesssim \sigma^{-p}\E\bigl\{ \langle \delta_0, \sqrt{R^*(k)R(k)}\delta_0\rangle^{-5p}\langle \delta_0, [R^*(k)R(k)]^{\frac38}\delta_0\rangle^{4p}\bigr\} \\
& \lesssim_p \sigma^{-2p} \kappa^{5p} \lesssim_p \kappa^p
\end{align*}
since $\kappa^2\approx\sigma$, by strict admissiblity.  This completes the proof of \eqref{L53-3}.

Turning to  \eqref{L56-1.5}, we first consider the case $s=1$.  By \eqref{L56-3.5}, Sobolev embedding, and \eqref{L53-3.5}, for $1<p<\infty$ we get
\begin{align*}
\E\bigl\{ \bigl\| \langle x\rangle^{-\beta}\bigl(\tfrac{1}{g(x;k)}\bigr)'\bigr\|_{L^2}^p\bigr\}
&\lesssim \Bigl(\E\bigl\{\bigl\| \langle x\rangle^{-\frac\beta 4}\tfrac{1}{g(x;k)}\bigr\|_{L^{4p}}^{4p} \bigr\} \Bigr)^{\frac12} \Bigl(\E\bigl\{\bigl\| \langle x\rangle^{-\frac\beta 2}g'(x;k)\bigr\|_{L^{\frac{2p}{p-1}}}^{2p} \bigr\} \Bigr)^{\frac12}\\
&\lesssim_{p, \beta} \kappa^{2p} \Bigl(\E\bigl\{\| \langle x\rangle^{-\frac\beta2}[g(x;k)-\tfrac{1}{2k}]\|_{H^{1+\frac1{2p}}}^{2p} \bigr\}\Bigr)^{\frac12}\lesssim_{p, \beta} \kappa,
\end{align*}
provided $\beta>2$.  On the other hand, by \eqref{L53-3} and either Jensen's or H\"older's inequalities,
\begin{align*}
\Bigl( \E\bigl\{ \bigl\| \langle x\rangle^{-\beta}\tfrac{1}{g(x;k)}\bigr\|_{L^2}^p \bigr\} \Bigr)^{\frac1p} \lesssim_{p, \beta} \kappa \quad \text{for all $1\leq p<\infty$ and $\beta>1$.}
\end{align*}
Combining the two bounds above we get
\begin{align}\label{1052}
\E\bigl\{ \bigl\| \langle x\rangle^{-\beta}\tfrac{1}{g(x;k)}\bigr\|_{H^s}^p \bigr\}  \lesssim_p \kappa^p \quad \text{for all $1\leq p<\infty$, $0\leq s\leq 1$, and $\beta>2$.}
\end{align}

To treat $1<s<\frac32$, we observe that by the fractional (and usual) product rule and Sobolev embedding,
\begin{align*}
\bigl\| |\nabla|^{s-1}&\bigl[ \langle x\rangle^{-\beta}\bigl(\tfrac{1}{g(x;k)}\bigr)'\bigr]\bigr\|_{L^2}\\
&\lesssim \bigl\| |\nabla|^{s-1}\bigl[ \langle x\rangle^{-\gamma}\tfrac{1}{g(x;k)}\bigr]\bigr\|_{L^{\frac2{s-1}}}\bigl\|\langle x\rangle^{-\gamma}\tfrac{1}{g(x;k)}\bigr\|_{L^{\frac2{s-1}}}
\bigl\| \langle x\rangle^{-\beta+2\gamma}g'(x;k)\bigr\|_{L^{\frac2{3-2s}}}\\
&\quad+ \bigl\|\langle x\rangle^{-\gamma}\tfrac{1}{g(x;k)}\bigr\|_{L^\infty}^2\bigl\| |\nabla|^{s-1}\bigl[ \langle x\rangle^{-\beta+2\gamma}g'(x;k)\bigr]\bigr\|_{L^2}\\
&\lesssim \bigl\|\langle x\rangle^{-\gamma}\tfrac{1}{g(x;k)}\bigr\|_{H^1}^2\Bigl\{\bigl\| \langle x\rangle^{-\beta+2\gamma}g(x;k)\bigr\|_{H^s} + \bigl\| \langle x\rangle^{-\beta+2\gamma-1}g(x;k)\bigr\|_{H^{s-1}}\bigr\}.
\end{align*}
The result now follows from \eqref{L53-3.5} and \eqref{1052} by taking $\gamma>2$ and $\beta-2\gamma>1$.
\end{proof}

\section{Invariance of white noise for the \texorpdfstring{$\H_k$}{Hk} flow on the line }\label{S:L}

Our main objective in this section is to `send $L_0\to\infty$' and deduce invariance of white noise under the $\H_k$ flow \eqref{H_k flow q} on the entire real line.  The precise meaning of a solution in this context will be given in Definition~\ref{D:good} below. First, however, we must settle some notational conventions.

Starting in this section (and continuing until the end of the paper), we work only with the fully revealed potential, which corresponds to setting $L=L_0$ in \eqref{qL}. 

Throughout this section, we freeze the parameter $k$:
\begin{equation}\label{arg k}
k = \kappa e^{i\pi/8}
\end{equation}
with $\kappa$ large enough so that $k$ is strictly admissible.   Like the choice of opening angle in the definition of strict admissibility, this is essentially arbitrary; we prefer concrete numbers purely for expository reasons.  One caveat on this generality (and rationale for choosing specificity) is that in order to prove Lemma~\ref{L:gL dif} below in the requisite generality, we do need $k$ to lie in a slightly narrower sector than that used in the preceding sections.    

\begin{definition}\label{D:good} We say that $q:\R\to\mathcal S'(\R)$ is a \emph{global good solution} to the $\H_k$ flow \eqref{H_k flow q} if for every $T>0$ there exist $0<\alpha<\frac14$ and $C\geq 1$ such that
\begin{gather}
\sup_{t\in[-T,T]}\|e^{-\langle x\rangle^\alpha}q(t)\|_{H_\kappa^{-1}} <\infty,\label{good-solution0}\\
\int_{-T}^T\|e^{\langle x\rangle^\alpha}R(k)e^{-2\langle x\rangle^\alpha} \|_{H_\kappa^{-1}\to H_\kappa^1}^4\,dt <\infty, \label{good-solution}\\
\sup_{L \in 2^{\mathbb{N}}}  [\log L]^{-C} \int_{-T}^T\|e^{2\langle x\rangle^\alpha} \varphi_LR(k) e^{-2\langle x\rangle^\alpha}\|_{H_{\kappa}^{-1}\to H_\kappa^1}^2\, dt < \infty, \label{good-solution'}
\end{gather}
and
\begin{equation}\label{weak-soln-Hk}
\langle \varphi,q(t) \rangle = \bigl\langle\varphi, e^{4\Re k^2 t\partial} q(0)+\int_0^t e^{4\Re k^2(t-s)\partial} \Re[16 k^5 g'(q(s))]\,ds\bigr\rangle
\end{equation}
for all $\varphi\in C_c^\infty(\R)$ and all $t\in[-T,T]$.  Here $R(k)$ is the resolvent associated to the potential $q$ with $L_0=\infty$ and $\varphi_L$ denotes a smooth cutoff as in \eqref{phiL}.
\end{definition}

This notion of good solution is motivated by the fact that if $q(t)$ is white noise distributed at each time $t\in[-T,T]$, then $q$ satisfies the bounds in \eqref{good-solution} and \eqref{good-solution'} not only almost surely, but even in expectation. Indeed, this is a consequence of Corollary~\ref{C:9}. 

The main results of this section are Propositions~\ref{P:666} and~\ref{P:uniq}.  Taken together they yield

\begin{theorem}\label{T:invariance-Hk-line}For almost every sample $q^0$ of white noise on the real line, there exists a unique global good solution $q$ to the $\H_k$ flow \eqref{H_k flow q} on the real line with initial data $q^0$. As a random variable, $q(t)$ is white noise distributed at each $t\in\R$. 
\end{theorem}

The remainder of this section is dedicated to the proof of this theorem, as well as certain additional conclusions that will be needed in the next section.

Given white noise $q^0$ on the line and $L\in 2^{\mathbb{N}}$, we define
\begin{equation}\label{qL0-S6}
q_L^0(x) = \sum_{n\in\mathbb{Z}} \bigl[1_{[-L,L]}q^0\bigr](x-2nL),
\end{equation}
which is $2L$-periodic white noise.  Using Proposition~\ref{P:H kappa}, we define $q_L(t)$ to be the global $2L$-periodic solution to the $\H_k$ flow \eqref{H_k flow q} with initial data $q_L^0$.  Theorem~\ref{T:Hktorus} yields that as a random variable, $q_L(t)$ is $2L$-periodic white noise for all $t\in\R$.  In particular, $q_L(t)$ obeys the bounds \eqref{good-solution} and \eqref{good-solution'} almost surely.  Below we will also show that $q_L$ satisfies \eqref{good-solution0}, so that $q_L$ is a global good solution to \eqref{H_k flow q}. 

We will prove that almost surely, the solutions $q_L(t)$ converge to a limit $q(t)$ for all $t\in\R$ in a suitable norm as $L\to\infty$; see Lemma~\ref{T:L0-to-infinity} below.  We will then show that $q(t)$ is white noise distributed for each $t$ and that $q$ is a global good solution to the $\H_k$ flow \eqref{H_k flow q} on the line; see Proposition~\ref{P:666}.  
Finally, we will demonstrate the uniqueness of good solutions to the $\H_k$ flow in Proposition~\ref{P:uniq}; this implies, in particular, that the solution $q(t)$ is independent of the particular sequence used to construct it.

We begin with the following convergence result.

\begin{lemma}\label{T:L0-to-infinity} Fix $T>0$, $\alpha\in(0,\tfrac14)$, and $\delta>0$. Then almost surely,
\begin{equation}\label{LConvergence1}
\sum_{L\in 2^{\mathbb{N}}} L^{2\delta} \sup_{t\in [-T,T]} \|e^{-\langle x\rangle^\alpha}[q_{2L}(t) - q_{L}(t)]\|_{H_{\kappa}^{-1}}^2<\infty.
\end{equation}
\end{lemma}

\begin{proof} For $L\in 2^{\mathbb{N}}$ and $t\in [-T,T]$, we define 
\[
d_{L}(t) = \| e^{-\langle x\rangle^\alpha}[q_{2L}(t)-q_{L}(t)]\|_{H_{\kappa}^{-1}}^2. 
\]
As $q_{2L}^0-q_L^0=0$ on $[-L,L]$, using \eqref{weighted-wn} we obtain
\begin{align}\label{dL0-2}
\E\{d_{L}(0)\} & \lesssim \E\bigl\{\|e^ {-\langle x\rangle^\alpha} 1_{|x|\geq L} q_{2L}^0\|^2_{H^{-1}_\kappa} +\|e^ {-\langle x\rangle^\alpha} 1_{|x|\geq L} q_{L}^0\|^2_{H^{-1}_\kappa} \bigr\}\lesssim e^{-L^\alpha}.
\end{align}

In what follows, we will show
\begin{equation}\label{dL0-1}
\sup_{t\in[-T,T]} d_{L}(t) \leq d_{L}(0)\exp\biggl\{\int_{-T}^T X_L(t)\,dt\biggr\},
\end{equation}
where $X_L(t)\geq 0$ is a quantity that satisfies
\begin{equation}\label{dL0*}
X_L(t) \approx \kappa^2+\kappa^3\|R_{2L}(k)\|_{H_{\kappa}^{-1}\to H_{\kappa}^1}\| e^{\langle x\rangle^\alpha}R_{L}(k)\varphi_Le^{-\langle x\rangle^\alpha}\|_{H_{\kappa}^{-1}\to H_{\kappa}^1}  
\end{equation}
uniformly for $t\in[-T,T]$.  By \eqref{RLk-itself} and Proposition~\ref{P:8}, for any $1\leq p<\infty$,
\begin{equation}\label{dL0-3}
\E\bigl\{|X_L(t)|^p\bigr\}  \lesssim_p \kappa^{2p}+\kappa^{3p}[\log L]^{3p}\qtq{uniformly for}t\in[-T,T]. 
\end{equation}

Assuming \eqref{dL0-1} and \eqref{dL0-3} for the moment, let us show that \eqref{LConvergence1} holds.   In view of \eqref{dL0-2} and \eqref{dL0-3}, we have that both
\begin{align*}
\sup_{L\in 2^{\mathbb N}} e^{L^\alpha/2} d_L(0) \qtq{and}   \sup_{L\in 2^{\mathbb N}} \tfrac{1}{[\log L]^{5}} \!\int_{-T}^T X_L(t)\,dt 
\end{align*}
are almost surely finite (indeed, the \emph{sum} in $L$ has finite expectation).   Building on this, \eqref{LConvergence1} follows from \eqref{dL0-1}.

It remains to prove \eqref{dL0-1} and \eqref{dL0*}.   We start with \eqref{dL0-1}, which we will prove using the equation \eqref{H_k flow q} and Gronwall's inequality.  In particular, we will establish the bound 
\[
|\partial_t d_L(t)| \leq X_L(t)d_L(t)
\]
for a nonnegative quantity $X_L(t)$ obeying \eqref{dL0*}. 

To this end, we write $g_L(k)$ and $g_{2L}(k)$ for the diagonal Green's functions associated to $q_L$ and $q_{2L}$.  Using \eqref{H_k flow q} and the identity
\begin{equation}\label{d-pr}
(e^{-\langle x\rangle^\alpha}f)' = e^{-\langle x\rangle^\alpha}f' - \tfrac{\alpha x}{\langle x\rangle^{2-\alpha}} e^{-\langle x\rangle^\alpha} f,
\end{equation}
we compute
\begin{align}
\partial_t d_L(t) & = 32\Re\bigl\{ k^5\bigl\langle e^{-\langle x\rangle^{\alpha}}[q_{2L}-q_L],e^{-\langle x\rangle^{\alpha}}[g_{2L}(k)-g_L(k)]'\bigr\rangle_{H_{\kappa}^{-1}}\bigr\} \label{dprime1}\\
& \quad +8\alpha\Re\bigl\{ k^2\bigl\langle e^{-\langle x\rangle^\alpha}[q_{2L}-q_L],\tfrac{x}{\langle x\rangle^{2-\alpha}}e^{-\langle x\rangle^{\alpha}}[q_{2L}-q_L]\bigr\rangle_{H_{\kappa}^{-1}}\bigr\}. \label{dprime22}
\end{align}

Applying H\"older in \eqref{dprime1} and using that multiplication by $x\langle x\rangle^{\alpha-2}$ is bounded on $H_\kappa^{-1}$ to handle \eqref{dprime22}, we get 
\[
|\partial_t d_L(t)| \lesssim \kappa^2 d_L(t) + \kappa^5 \|e^{-\langle x\rangle^{\alpha}}[g_{2L}(k)-g_L(k)]\|_{L^2}\,d_L(t)^{\frac12}.
\]
In this way, the proof of \eqref{dL0-1} and thence also \eqref{LConvergence1} will be complete once we show 
\begin{align}\label{dL0-41}
 \|&e^{-\langle x\rangle^{\alpha}}[g_{2L}(k)-g_L(k)]\|_{H^1_\kappa}\\
 &\lesssim \kappa^{-1}\|R_{2L}(k)\|_{H_{\kappa}^{-1}\to H_{\kappa}^1}\| e^{\langle x\rangle^\alpha}R_{L}(k)\varphi_Le^{-\langle x\rangle^\alpha}\|_{H_{\kappa}^{-1}\to H_{\kappa}^1} \|e^{-\langle x\rangle^{\alpha}}[q_{2L}-q_L]\|_{H^{-1}_\kappa}\notag.
\end{align}

As both $g_{2L}$ and $g_L$ are $4L$-periodic,
$$
\|e^{-\langle x\rangle^{\alpha}}[g_{2L}(k)-g_L(k)]\|_{H^1_\kappa} \lesssim \|e^{-\langle x\rangle^{\alpha}}\varphi_L[g_{2L}(k)-g_L(k)]\|_{H^1_\kappa},
$$
where $\varphi_L$ is as in \eqref{phiL}.  We estimate this norm by duality.  For $f\in H^{-1}_\kappa$, we use the resolvent identity
\[
R_{2L}-R_{L} = -R_{2L}(q_{2L}-q_L)R_L
\]
to write
\[
\langle e^{-\langle x\rangle^\alpha}\varphi_L[g_{2L}-g_L], f\rangle = -\tr\{ R_{2L}(q_{2L}-q_L)R_L\varphi_L e^{-\langle x\rangle^\alpha}f\}. 
\]
Cycling the trace and using Lemma~\ref{I2H1}, we  get
\begin{align*}
\bigl|&\tr\bigl\{ R_{2L}(k)(q_{2L}-q_L)R_{L}(k)\varphi_L e^{-\langle x\rangle^\alpha}f\bigr\}\bigr| \\
 & \lesssim \| R_{2L}(k)\|_{H^{-1}_\kappa\to H^1_\kappa} \| \sqrt{R_0(\kappa)}(q_{2L}-q_L)e^{-\langle x\rangle^\alpha}\sqrt{R_0(\kappa)}\|_{\mathfrak{I}_2} \\
& \quad\quad \times \|e^{\langle x\rangle^\alpha} R_{L}(k)\varphi_L e^{-\langle x\rangle^\alpha}\|_{H^{-1}_\kappa\to H^1_\kappa}\|\sqrt{R_0(\kappa)}f\sqrt{R_0(\kappa)}\|_{\mathfrak{I}_2} \\
& \lesssim \kappa^{-1}\|R_{2L}(k)\|_{H_{\kappa}^{-1}\to H_{\kappa}^1} \|e^{-\langle x\rangle^\alpha}(q_{2L}-q_L)\|_{H_\kappa^{-1}}   \|e^{\langle x\rangle^\alpha}R_L(k)\varphi_L e^{-\langle x\rangle^\alpha}\|_{H_{\kappa}^{-1}\to H_{\kappa}^1}\|f\|_{H_\kappa^{-1}}.
\end{align*}
Taking the supremum over unit vectors $f\in H^{-1}_\kappa$, we conclude that \eqref{dL0-41} holds.
\end{proof}

Our next result shows convergence in probability for the diagonal Green's function at each time.  This is a crucial stepping stone to the stronger convergence we will demonstrate in Proposition~\ref{P:666} below.

\begin{lemma}\label{L:6.4?}
Given $T>0$, $\alpha\in (0,\frac14)$, $\eps>0$, and $\vk$ strictly admissible,
\begin{align*}
\lim_{L\to\infty}\ \sup_{|t|\leq T}
	\PP\biggl\{\sup_{L'\in 2^\mathbb{N}L} \bigl\| e^{-\langle x\rangle^\alpha} [g_{L'}(t;\vk)-g_L(t;\vk)\bigr\|_{H^1} > \eps \biggr\} = 0,
\end{align*}
where $g_L(t;\vk)=g_L(x;q_L(t),\vk)$ is the diagonal Green's function associated to $q_L(t)$. 
\end{lemma}

\begin{proof}
Recall from \eqref{dL0-41} that
\begin{align*}
\|&e^{-\langle x\rangle^{\alpha}}[g_{2L}(t;\vk)-g_L(t;\vk)]\|_{H^1}\\
 &\ \lesssim \|R_{2L}(\vk)\|_{H_{|\vk|}^{-1}\to H_{|\vk|}^1}\| e^{\langle x\rangle^\alpha}R_{L}(\vk)\varphi_Le^{-\langle x\rangle^\alpha}\|_{H_{|\vk|}^{-1}\to H_{|\vk|}^1} \|e^{-\langle x\rangle^{\alpha}}[q_{2L}-q_L]\|_{H^{-1}_{|\vk|}}.
\end{align*}
Thus, letting $\delta>0$ and defining
\begin{align*}
X(t) = \sum_{M\in 2^\mathbb{N}} M^{-\delta} \|R_{2M}(\vk)\|_{H_{|\vk|}^{-1}\to H_{|\vk|}^1}^2
	\| e^{\langle x\rangle^\alpha}R_{M}(\vk)\varphi_M e^{-\langle x\rangle^\alpha}\|_{H_{|\vk|}^{-1}\to H_{|\vk|}^1}^2
\end{align*}
and
\begin{align*}
Y_L(t) = L^{-\delta} \sum_{M\in 2^\mathbb{N}} M^{2\delta} \|e^{-\langle x\rangle^{\alpha}}[q_{2M}-q_M]\|_{H^{-1}_{|\vk|}}^2,
\end{align*}
we get
\begin{align*}
\sup_{L'\in 2^\mathbb{N}L} \|e^{-\langle x\rangle^{\alpha}}[g_{L'}(t;\vk)-g_L(t;\vk)]\|_{H^1}^2 \lesssim X(t) Y_L(t).
\end{align*}

Using Lemma~\ref{T:L0-to-infinity} and the fact that $H^1_\kappa$ and $H^1_{|\vk|}$ have comparable norms, we get
$$
\sup_{|t|\leq T} Y_L(t) \to 0\quad\text{as}\quad L\to\infty
$$
almost surely and so also in probability.  On the other hand, by \eqref{RLk-itself} and Proposition~\ref{P:8},
$$
\sup_{|t|\leq T} \E\bigl\{ X(t) \} < \infty .
$$

Noting that for every $\eta>0$,
$$
\PP\{ X(t) Y_L(t) > \eta \} \leq \PP\{ X(t) > \eta^{-1} \} + \PP\{ Y_L(t) > \eta^2 \},
$$
this gives
$$
\sup_{|t|\leq T} \PP\{ X(t) Y_L(t) > \eta \} \leq \eta \sup_{|t|\leq T}  \E\{ X(t) \} + \PP\{ \sup_{|t|\leq T} Y_L(t) > \eta^2 \},
$$
which proves the lemma.
\end{proof}

To continue, for fixed $T>0$, the estimate \eqref{LConvergence1} implies that almost surely, there exists a limit $q$ on $[-T,T]$ satisfying
\[
e^{-\langle x\rangle^\alpha}q \in C_t ([-T,T]; H_\kappa^{-1}). 
\]
Taking a sequence $T_n\to\infty$, we can find a full probability event on which $q_L(t)$ converges to $q(t)$  for all $t\in\R$.  We wish to prove that as a random variable, $q(t)$ is white noise distributed for all $t\in\R$.  Due to the extremely rapid decay of the weight (which is essential for proving Lemma~\ref{T:L0-to-infinity}), we do not yet know enough about $q(t)$ to even make sense of it as a tempered distribution.  To remedy this, we establish improved bounds for $q_L$, which will be inherited by $q$.

\begin{lemma}\label{L:66} For any $1\leq p<\infty$, $\beta>1$, and $T>0$,
\begin{equation}\label{qL-weighted}
\E\biggl\{\sup_{t\in[-T,T]}\|\langle x\rangle^{-\beta}q_L(t)\|_{H_\kappa^{-1}}^p \biggr\} \lesssim_{p,\beta, T, \kappa} 1 
\end{equation}
uniformly over $L\in 2^{\mathbb{N}}$.
\end{lemma}

\begin{proof} The proof of \eqref{qL-weighted} is similar to that of Lemma~\ref{T:L0-to-infinity}.  We define
\[
A_L(t) = \|\langle x\rangle^{-\beta}q_L(t)\|_{H_\kappa^{-1}}^2. 
\]
Using the equation \eqref{H_k flow q} and the identity
\[
(\langle x\rangle^{-\beta} f)' = \langle x\rangle^{-\beta} f' - \tfrac{\beta x}{\langle x\rangle^2} \langle x\rangle^{-\beta} f,
\]
we derive
\[
|\partial_tA_L(t) | \lesssim \kappa^2 A_L(t) + \kappa^5\sqrt{A_L(t)}\| \langle x\rangle^{-\beta} g_L\|_{L^2}. 
\]
We estimate the $L^2$ norm by duality: for $h\in L^2$ and $\beta>1$, 
\begin{align*}
\bigl| \bigl\langle \langle x\rangle^{-\beta} g_L,h\bigr\rangle\bigr|  &= |\tr\{\langle x\rangle^{-\beta} R_L h\}| \\
& \lesssim \|\sqrt{R_0(\kappa)}\langle x\rangle^{-\frac\beta 2} h\sqrt{R_0(\kappa)}\|_{\mathfrak{I}_1} \| \langle x\rangle^{-\frac\beta 2}R_L(k)\|_{H_{\kappa}^{-1}\to H_\kappa^1} \\
& \lesssim \kappa^{-1} \|\langle x\rangle^{-\frac\beta 2} h\|_{L^1} \|\langle x\rangle^{-\frac\beta 2}R_L(k)\|_{H_\kappa^{-1}\to H_\kappa^1} \\
& \lesssim \kappa^{-1}\|h\|_{L^2}\|\langle x\rangle^{-\frac\beta 2}R_L(k)\|_{H_\kappa^{-1}\to H_\kappa^1}.
\end{align*}
Taking the supremum over unit vectors $h\in L^2$, we get
\begin{align*}
|\partial_tA_L(t) | &\lesssim \kappa^2 A_L(t) + \kappa^4\sqrt{A_L(t)}\|\langle x\rangle^{-\frac\beta 2}R_L(k)\|_{H_\kappa^{-1}\to H_\kappa^1}\\
&\lesssim  \kappa^2 A_L(t) + \kappa^6\|\langle x\rangle^{-\frac\beta 2}R_L(k)\|_{H_\kappa^{-1}\to H_\kappa^1}^2.
\end{align*}
Thus by Gronwall's inequality, 
\[
\sup_{|t|\leq T} A_L(t) \lesssim e^{C\kappa^2 T}\biggl[A_L(0) + \kappa^6 \int_{-T}^T \|\langle x\rangle^{-\frac\beta 2}R_L(k)\|_{H_\kappa^{-1}\to H_\kappa^1}^2\,ds\biggr]. 
\]
As $q_L(t)$ is white noise distributed for each time, the claim follows from \eqref{weighted-wn}, Lemma~\ref{L3}, and \eqref{E:9-2}. 
\end{proof}

It is our intention to employ Lemma~\ref{L:KCT'} to upgrade the convergence given in Lemma~\ref{L:6.4?}.  Our next lemma provides the requisite equicontinuity (cf. Lemma~\ref{L:KCT}).

\begin{lemma}\label{L:gL dif}
Fix $1\leq p<\infty$.  For any strictly admissible $\varkappa$ with $|\varkappa|\leq 10 \kappa$, 
\begin{align}
\E \bigl\{ \bigl\|\langle x\rangle^{-\beta}\bigl[\tfrac1{g_L(t;\varkappa)} -  \tfrac1{g_L(s;\varkappa)}\bigr]\bigr\|_{H^{-2}}^p\bigr\}\lesssim_{p, \beta} |t-s|^p |\varkappa|^{3p} \quad\text{for $\beta>4$},\label{228}\\
\E \bigl\{ \bigl\|\langle x\rangle^{-\beta}\bigl[g_L(t;\varkappa) -  g_L(s;\varkappa)\bigr]\bigr\|_{H^{-1}}^p\bigr\}\lesssim_{p, \beta} |t-s|^{p/2} \quad\text{for $\beta>6$}, \label{229}
\end{align}
uniformly in $L\in 2^{\mathbb N}$ and $t,s\in\R$.
\end{lemma}

\begin{proof}
To ease notation, throughout the proof we omit the subscript $L$.  All our estimates will be independent of $L$.

We will first prove \eqref{228} under the additional restriction that $|k-\vk|> 10^{-6} \kappa$.  The significance of the absolute constant here is that it guarantees that even after doubling the radius of the omitted ball, it is still contained inside the strictly admissible sector.  This restriction on $\vk$ will be removed later.

Using \eqref{dt1/g}, we write
\begin{align}
\tfrac1{g(t;\varkappa)} -  \tfrac1{g(s;\varkappa)}= -4\int_s^t \biggl\{\tfrac{k^2\vk^2}{k^2-\vk^2}&\tfrac{1}{2g(\tau;\vk)} + \tfrac{k^5}{k^2-\vk^2}\tfrac{1}{g(\tau;\vk)}\bigl[g(\tau;k)-\tfrac1{2k}\bigr] \label{6:dt1/gt}\\
&+\tfrac{\bar k^2\vk^2}{\bar k^2-\vk^2}\tfrac{1}{2g(\tau;\vk)} + \tfrac{\bar k^5}{\bar k^2-\vk^2}\tfrac{1}{g(\tau;\vk)}\bigl[g(\tau;\bar k)-\tfrac1{2\bar k}\bigr]\biggr\}'\, d\tau. \notag
\end{align}
Thus, by H\"older, \eqref{L53-5.5}, \eqref{L56-3.5}, and \eqref{1052} we get
\begin{align*}
\text{LHS}\eqref{228} &\lesssim_p |t-s|^{p-1} \int_s^t \E\bigl\{ |\vk|^{2p}\bigl\|\langle x\rangle^{-\beta}\tfrac{1}{g(\tau;\vk)}\bigr\|_{H^{-1}}^p\\
&\qquad\qquad\qquad\qquad\quad  + \kappa^{3p}\bigl\|\langle x\rangle^{-\beta}\tfrac{1}{g(\tau;\vk)}\bigl[g(\tau;k)-\tfrac1{2 k}\bigr]\bigr\|_{H^{-1}}^p\bigr\}\,d\tau\\
&\lesssim_{p, \beta} |t-s|^p |\varkappa|^{3p},
\end{align*}
whenever $\beta>4$, which settles \eqref{228}.

To remove the restriction on $\vk$, we recall that $1/g$ is an analytic function of $\vk$.  Thus its values in the omitted ball can be reconstructed from the values on the boundary of the doubled ball and the Cauchy (or Poisson) integral formula.  That \eqref{228} continues to hold (albeit with a slightly larger constant) then follows just by using the triangle inequality. 

Interpolating between \eqref{228} and \eqref{L56-3.5}, we get
\begin{align*}
\E \bigl\{ \bigl\|\langle x\rangle^{-\beta}\bigl[\tfrac1{g(t;\varkappa)} -  \tfrac1{g(s;\varkappa)}\bigr]\bigr\|_{H^{-1}}^p\bigr\}\lesssim_{p, \beta} |t-s|^{p/2} |\varkappa|^{2p} 
\end{align*}
whenever $\beta>4$, uniformly in $t,s\in\R$.  The claim \eqref{229} now follows from this and \eqref{L53-3.5}, by observing
$$
g(t;\varkappa) -  g(s;\varkappa) = -g(t;\vk)g(s;\vk)\bigl[\tfrac1{g(t;\varkappa)} -  \tfrac1{g(s;\varkappa)}\bigr].
$$
This completes the proof of the lemma.
\end{proof}

We are now ready to formulate our result concerning the \emph{existence} of good solutions to the $\H_k$ flow.  The question of uniqueness will be discussed later.

\begin{proposition}\label{P:666} For almost every initial data $q^0$, there is a global good solution $q(t)$ of the $\H_k$ flow on the line that satisfies all of the following:\\
{\rm(a)} For any $1\leq p<\infty$, $T>0$, and $\beta>1$,
\begin{align}\label{666a}
\lim_{L\to \infty}\E\bigl\{\| \langle x\rangle^{-\beta} [q_L(t)-q(t)] \|_{C_t([-T,T]; H^{-1})}^p\bigr\} = 0.
\end{align}
{\rm(b)} For every $t\in\R$, $q(t)$ is white noise distributed.\\
{\rm(c)} Given a strictly admissible $\vk$ with $|\vk|\leq 10\kappa$ and writing $g(t,x;\vk)=g(x;q(t),\vk)$, we have
\begin{align}\label{666c}
\E\bigl\{\bigl\| \langle x\rangle^{-\beta}g(\vk) \bigr\|_{C_t^{\gamma}([-T,T]; H^{\sigma})}^p\bigr\}+\E\bigl\{\bigl\| \langle x\rangle^{-\beta}\tfrac1{g(\vk)} \bigr\|_{C_t^{\gamma}([-T,T]; H^{\sigma})}^p\bigr\}\lesssim_{|\vk|,p, \beta,\sigma,T} 1,
\end{align}
uniformly in $k$, for any $T>0$, $1\leq p<\infty$, $1\leq \sigma<\frac32$, $\beta>6$, and some $\gamma=\gamma(p, \sigma)>0$.  Moreover, if in addition $\beta>18$, then
\begin{align}\label{666b}
\bigl\| \langle x\rangle^{-\beta} [g_L(\vk)-g(\vk)] \bigr\|_{C_t([-T,T]; H^{1})}
	+ \bigl\| \langle x\rangle^{-\beta} \bigl[\tfrac{1}{g_L(\vk)}-\tfrac{1}{g(\vk)}\bigr] \bigr\|_{C_t([-T,T]; H^{1})}
	\to 0
\end{align}
in $L^p(d\PP)$ as $L\to \infty$.\\
{\rm(d)} $g(t)$ and $q(t)$ are related by \eqref{E:q from g} and $1/g(t)$ solves \eqref{dt1/g}.
\end{proposition}

\begin{proof}
We first define $q(t)$ as the almost sure limit guaranteed by Lemma~\ref{T:L0-to-infinity}.  Consequently,
\begin{equation}\label{bad weight}
\| e^{-\langle x\rangle^{\alpha}} [q_L(t)-q(t)] \|_{C_t([-T,T];H^{-1})}  \to 0 \qtq{as} L\to \infty
\end{equation}
almost surely.  Combining this with Lemmas~\ref{L:66} and~\ref{L:same conv}, we immediately deduce that \eqref{bad weight} also holds in $L^p(d \PP)$ sense, that is,
\begin{align}\label{1205}
\lim_{L\to \infty}\E\bigl\{\| e^{-\langle x\rangle^{\alpha}} [q_L(t)-q(t)] \|_{C_t([-T,T];H^{-1})}^p\bigr\} = 0  \qtq{for any} 1\leq p<\infty.
\end{align}
This can then be upgraded to \eqref{666a} using Lemma~\ref{L:66} together with the fact that
$$
\| \langle x\rangle^{-\beta} f \|_{H^{-1}}  \lesssim e^{2r^\alpha} \| e^{-\langle x\rangle^{\alpha}} f \|_{H^{-1}}
	+ r^{\beta'-\beta} \| \langle x\rangle^{-\beta'} f \|_{H^{-1}} 
$$
for any $r\gtrsim 1$ and $1<\beta' < \beta$.

Recall from Theorem~\ref{T:Hktorus} that $q_L(t)$ is $2L$-periodic white noise for each $t\in\R$; thus part (b) follows from \eqref{666a} via the characterization \eqref{Minlos defn}.

We turn now to part (c).  Combining \eqref{L53-3.5} (with $1\leq \sigma<s<\frac32$) and~\eqref{229} we deduce that
\begin{align}
\E \bigl\{\bigl\| \langle x\rangle^{-\beta} [g_L(t;\vk)-g_L(s;\vk)] \bigr\|_{H^{\sigma}}^p\bigr\} \lesssim_{|\vk|,p,\beta, \sigma}   |t-s|^{\frac{p(s-\sigma)}{2(1+s)}}.
\end{align}
Using again \eqref{L53-3.5} and invoking Lemma~\ref{L:KCT}, we obtain
\begin{align}\label{g328}
\sup_{L\in 2^{\mathbb N}}  \E \bigl\{\bigl\| \langle x\rangle^{-\beta} g_L(\vk)\bigl\|_{C^\gamma([-T,T];H^{\sigma})}^p \bigr\}\lesssim_{|\vk|,p,\beta, \sigma,T} 1
\end{align}
for some $\gamma(p,\sigma)>0$.

Taking $\sigma=1$ in \eqref{g328} together with Lemma~\ref{L:6.4?} provides the prerequisites to apply Lemma~\ref{L:KCT'} and so deduce that there is some process $G(t)$ so that
\begin{align}
\lim_{L\to \infty} \E\bigl\{\| e^{-\langle x\rangle^{\alpha}}  [g_L(t;\vk)-G(t)] \|_{C_t([-T,T];H^1)}^p\bigr\}=0 \qtq{for any} 1\leq p<\infty.
\end{align}
We will verify that $G(t)=g(x;q(t),\vk)$ in a moment; first, we observe that the exponential weight may be replaced by the polynomial weight appearing in \eqref{666b} by the same argument we used to do this for $q(t)$ in part (a) above.

In view of the preceding, to verify that $G(t)=g(x;q(t),\vk)$, it suffices to show
\begin{align}\label{g353}
\lim_{L\to\infty}
\E\bigl\{ \| e^{-4\langle x\rangle^\alpha} [ g_L(t;\vk) - g(t;\vk) ] \|_{L^2_t([-T,T]; H^1)}^2 \bigr\} = 0.
\end{align}
To this end, we note that the resolvent identity gives
\begin{align*}
&\|e^{-4\langle x\rangle^{\alpha}}[g_{L}(t;\vk)-g(t;\vk)] \|_{H^1}^2 \\
 &\lesssim \|e^{-2\langle x\rangle^\alpha} R(\vk) e^{\langle x\rangle^\alpha}\|_{H_{|\vk|}^{-1}\to H_{|\vk|}^1}^2
 	\|e^{-2\langle x\rangle^{\alpha}}[q_L-q]\|_{H^{-1}_{|\vk|}}^2
	\| e^{\langle x\rangle^\alpha}R_{L}(\vk)  e^{-2\langle x\rangle^\alpha}\|_{H_{|\vk|}^{-1}\to H_{|\vk|}^1}^2 .
\end{align*}
This then yields \eqref{g353} by taking expectation, using \eqref{1205} and \eqref{E:9-3}, and then integrating in time.  

As noted above, the identification $G(t)=g(x;q(t),\vk)$ completes the proof of the first claim in \eqref{666b}.  Combining this with \eqref{g328} also yields the first claim in \eqref{666c}.

We turn now to the claims regarding $1/g$ in \eqref{666c} and \eqref{666b}.  The former claim follows by the same argument used to prove \eqref{g328}, by employing \eqref{L56-1.5} and~\eqref{228} in place of \eqref{L53-3.5} and~\eqref{229}.  The latter claim follows from \eqref{666c}, the identity
$$
\tfrac{1}{g} - \tfrac{1}{g_L}=\tfrac{1}{g_L} \tfrac{1}{g} [g_L -  g] 
$$ 
and the convergence $g_L\to g$ shown above.

We next show that $q$ is a good solution (cf. Definition~\ref{D:good}).  Property \eqref{good-solution0} follows already from part (a), while properties \eqref{good-solution} and \eqref{good-solution'} follow from \eqref{E:9-3} and \eqref{E:9-4}, respectively, because we have shown that $q(t)$ is white noise distributed at every time.  That $q(t)$ obeys \eqref{weak-soln-Hk} follows from the fact that $(q_L,g_L)$ obey this identity at finite $L$ and the convergence shown in parts (a) and (c).

Consider now part (d).  By Lemma~\ref{L:diffeo} and Proposition~\ref{P:H kappa}, \eqref{E:q from g} and \eqref{dt1/g} both hold at finite $L$.  These results then carry over to infinite volume by virtue of parts (a) and (c).  This completes the proof of Proposition~\ref{P:666}.
\end{proof}

To complete the proof of Theorem~\ref{T:invariance-Hk-line}, it remains to establish uniqueness within the class of good solutions. 

\begin{proposition}[Uniqueness of good solutions]\label{P:uniq} Suppose $q,\tilde q$ are global good solutions to the $\H_k$ flow \eqref{H_k flow q} on $\R$.  If $q(0)=\tilde q(0)$, then $q(t)\equiv \tilde q(t)$ for all $t\in \R$. 
\end{proposition} 

\begin{proof} Evidently, it suffices to work on an interval $[-T,T]$ for an arbitrary $T>0$.  We proceed similarly to the proof of Lemmas~\ref{T:L0-to-infinity} and \ref{L:66}.  For $t\in[-T,T]$, we define
\[
M(t):=\|e^{-2\langle x\rangle^\alpha -2\langle x\rangle^{\tilde\alpha} }[q(t)-\tilde q(t)]\|_{H_\kappa^{-1}}^2,
\]
where $\alpha$ and $\tilde\alpha$ are the parameters attendant to the two solutions $q$ and $\tilde q$ appearing in Definition~\ref{D:good}.  By symmetry, we may assume $\alpha\leq\tilde\alpha$.

By assumption, $M(0)=0$.  Computing the time derivative leads to the estimate
\begin{align*}
\tfrac{d\,}{dt}M(t) & \lesssim \kappa^2 M(t) + \kappa^5 [M(t)]^{\frac12}\| e^{-2\langle x\rangle^\alpha -2\langle x\rangle^{\tilde\alpha} } [g(t)-\tilde g(t)]\|_{L^2},
\end{align*}
where $g,\tilde g$ denote the diagonal Green's functions for $q$ and $\tilde q$, respectively.  As before, we estimate the $L^2$-norm by duality.  For $h\in L^2$ and $L\geq 1$,\begin{align}
|\langle e^{-2\langle x\rangle^\alpha -2\langle x\rangle^{\tilde\alpha} } [g-\tilde g],h\rangle|
	&\leq |\tr\{e^{-2\langle x\rangle^{\tilde\alpha}}\tilde R\varphi_L^2(q-\tilde q)Re^{-2\langle x\rangle^\alpha} h\}| \label{Unique1}\\
& \quad + |\tr\{e^{-2\langle x\rangle^{\tilde\alpha}}\tilde R(1-\varphi_L^2)(q-\tilde q)Re^{-2\langle x\rangle^\alpha}h\}|, \label{Unique2}
\end{align}
where $\varphi_L$ denotes a smooth cutoff as in \eqref{phiL} and $R,\tilde R$ are the resolvents corresponding to $q$ and $\tilde q$, respectively. 

Using Lemma~\ref{I2H1}, we first estimate
\begin{align*}
\eqref{Unique1} & \lesssim \|e^{-2\langle x\rangle^{\tilde\alpha}} \tilde R\varphi_L\,e^{2\langle x\rangle^{\tilde\alpha}}\|_{H_\kappa^{-1}\to H_\kappa^1}
	\| \sqrt{R_0(\kappa)}e^{-2\langle x\rangle^\alpha -2\langle x\rangle^{\tilde\alpha}}(q-\tilde q)\sqrt{R_0(\kappa)}\|_{\mathfrak{I}_2} \\
& \quad \times \|e^{2\langle x\rangle^\alpha}\varphi_L Re^{-2\langle x\rangle^\alpha}\|_{H_\kappa^{-1}\to H_\kappa^1} \|\sqrt{R_0(\kappa)}h\sqrt{R_0(\kappa)}\|_{\mathfrak{I}_2} \\
& \lesssim  \kappa^{-1}  \|h\|_{L^2}\|e^{-2\langle x\rangle^{\tilde\alpha}} \tilde R\varphi_L\,e^{2\langle x\rangle^{\tilde \alpha}}\|_{H_\kappa^{-1}\to H_\kappa^1}\|e^{2\langle x\rangle^\alpha}\varphi_L Re^{-2\langle x\rangle^\alpha}\|_{H_\kappa^{-1}\to H_\kappa^1}[M(t)]^{\frac12}.
\end{align*}
Similarly, using \eqref{good-solution0}, \eqref{multiplier1}, and Lemma~\ref{I2H1}, we estimate
\begin{align*}
\eqref{Unique2}
& \lesssim \| e^{-2\langle x\rangle^{\tilde\alpha}}\tilde R e^{\langle x\rangle^{\tilde\alpha}}\|_{H_\kappa^{-1}\to H_\kappa^1}
	\|\sqrt{R_0(\kappa)}e^{-\langle x\rangle^{\tilde\alpha}}(q-\tilde q)\sqrt{R_0(\kappa)}\|_{\mathfrak{I}_2} \\
& \quad \times \|e^{-\langle x\rangle^\alpha}(1-\varphi_L^2)\|_{H_\kappa^{-1}\to H_\kappa^{-1}}\| e^{\langle x\rangle^\alpha}Re^{-2\langle x\rangle^\alpha}\|_{H_\kappa^{-1}\to H_\kappa^1} \|\sqrt{R_0(\kappa)}h\sqrt{R_0(\kappa)}\|_{\mathfrak{I}_2} \\
& \lesssim C_T\kappa^{-1}\|h\|_{L^2}e^{-\frac12 \langle L\rangle^\alpha} \| e^{-2\langle x\rangle^\alpha}\tilde R e^{\langle x\rangle^\alpha}\|_{H_\kappa^{-1}\to H_\kappa^1}\| e^{\langle x\rangle^\alpha}Re^{-2\langle x\rangle^\alpha}\|_{H_\kappa^{-1}\to H_\kappa^1} \end{align*}
for some $C_T>0$. (We use here that if \eqref{good-solution0} holds with $\alpha$, it also holds with $\tilde\alpha\geq \alpha$.)

Continuing from above, we deduce
\begin{align*}
\tfrac{d\,}{dt}M(t) \lesssim M(t)&\bigl\{\kappa^2+\kappa^4\|e^{-2\langle x\rangle^{\tilde\alpha}} \tilde R\varphi_L\,e^{2\langle x\rangle^{\tilde\alpha}}\|_{H_\kappa^{-1}\to H_\kappa^1}\|e^{2\langle x\rangle^\alpha}\varphi_L Re^{-2\langle x\rangle^\alpha}\|_{H_\kappa^{-1}\to H_\kappa^1} \\
& +\kappa^8C_T^2 \| e^{-2\langle x\rangle^{\tilde\alpha}}\tilde R e^{\langle x\rangle^{\tilde\alpha}}\|_{H_\kappa^{-1}\to H_\kappa^1}^2\| e^{\langle x\rangle^\alpha}Re^{-2\langle x\rangle^\alpha}\|_{H_\kappa^{-1}\to H_\kappa^1} ^2\bigr\} + e^{-\langle L\rangle^\alpha}.
\end{align*}
Applying Gronwall's inequality and the bounds \eqref{good-solution} and \eqref{good-solution'} for good solutions, we deduce
\[
\sup_{t\in[-T,T]} M(t) \lesssim e^{-\langle L\rangle^\alpha+C(T,\kappa)[\log L]^C}.
\]
Sending $L\to \infty$, we deduce $M(t)\equiv 0$, which completes the proof.\end{proof}

This uniqueness statement allows us to quickly dispel any concerns that our notion of the $\mathcal H_k$ flow is not truly a flow (i.e., a one-parameter group).  In proving this, it will be convenient to adopt the abbreviation
\begin{align}\label{W defn}
\| f\|_W  := \bigl\|\langle x\rangle^{-\beta} f \bigr\|_{H^{-1}(\R)},
\end{align}
where $\beta>1$ is fixed.
 
\begin{corollary}[Group property for $\mathcal H_k$]\label{C:semigroup}
There is a one parameter group of transformations $\Phi:\R\times W\to W$ preserving white noise measure so that
$$
\PP\bigl(\{ q(t) = \Phi(t,q^0) \text{ for all $t\in\R$}\}\bigr) = 1,
$$
where $q(t)$ denotes the solution to the $\mathcal H_k$ flow with initial data $q^0$ and $q^0$ is chosen at random following a white noise law.
\end{corollary}

\begin{proof}
Let $\mathcal O$ denote the subset of $W$ comprised of initial data for which the $\mathcal H_k$ flow admits a good global solution.  In view of Proposition~\ref{P:666}, $\PP(q^0\in\mathcal O)=1$.  On this set we define $\Phi$ in the obvious way; on the complement of $\mathcal O$ we define $\Phi$ to be the identity.

The definition of good solution guarantees that if $q(t)$ is a good solution, so is $q(t+T)$ for any $T\in \R$.  Thus the set $\mathcal O$ is invariant under $\Phi$.  As good solutions are unique, we may also deduce that $\Phi(t)\circ\Phi(s)=\Phi(t+s)$ for all $t,s\in \R$.
\end{proof}

\section{Invariance of white noise for KdV on the line}\label{S:KdV}

Our goal in this section is to complete the proof of the main result of the paper, namely, the existence of KdV dynamics for white noise initial data on the whole line and the invariance of white noise measure under these dynamics; see Theorem~\ref{T:ktoinfinity} below. Our procedure for doing this is to take a limit of the $\H_k$ flows, as $k\to\infty$.  Recall that for $H^{-1}(\R)$ initial data, this limit was shown to coincide with the KdV flow in \cite{KV}.  For concreteness, we shall send $k\to\infty$ along the sequence
\[
k_n = 2^n e^{i\pi/8}  \qtq{where}  n\in\mathbb{N}.
\]
Any other sequence with $\eps<\arg(k_n)<\frac\pi4 -\eps$ would serve equally well and would lead to the same limit.  The sparseness imposed on this sequence is purely so that we may prove $L^p(d\PP)$ convergence by summing increments; this restriction is readily removed via Urysohn's subsequence principle.

\begin{theorem}[Invariance of white noise for KdV on the line]\label{T:ktoinfinity} Given initial data $q^0$ following the white noise distribution on the line, let $q_n$ denote the corresponding global good solution to the $\H_{k_n}$ flow \eqref{H_k flow q} on the line, whose existence is guaranteed by Theorem~\ref{T:invariance-Hk-line}.  For any $T>0$, $q_n$ converges to a limit $q$ in the following sense: for any $1\leq p<\infty$ and $\beta>24$,
\begin{equation}\label{E:qnq}
\lim_{n\to \infty}\E\bigl\{ \bigl\|\langle x\rangle^{-\beta} (q_n-q) \bigr\|_{C_t( [-T,T]; H^{-1})}^p\bigr\} =0.
\end{equation}
Moreover, $q(t)$ is white noise distributed for every $t\in\R$. 
\end{theorem}

While the early part of this section is devoted to the proof of this theorem, we also include a further discussion on the nature of our solutions and two results, namely, Corollaries~\ref{C:intrinsic} and~\ref{C:KdV group}, that illustrate important virtues of these solutions.

The path to proving Theorem~\ref{T:ktoinfinity} is as follows: we first control the both the diagonal Green's function and its reciprocal and then use Lemma~\ref{L:diffeo} to deduce control on the solutions $q_n$ themselves.  As a first step, we treat $1/g_n$. 

\begin{proposition}\label{P:12main} Fix $k\in \{2^n e^{i\pi/8}:\,  n\in\mathbb{N}\}$ and let $\vk$ be strictly admissible so that $|\vk|\leq \frac12|k|$.  Then for $\beta>18$,
\begin{equation} \label{L55-sum}
 \sup_{t\in[-T,T]}\E\biggl\{ \Bigl\|\langle x\rangle^{-\beta}\Bigl[\tfrac{1}{2g(q_{2k}(t);\vk)} - \tfrac{1}{2g(q_k(t);\vk)}\Bigr]\Bigr\|_{H^{-2}}^2 \biggr\} \lesssim_{|\vk|,T}|k|^{-1}. 
\end{equation}
\end{proposition}

\begin{proof}  In view of Proposition~\ref{P:666}, it suffices to prove the result in finite volume with a bound independent of $L\in 2^\mathbb{N}$.  Henceforth, we will regard $L$ as fixed and suppress the dependence of $q$ and $g$ on $L$.

As the $\H_k$ flows commute on $H^{-1}(\R/2L\Z)$, we may write
\[
q_{2k}(t) = e^{tJ\nabla[\mathcal{H}_{2k}-\mathcal{H}_k]}q_{k}(t).
\] 
Using invariance of white noise under the $\H_k$ flow, we may therefore replace
\[
\tfrac{1}{2g(q_{2k}(t);\vk)}-\tfrac{1}{2g(q_{k}(t);\vk)}\qtq{by}\tfrac{1}{2g(t;\vk)}-\tfrac{1}{2g(0;\vk)}
\]
in \eqref{L55-sum}, where $g(t;\vk):=g(q(t),\vk)$ and
\[
q(t):=e^{tJ\nabla[\H_{2k}-\H_k]}q,\quad \text{with $q$ white noise distributed.}
\]
Applying the fundamental theorem of calculus and Minkowski's inequality, we get
\begin{align*}
\sup_{|t|\leq T}\E\Bigl\{ \bigl\|\langle x\rangle^{-\beta}\bigl[ \tfrac{1}{2g(t;\vk)}-\tfrac{1}{2g(0;\vk)}\bigr]\bigr\|_{H^{-2}}^2 \Bigr\} & \lesssim \E\biggl\{ \biggl[\int_{-T}^T \bigl\|\langle x\rangle^{-\beta}\partial_t \tfrac{1}{g(t;\vk)}\bigr\|_{H^{-2}}\,dt\biggr]^2\biggr\}\\
& \lesssim T \sup_{t\in[-T,T]}\E\bigl\{ \bigl\|\langle x\rangle^{-\beta}\partial_t \tfrac{1}{g(t;\vk)}\bigr\|_{H^{-2}}^2\bigr\} .
\end{align*}
As white noise is invariant under the difference flow, the expectation here is actually independent of $t\in[-T,T]$.  With this in mind, we suppress the $t$ dependence in what follows.

Using \eqref{dt1/g}, we now write
\begin{align}
\partial_t \tfrac{1}{2g(\vk)}& = \Bigl\{\tfrac1{g(\vk)}\Bigl( -2(2k)^3[g(2k)-\tfrac{1}{4k}]+2k^3[g(k)-\tfrac{1}{2k}]\Bigr) \label{57}  \\
& \quad +\tfrac{1}{g(\vk)}\Bigl( -2(2\bar  k)^3[g(2\bar k)-\tfrac{1}{4\bar k}]+2\bar k^3[g(\bar k)-\tfrac{1}{2\bar k}]\Bigr) \label{58} \\
& \quad - 2\tfrac{(2k)^3\vk^2}{(2k)^2-\vk^2}  \tfrac1{g(\vk)} \bigl[g(2k)-\tfrac{1}{4k}\bigr] + 2\tfrac{k^3\vk^2}{k^2-\vk^2}\tfrac{1}{g(\vk)}\bigl[g(k)-\tfrac{1}{2k}\bigr] \label{59}\\
& \quad - 2\tfrac{(2\bar k)^3\vk^2}{(2\bar k)^2-\vk^2} \tfrac1{g(\vk)}\bigl[ g(2\bar k)-\tfrac{1}{4\bar k}\bigr] + 2\tfrac{\bar k^3\vk^2}{\bar k^2-\vk^2}\tfrac1{g(\vk)}\bigl[g(\bar k)-\tfrac{1}{2\bar k}\bigr]\label{510} \\
&\quad  + \left[\tfrac{3k^2{\vk^4}}{(k^2-\vk^2)(4k^2-\vk^2)}+\tfrac{3\bar k^2 {\vk^4}}{(\bar k^2-\vk^2)(4\bar k^2-\vk^2)}\right]\tfrac{1}{g(\vk)}\Bigr\}'. \label{511} 
\end{align}
It remains to estimate the contribution of each of these terms. In what follows, implicit constants are allowed to depend on $\vk$. 

We start with \eqref{59} and \eqref{510}.  By symmetry, it suffices to consider the second term in \eqref{59}. Using \eqref{L53-1.5} and \eqref{L56-3.5}, we estimate this contribution by 
\begin{align*}
&\lesssim_{|\vk|} |k|\E\bigl\{\bigl\|\langle x\rangle^{-\beta}\tfrac1{g(\vk)}\bigl[g(k)-\tfrac{1}{2k}\bigr]\bigr\|_{L^2}^2\bigr\}\\
& \lesssim_{|\vk|} |k| \E\bigl\{\|\langle x\rangle^{-\beta/2}\tfrac{1}{g(\vk)}\|_{L^4}^4\bigr\}^{\frac12} 
	\E\bigl\{\|\langle x\rangle^{-\beta/2}[g(k)-\tfrac{1}{2k}]\|_{L^4}^4\bigr\}^{\frac12}
\lesssim_{|\vk|} |k|^{-4}.
\end{align*} 

Using Cauchy--Schwarz and \eqref{L56-3.5}, we estimate the contribution of \eqref{511} by
\begin{align*}
\lesssim_{|\vk|}|k|^{-2}\E\bigl\{ \|\langle x\rangle^{-\beta}\tfrac{1}{g(\vk)}\|_{L^2}^2\bigr\}\lesssim_{|\vk|} |k|^{-2}.
\end{align*}

It remains to treat \eqref{57} and \eqref{58}.  By symmetry, it suffices to consider \eqref{57}.  Using \eqref{g-to-quadratic}, we write
\begin{align}
-(2k)^3[g(2k)-\tfrac{1}{4k}] + k^3[g(k)-\tfrac{1}{2k}]& = [(2k)^2R_0(4k)-k^2R_0(2k)]q \label{512}\\
& \quad - (2k^3)\langle \delta_x, R_0(2k)qR(2k)qR_0(2k)\delta_x\rangle \label{513}\\
& \quad + k^3\langle \delta_x, R_0(k) q R(k) q R_0(k) \delta_x\rangle. \label{514} 
\end{align}

We first consider \eqref{512}.  By an explicit computation, \eqref{L56-1.5}, and H\"older's inequality, we estimate the contribution of \eqref{512} by
\begin{align*}
&\lesssim \E\bigl\{\bigl\|\langle x\rangle^{-\beta}\tfrac{1}{g(\vk)}[(2k)^2R_0(4k)-k^2R_0(2k)]q\bigr\|_{H^{-1}}^2 \bigr\} \\
& \lesssim |k|^4 [\E\{ \| \langle x\rangle^{2-\beta}\tfrac{1}{g(\vk)}\|_{H^1}^4\}]^{\frac12}  [\E\{ \|\langle x\rangle^{-2} \partial^2 R_0(2k)R_0(4k)q\|_{H^{-1}}^4\}]^{\frac12} \\
& \lesssim_{|\vk|} |k|^4  [\E\{ \|\langle x\rangle^{-2} \partial R_0(2k)R_0(4k)q\|_{L^2}^4\}]^{\frac12}.
\end{align*}
As $q$ is white noise distributed, a computation using \eqref{ONB rep} shows
\begin{align*}
\E\bigl\{\| \langle x\rangle^{-2}\partial R_0(2k)R_0(4k)q\|_{L^2}^2\bigr\} &= \| \langle x\rangle^{-2}\partial R_0(2k)R_0(4k)\|_{\HS}^2 \lesssim |k|^{-5},
\end{align*}
and hence by Lemma~\ref{L3},
\[
\E\bigl\{ \|\langle x\rangle^{-2}\partial R_0(2k)R_0(4k)q\|_{L^2}^4\bigr\}^{\frac12} \lesssim |k|^{-5}.
\]
Thus the contribution of \eqref{512} is controlled by $|k|^{-1}$, which is acceptable. 

It remains to treat the contributions of \eqref{513} and \eqref{514}.  By symmetry, it suffices to consider \eqref{514}.  We expand to one order higher in $q$ and estimate it by
\begin{align}\label{514-1}
&\lesssim |k|^6\E\bigl\{\bigl\|\langle x\rangle^{-\beta}\tfrac1{g(\vk)}\langle \delta_x, R_0(k)qR(k)qR_0(k)\delta_x\rangle\bigr\|_{L^2}^2 \bigr\} \notag\\
& \lesssim |k|^6 \E\biggl\{\| \langle x\rangle^{-\beta/2} \tfrac{1}{g(\vk)}\|_{L^4}^2\biggl[  \| \langle x\rangle^{-\beta/2}\langle\delta_x,R_0(k)qR_0(k)qR_0(k)\delta_x\rangle\|_{L^4}^2\notag\\
& \quad\quad\quad\quad\qquad\quad+ \| \langle x\rangle^{-\beta/2}\langle \delta_x,R_0(k)qR_0(k)qR(k)qR_0(k)\delta_x\rangle\|_{L^4}^2\biggr]\biggr\}\notag \\
& \lesssim_{|\vk|}|k|^6 \Bigl[ \E\bigl\{\bigl \|\langle x\rangle^{-\beta/2}\langle \delta_x, R_0(k)qR_0(k)qR_0(k)\delta_x\rangle\bigr\|_{L^4}^4\bigr\}\notag\\
& \quad\quad\quad \qquad+ \E\bigl\{\bigl\| \langle x\rangle^{-\beta/2}\langle \delta_x,R_0(k)qR_0(k)qR(k)qR_0(k)\delta_x\rangle\bigr\|_{L^4}^4\bigr\}\Bigr]^{\frac12} .
\end{align}

We estimate
\begin{align*}
| \langle \delta_0, R_0(k)qR_0(k)qR_0(k)\delta_0\rangle| & \lesssim \|qR_0(k)\delta_0\|_{H_{\kappa}^{-1}}^2 \|R_0(k)\|_{H_{\kappa}^{-1}\to H_{\kappa}^1} 
\lesssim \|qR_0(k)\delta_0\|_{H_{\kappa}^{-1}}^2
\end{align*}
where, as before, $\kappa=|k|$.  As $q$ is white noise distributed, we may use \eqref{weighted-wn} followed by Lemma~\ref{L3} to obtain
\[
\E\bigl\{ \|qR_0(k)\delta_0\|_{H_\kappa^{-1}}^p\} \lesssim_p |k|^{-2p}
\] 
for any $1\leq p<\infty$.  Thus, using translation invariance of white noise, we deduce
\begin{align}
\E\bigl\{\| \langle x\rangle^{-\beta/2}\langle\delta_x,R_0(k)qR_0(k)qR_0(k)\delta_x\rangle\|_{L^4}^4\bigr\} & \lesssim \E\{ \|qR_0(k)\delta_0\|_{H_{\kappa}^{-1}}^{8}\}\| \langle x\rangle^{-\beta/2}\|_{L^4}^4   \notag\\
& \lesssim |k|^{-16}. \label{E:MLA}
\end{align}

To handle the contribution of the final term in \eqref{514-1}, we use Lemma~\ref{P:11}, \eqref{E:C:q op}, and \eqref{E:9-2} to estimate
\begin{align*}
\E\bigl\{ |&\langle \delta_0,R_0(k)qR_0(k)qR(k)qR_0(k)\delta_0\rangle|^4\bigr\} \\
& \lesssim \E\bigl\{ \|\langle x\rangle^4\delta_0\|_{H_{\kappa}^{-1}}^{8}  \| \langle x\rangle^{-4} R_0(k) \langle x\rangle^4\|_{H_{\kappa}^{-1}\to H_{\kappa}^1}^8
	\|\langle x\rangle^{-2} R_0(k) \langle x\rangle^{2}\|_{H_\kappa^{-1}\to H_\kappa^1}^4 \\
&\hspace*{13em} \times  \|\langle x\rangle^{-2}q\|_{H_{\kappa}^{1}\to H_{\kappa}^{-1}}^{12} \|R(k)\langle x\rangle^{-2}\|_{H_{\kappa}^{-1}\to H_{\kappa}^1}^4 \bigr\} \\
& \lesssim |k|^{-4}|k|^{-12} \lesssim |k|^{-16}. 
\end{align*}
By translation invariance of white noise, we get
\begin{align}
\E\{\|\langle x\rangle^{-\beta/2}\langle \delta_x,R_0(k)qR_0(k)qR(k)qR_0(k)\delta_x\rangle\|_{L^4}^4\} &\lesssim |k|^{-16} \| \langle x\rangle^{-\beta/2}\|_{L^4}^4 \notag\\
&\lesssim  |k|^{-16}. \label{E:MLB}
\end{align}

Thus, continuing from \eqref{514-1}, we deduce 
\[
|k|^6\E\bigl\{\bigl\|\langle x\rangle^{-\beta}\tfrac1{g(\vk)}\langle \delta_x, R_0(k)qR(k)qR_0(k)\delta_x\rangle\bigr\|_{L^2}^2 \bigr\} \lesssim_{|\vk|} |k|^{-2},
\]
which is acceptable.  This completes the proof of the proposition.\end{proof}

Proposition~\ref{P:12main} provides the key input to complete the

\begin{proof}[Proof of Theorem~\ref{T:ktoinfinity}]
Throughout this proof, we write $g_n(t; \vk)=g(q_n(t),\vk)$.  Combining \eqref{L55-sum} and \eqref{666c}, we see that
\begin{align*}
\lim_{n\to\infty} \sup_{|t|\leq T} \E\biggl\{ \Bigl( \sum_{m=n}^\infty \bigl\|\langle x\rangle^{-\beta}\bigl[\tfrac{1}{g_n(t; \vk)} - \tfrac{1}{g_m(t;\vk)}\bigr]\bigr\|_{H^1} \Bigr)^2\biggr\}=0.
\end{align*}
Thus, invoking \eqref{666c} again, we may apply Lemma~\ref{L:KCT'} to deduce that there is a process $\frac1g$ so that for $\beta>18$,
\begin{align}\label{end1}
\lim_{n\to \infty}\E\bigl\{\bigl\|\langle x\rangle^{-\beta} \bigl[\tfrac1{g_n(\vk)}-\tfrac1{g(\vk)}\bigr]\bigr\|_{C_t([-T,T];H^1)}^p\bigr\}=0.
\end{align}

In view of \eqref{666c} and the identity 
$$
g_n-g_m = g_ng_m\bigl[\tfrac1{g_m}- \tfrac{1}{g_n} \bigr],
$$
it now follows that $g_n(\vk)\to g(\vk)$ in the same topology:
\begin{align}\label{end2}
\lim_{n\to \infty}\E\bigl\{\bigl\|\langle x\rangle^{-\beta} \bigl[g_n(\vk) - g(\vk)\bigr]\bigr\|_{C_t([-T,T];H^1)}^p\bigr\}=0,
\end{align}
provided $\beta>30$.

A posteriori, we may relax the condition on $\beta$ in \eqref{end1} and \eqref{end2} to merely $\beta>6$, by interpolating with \eqref{666c}.

We now turn our attention to $q_n$.  Recall from Proposition~\ref{P:666}(d) that
$$
q_n = \bigl[\tfrac{g_n'(\vk)}{2g_n(\vk)}\bigr]' + \bigl[\tfrac{g_n'(\vk)}{2g_n(\vk)}\bigr]^2 +  \bigl[\tfrac{1}{4g_n^2(\vk)} - \vk^2 \bigr].
$$
As both $g_n$ and $\frac1{g_n}$ are convergent in the senses \eqref{end1} and \eqref{end2}, respectively, for $\beta >6$, it follows that $q_n$ are convergent in the sense \eqref{E:qnq} for $\beta>24$.

It remains to show that $q(t)$ is white noise distributed.  This follows immediately from the characterization \eqref{Minlos defn}, \eqref{E:qnq}, dominated convergence, and the fact that each $q_n(t)$ is white noise distributed (see Proposition~\ref{P:666}).
\end{proof}

In \cite{KV} it is shown that $C_t(\R; H_x^{-1}(\R))$ solutions to KdV constructed as the (unique!) limit of Schwartz solutions are in fact distributional solutions. Nevertheless, the uniqueness question for distributional solutions remains open at this time.  In the torus setting, it is currently unknown if the Kappeler--Topalov solutions (constructed in \cite{MR2267286}) are distributional solutions in any sense.  The work \cite{Christ'05} of Christ gives strong evidence that distributional solutions may not be unique at negative regularity.

The time integration inherent in the definition of distributional solutions is essential in making any sense of the KdV equation \eqref{KdV} for such irregular data, due to the impossibility of defining $q(t,x)^2$ pointwise in time.  It is natural to ask if the solutions constructed in Theorem~\ref{T:ktoinfinity} are distributional solutions.  We believe that they are; nevertheless, verifying this seems to be a challenging problem.   Before presenting what additional properties of our solutions we can prove at this time (namely Corollaries~\ref{C:intrinsic} and~\ref{C:KdV group}) we would like to devote a little space to sharing our current perspective on this topic.

Essential to showing that the solutions to KdV constructed in \cite{KV} are distributional solutions is the local smoothing effect, which guarantees that
\begin{equation}\label{LS}
\iint \frac{q(t,x)^2}{\langle t\rangle^2\langle x\rangle^2}\,dx\,dt < \infty.
\end{equation}
This immediately yields that the nonlinearity can be interpreted as a space-time distribution.

White noise solutions cannot obey \eqref{LS}, even under stronger space-time localization. Indeed, by invariance of white noise measure, the local $L^2_x$ norm is almost surely infinite at every time.  The intuitive picture here is analogous to the failure of local smoothing for periodic initial data (i.e., the torus).

A naive explanation for the divergence of the square is that it originates in low frequencies and will be subdued by the derivative outside the nonlinearity. Indeed, the most obvious way to remove the divergence in expectation is to Wick-order the nonlinearity, which will not affect the dynamics precisely due to this derivative.  However, the underlying divergence of the square of white noise is much more severe than that.  Indeed, choosing trigonometric series in the representation \eqref{ONB rep} of white noise on the circle, one immediately sees that the average value of the zero Fourier mode diverges, while that of other modes remains zero.  However, the fluctuations of all modes diverge (the probability laws lose tightness)! 

Nevertheless, we still believe that our solutions are distributional solutions.  The primary source of this hope comes from a subtly different dispersive effect, which by analogy with local smoothing, we would call the local averaging effect.  A simple computation with the Airy equation reveals its underlying phenomenology:  Let us write $q(t) = e^{-t\partial^3} q^0$ for the solution of the Airy equation with initial data $q^0$, which is white noise distributed.  Then $q(t)$ is a Gaussian process with covariance
$$
C(t,x;s,y):=\E\{ q(t,x) q(s,y) \} = (3|t-s|)^{-1/3} \Ai\Bigl( (3[t-s])^{-1/3} [x-y]\Bigr),
$$
where $\Ai$ denotes the Airy function.  It is now reasonably easy to verify that for any Schwartz function $\psi:\R^2\to\R$,
$$
\iint :q(t,x)^2: \psi(t,x)\,dt\,dx,
$$
where colons represent Wick ordering, yields a well-defined random variable with second moment
$$
2  \iint\iint \psi(t,x) C(t,x;s,y)^2 \psi(s,y)\,dy\,ds\,dx\,dt  <\infty.
$$
See \cite[Theorem~I.3]{Pphi2}.  The intuitive explanation behind this is that the time integration averages out some of the divergence by exploiting the rapid oscillation in time of waves of high wave number.

One word of warning seems appropriate at this point, namely, that the local averaging effect of the $\mathcal H_k$ flows is too weak to make sense of the square of these solutions (with white noise initial data) as space-time distributional solutions in the above sense.  Therein, we already see the first major obstacle to be overcome in treating the question of distributional  solutions by the methods of this paper.
 
Nevertheless, we will demonstrate that the solutions constructed in Theorem~\ref{T:ktoinfinity} are solutions in at least one intrinsic sense.  By `intrinsic' we mean determined by the solution itself, independent of any limiting procedures used to construct it.  The idea stems from Proposition~3.1 in \cite{KV} and as such, let us first articulate it in the setting of $C_t^{ } H^{-1}_x$ solutions on the line and on the circle:

\begin{definition}
We say that $q\in C_t^{ } H^{-1}_x$ is a \emph{green solution} to \eqref{KdV} if for all $\vk>0$ sufficiently large, the associated diagonal Green's function $g(t,x;q,\vk)$ obeys
\begin{align}\label{green solution}
\tfrac 1{2g(t; \vk)}- \tfrac 1{2g(0;\vk)} =  \Bigl(\int_0^t \tfrac {q(s)}{g(s;\vk)} - \tfrac{2\vk^2}{g(s;\vk)} \, ds\Bigr)'  .
\end{align}
Note that this is inherently meaningful as an identity of $C_t^{ } H^{-2}_x$ functions. 
\end{definition}

The arguments of \cite{KV} give immediately that the solutions constructed therein are green solutions.  It is also our belief that green solutions are unique in that setting; however, this is a challenging open problem (cf. \cite{Christ'05,MR0988885}).  Nevertheless, we do know that green solutions are unique and coincide with classical notions of solution already for $q\in C_t^{ }L^2_x$.  More explicitly, sending $\vk\to\infty$ in \eqref{green solution}, one finds that green solutions are distributional solutions, which are then unique by the results of \cite{MR2789490,YZhou}.

As our next result, we observe that the solutions constructed in Theorem~\ref{T:ktoinfinity} are green solutions in the sense natural to this problem; once again, the almost sure uniqueness of such solutions is an interesting open problem.

\begin{corollary}[Intrinsic solutions]\label{C:intrinsic}
Let $q(t)$ denote the ensemble of solutions constructed in Theorem~\ref{T:ktoinfinity}, let $\vk$ be strictly admissible, and let $g(t)=g(t,x;\vk, q(t))$.  Then the couple $(q(t),g(t))$ obeys \eqref{green solution} for all times $t\in\R$ (as an equality of tempered distributions), excepting perhaps a set of initial data of probability zero. 
\end{corollary}

\begin{proof}
Recall \eqref{6:dt1/gt}, which is merely a rewriting of \eqref{dt1/g}.  In view of this identity obeyed by solutions of the $\mathcal H_k$ flow, we may recover \eqref{green solution} for solutions of the KdV equation employing \eqref{E:qnq}, \eqref{end1}, and by verifying that (after excluding a set of initial data of probability zero) 
\begin{equation}\label{green convergence}
\bigl\| \langle x\rangle^{-\beta} \bigl\{ 4k_n^3\bigl[ \tfrac{1}{2k_n} - g_n(t)\bigr]  - q_n(t)\bigr\} \bigr\|_{L^2([-T,T]; H^{-1})} \to 0
\end{equation}
as $n\to \infty$ for every $T>0$ and some $\beta>0$.  Here $q_n$ and $g_n$ denote the solution of the $\H_{k_n}$ flow and its associated diagonal Green's function, as earlier in this section.

By the resolvent identity and direct simplification of the term linear in $q$, 
\begin{align*}
4k_n^3\bigl[ \tfrac{1}{2k_n} - g_n(t)\bigr]  - q_n(t) &= \partial^2 R_0(2k_n) q_n - 4k_n^3\langle\delta_x, R_0(k_n) q_n R_0(k_n) q_n R_0(k_n) \delta_x\rangle \\
&\quad + 4k_n^3\langle\delta_x, R_0(k_n) q_n R_0(k_n) q_n R(k_n) q_n R_0(k_n) \delta_x\rangle.
\end{align*}
The second and third terms in this expansion have already been estimated in a satisfactory manner; see \eqref{E:MLA} and \eqref{E:MLB}, respectively.  We turn our attention to the first term:
\begin{align*}
\E\bigl\{ \|  \langle x\rangle^{-\beta} \partial^2 R_0(2k_n) q_n \bigr\|_{H^{-1}}^2\bigr\}
	&\lesssim  \E\bigl\{ \|  \langle x\rangle^{-\beta} \partial R_0(2k_n) q_n \bigr\|_{L^2}^2\bigr\} \\
&\lesssim  \| \langle x\rangle^{-\beta} \partial R_0(2k_n) \|_{\HS}^2 \lesssim |k_n|^{-1}.
\end{align*}
Putting everything together, we deduce that
\begin{equation*}
\E\Bigl\{ \bigl\| \langle x\rangle^{-\beta} \bigl\{ 4k_n^3\bigl[ \tfrac{1}{2k_n} - g_n(t)\bigr]  - q_n(t)\bigr\} \bigr\|_{H^{-1}(\R)}^2\Bigr\} \lesssim |k_n|^{-1},
\end{equation*}
for $\beta$ sufficiently large.  The conclusion \eqref{green convergence} that we seek now follows by Fubini and the Borel--Cantelli lemma.
\end{proof}

Next, we would like to confirm that the solutions of KdV we have constructed have the group property.  What we show here is slightly weaker than what we proved for the $\mathcal H_k$ flow; nevertheless, it suffices for the questions of immediate interest to us.  In particular, the realization of the Koopman operators as a strongly continuous unitary group is precisely what is needed for the standard proof of the mean ergodic theorem and also suffices to characterize weak/strong mixing; see \cite{RS1}, for example.

For more sophisticated questions one may wish to realize the Koopman operators as coming from an honest one-parameter group in a concrete Polish space.  (By `honest' here, we mean that the group property holds without exceptional null sets.)  This can be done by invoking a classical theorem of Mackey (see \cite[Appendix~B]{Zimmer}).

\begin{corollary}[Group property for KdV]\label{C:KdV group}
The Koopman operators
$$
U(t):L^2(d\mu)\to L^2(d\mu) \qtq{via} [U(t)f](q^0) = f(q(t))
$$
form a strongly continuous one-parameter unitary group.  Here $d\mu$ denotes white-noise measure on $W$ endowed with the Borel $\sigma$ algebra.
\end{corollary}

\begin{proof}
From Theorem~\ref{T:ktoinfinity}, we know that our flow preserves white noise measure and that the trajectories are almost surely continuous in the Hilbert space $W$ defined in \eqref{W defn}, but now with $\beta>24$.  (Note that the standard theorems of measure theory apply on $W$ because it is completely metrizable and separable; see \cite[\S17]{Kechris}).

By the measure-preserving property and the density of bounded continuous functions in $L^2(d\PP)$, strong continuity of the Koopman operators follows from the almost-sure continuity of trajectories and the dominated convergence theorem.  The measure-preserving property also immediately guarantees that the Koopman operators are isometries.  Unitarity will follow once we verify the group property:
$$
U(t) U(s)  = U(t+s)  \quad\text{for any fixed $s,t\in\R$.}
$$
This in turn will follow if we show
\begin{align}\label{groupy}
\E\bigl\{ \bigl| \langle\psi,\Phi(t)\circ\Phi(s) q^0\rangle - \langle\psi,\Phi(t+s) q^0\rangle \bigr| \bigr\} = 0
\end{align}
for fixed $s,t\in \R$ and fixed Schwartz function $\psi$.  Here $\Phi:\R\times W\to W$ denotes the data-to-solution map constructed  (on a set of full probability)  in Theorem~\ref{T:ktoinfinity}.  The fact that $\Phi$ was constructed as the a.e. limit of a sequence of one-parameter groups $\Phi_n$, namely the $\mathcal H_{k_n}$ flows, will be pertinent to completing the present proof.  In truth, we have only explicitly proved $L^p(d\PP;W)$ convergence along the original sequence of parameters $k_n$; thus, we should now to pass to a subsequence $k_n\to\infty$ so as to guarantee almost sure convergence.  (Actually, with more attention to the minute details, one could avoid passing to a subsequence, but this is of no consequence.)

As the map $q^0\mapsto \langle\psi,\Phi(t) q^0\rangle \in L^2(d\PP)$, for each $\eps>0$, Lusin's and Tietze's theorems combine to guarantee the existence of an $F_\eps:W\to\R$ that is continuous, bounded, and satisfies
$$
\E\bigl\{ \bigl|\langle\psi,\Phi(t) q^0\rangle - F_\eps( q^0) \bigr| \bigr\} \leq \eps.
$$
In view of Corollary~\ref{C:semigroup}, we have also have 
$$
\langle\psi,\Phi_n(t)\circ\Phi_n(s) q^0\rangle - \langle\psi,\Phi_n(t+s) q^0\rangle \equiv 0.
$$
Thus, recalling also that $\Phi$ and $\Phi_n$ are measure preserving at any fixed time,
\begin{align*}
\text{LHS\eqref{groupy}} &\leq  \E\bigl\{ \bigl| \langle\psi,\Phi(t)\circ\Phi(s) q^0\rangle - \langle\psi,\Phi(t)\circ\Phi_n(s) q^0\rangle \bigr| \bigr\} \\
&\quad + \E\bigl\{ \bigl| \langle\psi,\Phi(t)\circ\Phi_n(s) q^0\rangle - \langle\psi,\Phi_n(t)\circ\Phi_n(s) q^0\rangle \bigr| \bigr\} \\
&\quad + \E\bigl\{ \bigl| \langle\psi,\Phi_n(t+s) q^0\rangle - \langle\psi,\Phi(t+s) q^0\rangle \bigr| \bigr\}  \\
&\leq 2\eps + \E\bigl\{ \bigl| F_\eps\circ\Phi(s)(q^0) - F_\eps\circ\Phi_n(s) (q^0) \bigr| \bigr\} \\
&\quad + \E\bigl\{ \bigl| \langle\psi,\Phi(t) q^0\rangle - \langle\psi,\Phi_n(t) q^0\rangle \bigr| \bigr\} \\
&\quad + \E\bigl\{ \bigl| \langle\psi,\Phi_n(t+s) q^0\rangle - \langle\psi,\Phi(t+s) q^0\rangle \bigr| \bigr\}  \\
&\leq 2\eps + o(1) \quad\text{as $n\to\infty$.}
\end{align*}
Note that these last three terms converge to zero by virtue of Theorem~\ref{T:ktoinfinity}, the continuity of $F_\eps$, and Lemma~\ref{L:same conv}.  As $\eps>0$ was arbitrary, this completes the proof of \eqref{groupy} and hence that of Corollary~\ref{C:KdV group}.
\end{proof}

Corollary~\ref{C:KdV group} also allows us to rigorously formulate our loftiest ambition with regard to the model discussed in this paper: 

\begin{conjecture}\label{mixing conjecture}
The KdV flow is mixing, namely, as $t\to\infty$, the operators $U(t)$ converge to the projection onto constants in the weak operator topology.
\end{conjecture}

\end{document}